\documentclass[paper, 9pt]{article}
\usepackage{fullpage}
\usepackage{latexsym}
\usepackage{amsfonts, amsthm}
\usepackage[all]{xy}
\usepackage{amsmath}
\usepackage{xcolor}
\usepackage{amssymb}
\usepackage{amscd}
\usepackage{mathrsfs}
\usepackage[english]{babel}
\usepackage{hyperref}
\usepackage[utf8]{inputenc}

\usepackage{fancyhdr}
\usepackage{verbatim} 
\usepackage{color} 

\newtheorem{defi}{Definition}[section]
\newtheorem{prop}[defi]{Proposition} 
\newtheorem{lemma}[defi]{Lemma} 
\newtheorem{teo}[defi]{Theorem}
\newtheorem{cor}[defi]{Corollary}
\theoremstyle{definition} 
\newtheorem{rmk}[defi]{Remark}

\newtheorem*{teo*}{Theorem}

\newcommand{\Z}{\mathbb{Z}}

\newcommand{\Sp}{\mathrm{Spec}}
\newcommand{\g}[1] {\bold{#1}}
\providecommand{\bysame}{\leavevmode\hbox to3em{\hrulefill}\thinspace}
\providecommand{\MR}{\relax\ifhmode\unskip\space\fi MR }

\providecommand{\href}[2]{#2}

\begin{document}
\title{On the homotopy exact sequence for log algebraic fundamental groups}
\author{Valentina Di Proietto\footnote{College of Engineering, Mathematics and Physical Sciences, University of Exeter,
EX4 4RN,
Exeter, United Kingdom.} \, and Atsushi Shiho\footnote{Graduate School of 
Mathematical Sciences, the University of Tokyo, 3-8-1 Komaba, Meguro-ku, Tokyo 153-8914, Japan. Mathematics Subject Classification (2010): 14F35, 14F40.}}
\date{}
\maketitle

\textsc{Abstract -} \begin{small} 
 We construct a log algebraic version of the homotopy sequence for a normal crossing log variety over a log point of characteristic zero and prove some exactness properties of it. Our proofs are purely algebraic.

\end{small}

\date{}
\maketitle
\section{Introduction}
First we recall the classical homotopy exact sequence in the topological setting. Let $p:E\rightarrow B$ be a fibration (a continuous map satisfying the homotopy lifting property) of topological manifolds with path connected base $B$. Let $b_0$ be a point of $B$ and let  $F=p^{-1}(b_0)$. We denote by $i:F\hookrightarrow E$ the canonical inclusion. Also, we choose a point $f_0$ in $F$ and we denote by $e_0$ the point $i(f_0)$. Then we have the exact sequence
\begin{equation}\label{HEStop}
\dots\rightarrow \pi^{\mathrm{top}}_n(F, f_0)\rightarrow \pi^{\mathrm{top}}_n(E, e_0)\rightarrow \pi^{\mathrm{top}}_n(B, b_0)\rightarrow \pi^{\mathrm{top}}_{n-1}(F, f_0)\rightarrow \pi^{\mathrm{top}}_{n-1}(E, e_0)\rightarrow \pi^{\mathrm{top}}_{n-1}(B, b_0)\rightarrow\cdots
\end{equation}
of topological homotopy groups. 

In the \'etale setting an analogous result is proved in \cite[corollaire 1.4, Expos\'e X]{SGA1}. Let $K$ be a field and let $X$ and $S$ be locally Noetherian connected $K$-schemes. Let $f:X\rightarrow S$ be a proper separable morphism with geometrically connected fibers. Also, let $x$ be a geometric point of $X$, $s=f \circ x$ and let $X_{s}$ be the fiber of $X$ at the geometric point $s$. Then we have the exact sequence
$$
\pi_1^{\mathrm{et}}(X_{s}, x)\rightarrow \pi_1^{\mathrm{et}}(X,x)\rightarrow \pi_1^{\mathrm{et}}(S,s)\rightarrow 1
$$   
of \'etale fundamental groups 
and the first map is injective when $S = \Sp(K)$ \cite[th\'eor\`eme 6.1, Expos\'e IX]{SGA1}.  

We are interested in proving an analogous result for the log algebraic fundamental group, which is defined as the Tannaka dual of a certain category of 
modules with integrable connection. 

For a smooth scheme $X$ over a field $K$ of characteristic $0$ and 
a $K$-rational point $x$ of $X$, the algebraic fundamental group $\pi_1(X/\Sp(K), x)$ is defined as 
the Tannaka dual of the category of $\mathcal{O}_X$-coherent $\mathcal{D}_X$-modules 
(which is equivalent to the category of coherent $\mathcal{O}_X$-modules with integrable connection 
because $K$ is of characteristic $0$). 

Then the homotopy exact sequence for algebraic fundamental groups is described as follows. 
Let $X$ and $S$ be smooth connected $K$-schemes of finite type and let $f:X\rightarrow S$ be a proper smooth morphism with geometrically connected fibers. Also, let $x$ be a $K$-rational point of $X$, put $s=f(x)$ and denote by $X_s$ the fiber of $f$ over $s$:
\begin{equation}\label{diagsmooth}
\xymatrix{
X_s\ar[r]\ar[d]&X\ar[d]^f \\
\Sp(K)\ar[r]^-s\ar[dr]&S\ar[d]\\
&\Sp(K). 
}
\end{equation}
Then the following sequence is exact: 
$$
\pi_1(X_{s}/\Sp(K), x)\rightarrow \pi_1(X/\Sp(K),x)\rightarrow \pi_1(S/\Sp(K),s)\rightarrow 1. 
$$

The homotopy exact sequence in this context is proven by Zhang \cite{Zha13}; another proof is given by Dos Santos \cite{Dos15} with a different method which is applicable also to the case where $K$ is of characteristic $p>0$. Lazda \cite{Laz15} proves the exactness of the homotopy sequence for the relatively pro-unipotent quotient of algebraic fundamental 
groups when $f$ admits a section, proving also the injectivity of the first map. Moreover, he studies the $p$-adic version, using the category of relatively unipotent isocrystals as a Tannakian category\footnote{There is also a work of Lazda and P\'al \cite{LP17}, which  appeared after the first version of this article was written.}. 

The first result of this paper is the analogue of the homotopy exact sequence for certain log schemes. 
The geometric situation we are interested in is as follows. Let $K$ be a field of characteristic $0$ and 
let $X$ be a 
% quasi-projective 
normal crossing log variety over $K$, which means that $X$ is \'etale locally a finite union of smooth varieties which meet transversally, equipped with a semistable log structure in the sense of \cite[definition 11.6]{KatoF96}.  We denote by $X^{\times}$ such a log scheme.  We denote by $\Sp(K)$ the 
spectrum of $K$ endowed with the trivial log structure and by $\Sp(K)^{\times}$ the 
spectrum of $K$ endowed with the log structure associated to the monoid homomorphism $\mathbb{N} \to K; \,\, 1 \mapsto 0$. 
Then the map of log schemes $f: X^{\times}\rightarrow \Sp(K)^{\times}$ induced by the structural morphism is log smooth. On the other hand, the morphisms  
$g: \Sp(K)^{\times}\rightarrow \Sp(K)$, $h := g \circ f: X^{\times}\rightarrow \Sp(K)$ 
are not log smooth, but they are still considered to be good morphisms. 
If we are given a $K$-rational point $x$ of $X$, 
the diagram we obtain as an analogue of the diagram \eqref{diagsmooth} 
is the following: 
$$
\xymatrix{
X^{\times}\ar[r]\ar[d]&X^{\times}\ar[d]^-{f}\\
\Sp(K)^{\times}\ar[r]^{\mathrm{id}}\ar[dr]&\Sp(K)^{\times}\ar[d]^{g}\\
&\Sp(K). 
}
$$
As in the smooth case described above, we would like to define the algebraic fundamental group as the 
Tannaka dual of the category of coherent $\mathcal{O}_X$-modules with integrable log connection. However, in this case, it turns out that the category of coherent $\mathcal{O}_X$-modules with integrable log connection is not Tannakian, because, unlike the case 
of smooth schemes over a field in characteristic $0$, it is not true that every coherent module with integrable connection is locally free in the log case. Neither can we consider the category of locally free modules with integrable connection on $X^{\times}/\Sp(K)^{\times}$ (or on $X^{\times}/\Sp(K)$) because it is not abelian. The category we consider is the category of locally free modules with integrable connection having nilpotent residues on $X^{\times}/\Sp(K)^{\times}$, $X^{\times}/\Sp(K)$ and 
 $\Sp(K)^{\times}/\Sp(K)$, and we define the log algebraic fundamental group as its Tannaka dual, which will be denoted by $\pi_1(X^{\times}/\Sp(K)^{\times}, x)$,  $\pi_1(X^{\times}/\Sp(K), x)$ and $\pi_1(\Sp(K)^{\times}/\Sp(K), \nu)$ respectively 
(for a suitable choice of the fiber functor $\nu$). We remark that $\pi_1(\Sp(K)^{\times}/\Sp(K), \nu)$ is indeed isomorphic to $\mathbb{G}_a$.

The first result of the paper is the following.
\begin{teo} Let the notation be as above. 
 Then the following sequence is exact: 
\begin{equation}\label{HES}
 \pi_1(X^{\times}/\Sp(K)^{\times}, x)\rightarrow \pi_1(X^{\times}/\Sp(K), x)\rightarrow \pi_1(\Sp(K)^{\times}/\Sp(K),\nu)\rightarrow 1.
 \end{equation}
\end{teo}

Moreover, we study the injectivity of the first arrow in \eqref{HES}. We expect that it is injective, by analogy with the topological case. Indeed, the log scheme $\Sp(K)^{\times}$ over $\Sp(K)$ should be thought as the analogue of the $1$-dimensional sphere (note that $\pi_1(\Sp(K)^{\times}/\Sp(K), \nu)\cong \mathbb{G}_a$ as remarked above). Hence, comparing with \eqref{HEStop}, we expect that there should be ``$\pi_2(\Sp(K)^{\times}/\Sp(K), x)$'', which should be $0$ whatever it means, at the left of $\pi_1(X^{\times}/\Sp(K)^{\times}, x)$ in the sequence \eqref{HES}. 

The study of the injectivity of the first arrow is the hardest part of the work. The result is as follows. 
 \begin{teo} Let the notations be as above. For a group scheme $G$ we denote by $G^{\mathrm{tri}}$ its maximal geometrically protrigonalizable quotient (for definition, see definition \ref{trig}). Then the following sequence is exact: 
\begin{equation}\label{HESdRsolv}
 1\rightarrow\pi_1(X^{\times}/\Sp(K)^{\times}, x)^{\mathrm{tri}}\rightarrow \pi_1(X^{\times}/\Sp(K), x)^{\mathrm{tri}}\rightarrow \pi_1(\Sp(K)^{\times}/\Sp(K),\nu)\rightarrow 1.
\end{equation}
\end{teo}
Here we note that $\pi_1(\Sp(K)^{\times}/\Sp(K),\nu) = \pi_1(\Sp(K)^{\times}/\Sp(K),\nu)^{\mathrm{tri}} = 
\mathbb{G}_a$. 
If $x$ comes from a section of $f: X^{\times} \rightarrow \Sp(K)^{\times}$, 
it induces a splitting $\pi(x^*): \pi_1(\Sp(K)^{\times}/\Sp(K),\nu)\rightarrow \pi_1(X^{\times}/\Sp(K), x)^{\mathrm{tri}}$ of the exact sequence \eqref{HESdRsolv}:  
\begin{equation*}
\xymatrix{
1\ar[r]& \pi_1(X^{\times}/\Sp(K)^{\times}, x)^{\mathrm{tri}}\ar[r]&\pi_1(X^{\times}/\Sp(K), x)^{\mathrm{tri}}\ar[r]& \pi_1(\Sp(K)^{\times}/\Sp(K),\nu)\ar[r] \ar @(dr,dl)[l]^{\pi(x^{*})}&1}. 
\end{equation*}
Hence we can define the monodromy action $N$ of $ \pi_1(\Sp(K)^{\times}/\Sp(K),\nu) = \mathbb{G}_a$ on $\pi_1(X^{\times}/\Sp(K), x)^{\mathrm{tri}}$ as follows: if we denote 
$\gamma \in \pi_1(X^{\times}/\Sp(K), x)^{\mathrm{tri}}$ 
the image of $1 \in \mathbb{G}_a(K) = K$ by $\pi(x^*)$, $N$ is given by 
$$N:\pi_1(X^{\times}/\Sp(K)^{\times}, x)^{\mathrm{tri}}\rightarrow \pi_1(X^{\times}/\Sp(K)^{\times}, x)^{\mathrm{tri}}; \,\,\,\, \tau\mapsto 
\gamma^{-1}\tau\gamma.$$
We will study (the prounipotent quotient of) this action in a forthcoming work in purely algebraic way, in which we will try to give an algebraic proof for the transcendental part of the proof of a $p$-adic anabelian criterion of good reduction for curves by Andreatta, Iovita and Kim \cite{AndIovKim15}.

The content of each section is as follows. In the second section, we describe our geometric setting and prove that various categories of modules with integrable connection are Tannakian. Then we define the log algebraic fundamental groups. In the third section, we prove the exactness of the homotopy sequence \eqref{HES} with four terms. The fourth section is devoted to the study of the kernel of the second map of \eqref{HES}: we describe it as the Tannaka dual of another category of modules with integrable connection. In the last section, we prove that the kernel of the maximal geometrically protrigonalizable quotient of the second map in \eqref{HES} is isomorphic to the maximal geometrically protrigonalizable quotient of the first term of \eqref{HES}, proving the exactness of the homotopy sequence \eqref{HESdRsolv} for the maximal geometrically protrigonalizable quotients.
\section*{Acknowledgements}
The main part of this work was done when the first author was at the Graduate School of Mathematical Sciences of the University of Tokyo supported by a postdoctoral fellowship and kaken-hi (grant-in-aid) of the Japanese Society for the Promotion of Science (JSPS). She would like to thank Niels Borne for his interest in this work and for having suggested to translate the notion of a $K(\pi, 1)$ space for the log point. She would also like to thank Olivier Benoist, Carlo Gasbarri and Adriano Marmora for useful discussions. 
The authors would like to thank the referee for the comments which greatly improved the article. The second author is partly supported by JSPS Grant-in-Aid for Scientific Research  (B)23340001, (C)25400008, (A)15H02048 and (C)17K05162.

\section{The categories of representations}\label{The categories of representations}

Throughout this paper, let $K$ be a field of characteristic $0$. 
Sheaves on schemes are considered on the \'{e}tale site. 
Also, throughout this paper, we freely use the notion of log structures developed in 
\cite{Kato89}, \cite{KatoF96}. Let $N$ be the log structure on $\Sp(K)$ associated to the pre log structure given by the monoid homomorphism $\mathbb{N}\rightarrow K; \,1 \mapsto 0$. By \cite[(1.3)]{Kato89}, $N$ is isomorphic to  
$\mathbb{N} \times K^{\times}$ through which the structure morphism 
$N \cong \mathbb{N} \times K^{\times} \to K$ of the log structure is written as 
$(0,a) \mapsto a, (n,a) \mapsto 0 \,(n \not= 0)$.
We consider $\Sp(K)$ as a log scheme in two ways: either endowed with the trivial log structure, in which case we denote it by $\Sp(K)$, or endowed with the log structure $N$, in which case we denote it by $\Sp(K)^{\times}$. The identity map $\Sp(K)\rightarrow \Sp(K)$ induces a morphism of log schemes $g:\Sp(K)^{\times}\rightarrow \Sp(K)$ which is not log smooth. 
By the definition given in \cite[(1.7)]{Kato89}, we see that 
the sheaf of $1$-differentials $\omega^1_{\Sp(K)^{\times}/\Sp(K)}$ is a $1$-dimensional $K$-vector space generated by the element $\mathrm{dlog} (1,1)$ with  
$(1,1) \in \mathbb{N} \times K^{\times} \cong N$. 

\begin{defi}\label{NC} 
Let $X^{\times} := (X,M)$ be a log scheme such that 
$X$ is connected. 
% and quasi-projective over $K$. 
We say that $X^{\times}$ is a normal crossing log variety 
if $X$ is a normal crossing variety in the sense of \cite[pp.342--343]{KatoF96} and $M$ is a log structure of semistable type in the sense of \cite[definition 11.6]{KatoF96}. 
\end{defi}

As explained in \cite[Section 11]{KatoF96}, if 
$X^{\times} := (X,M)$ is a normal crossing log variety, 
the structure morphism $X \to \Sp(K)$ is amplified to a morphism of log schemes 
$f:X^{\times} \to \Sp(K)^{\times}$. Also, \'etale locally on $X$, 
the structure morphism $X \to \Sp(K)$ 
factors as $X \to \Sp(K[x_1, \dots, x_n]/(x_1\dotsm x_r)) 
\to \Sp(K)$ for some $1 \leq r \leq n$ such that the first map is \'etale  and that 
$f$ admits a chart of the form 
\begin{equation}\label{chartoff}
\begin{array}{lll}
\mathbb{N}^r\rightarrow \mathcal{O}_X,  
& \mathbb{N}\rightarrow K,  & \mathbb{N}\rightarrow \mathbb{N}^r, \\
e_i \mapsto x_i
& 1\mapsto 0 & 
1 \mapsto (1, \dots, 1) 
\end{array}
\end{equation} 
where $e_i := (0, \dots, 1, \dots, 0)$ (with $1$ at the $i$-th place) and we denoted the image of $x_i 
\in K[x_1, \dots, x_n]/(x_1\dotsm x_r)$ in $\mathcal{O}_X$ also by $x_i$. 

The sheaf of $1$-differentials   
$\omega^1_{X^{\times}/\Sp(K)^{\times}}$ is a locally free $\mathcal{O}_X$-module which 
is described \'etale locally as 
\begin{equation}\label{formofomega1}
\omega^1_{X^{\times}/\Sp(K)^{\times}}\cong\dfrac{\bigoplus_{i=1}^r\mathcal{O}_{X}\mathrm{dlog}x_i}{\mathcal{O}_X \!\cdot\!\sum^r_{i=1}\mathrm{dlog}x_i} \oplus \bigoplus_{i =r+1}^n \mathcal{O}_X dx_{i}. 
\end{equation}

The morphism of log schemes $h:=g\circ f:X^{\times}\rightarrow \Sp(K)$ is not log smooth, but the sheaf of $1$-differentials is still locally free: \'etale locally it is given by 
$$\omega^1_{X^{\times}/\Sp(K)}\cong\bigoplus_{i=1}^r\mathcal{O}_{X}\mathrm{dlog}x_i \oplus \bigoplus_{ i=r+1}^n\mathcal{O}_X dx_{i}.$$

Throughout this paper, $X^{\times} := (X,M)$ will be a normal crossing log variety unless otherwise stated, and assume that the morphism $f:X^{\times} \to \Sp(K)^{\times}$ as above is given and fixed. (Thus $h = g \circ f$ is also given.) 

We are going to study the relation among the algebraic fundamental groups associated to $f:X^{\times}\rightarrow \Sp(K)^{\times}$, $h:X^{\times}\rightarrow \Sp(K)$ and $g:\Sp(K)^{\times}\rightarrow \Sp(K),$ defined as the Tannaka duals of the categories of certain coherent modules with integrable connection. 

\begin{rmk}
When we define the algebraic fundamental groups associated to $f$ and $h$, 
we will assume the existence of a $K$-rational point of $X$. 
This assumption implies geometric connectedness of $X$. 
\end{rmk}

We recall the definition of a module with integrable connection in the context of log geometry.

\begin{defi}\label{connections}

Let $(Y,L)\rightarrow(S,L')$ be a map of fine log schemes and let $E$ be an $\mathcal{O}_{Y}$-module. A connection on $E$ is an $\mathcal{O}_S$-linear map
$$\nabla_E: E \rightarrow E\otimes \omega^1_{(Y,L)/(S,L')}$$ 
which satisfies the Leibniz rule 
$$\nabla_E(ae)=a\nabla_E(e)+e\otimes da \quad (a \in \mathcal{O}_Y, e\in E).$$
We can extend $\nabla_E$ to a map 
$$ \nabla_{E,i}: E\otimes \omega^i_{(Y,L)/(S,L')}\xrightarrow{\nabla_{E,i}}E\otimes \omega^{i+1}_{(Y,L)/(S,L')}$$ 
for $i\geq1$ by $\nabla_{E,i}(e\otimes \omega)=e \otimes d\omega+\nabla_E(e)\wedge \omega$. We say that $\nabla_{E}$ is integrable if $\nabla_{E,1} \circ \nabla_{E}=0.$

Every $\mathcal{O}_Y$-module $E$ endowed with a connection $ \nabla_E$ induces an $\mathcal{O}_Y$-linear map from the sheaf of log derivations 
$\mathcal{D}er((Y,L)/(S,L')) = \mathcal{H}om_{\mathcal{O}_Y}(\omega^1_{(Y,L)/(S,L')}, \mathcal{O}_Y)$ \cite[definition 5.1, proposition 5.3]{KatoF96} to the sheaf of $\mathcal{O}_S$-endomorphisms of $E$
$$\nabla_E:\mathcal{D}er((Y,L)/(S,L'))\longrightarrow \mathcal{E}nd_{\mathcal{O}_S}(E)$$
defined by the composite
$$
\nabla_E(D):E\xrightarrow{\nabla_E} E\otimes \omega^1_{(Y,L)/(S,L')}\xrightarrow{\mathrm{id}\otimes D}E
\quad (D\in \mathcal{D}er((Y,L)/(S,L'))). $$ 
$\nabla_E(D)$ satisfies the Leibniz rule  
\begin{equation}\label{LeibnitzD}
\nabla_E(D)(ae)=a\nabla_E(D)(e)+D(a)e \quad 
(D\in \mathcal{D}er((Y,L)/(S,L')), a\in \mathcal{O}_Y, e\in E).
\end{equation}
Conversely, if $\omega^1_{(Y,L)/(S,L')}$ is locally free, every $\mathcal{O}_Y$-linear map
$$\mathcal{D}er((Y,L)/(S,L'))\longrightarrow \mathcal{E}nd_{\mathcal{O}_S}(E)$$
satisfying \eqref{LeibnitzD} comes from a unique connection. %(This follows form the fact that $E\otimes \omega^1\cong Hom(Hom(\omega^1, \mathcal{O}_X), E)$)

We denote by $MIC((Y,L)/(S,L'))$ the category whose objects are pairs $(E,\nabla_E)$ consisting of a coherent $\mathcal{O}_Y$-module $E$ and an integrable connection $\nabla_E$ and whose morphisms are horizontal morphisms: %of $\mathcal{O}_Y$-modules: 
here, for coherent modules $(E, \nabla_E)$, $(F, \nabla_F)$ with integrable 
connection, an $\mathcal{O}_Y$-linear morphism $\gamma: E \to F$ is called horizontal 
if it satisfies the equality $ (\gamma \otimes {\mathrm{id}}) \circ \nabla_E = \nabla_F \circ \gamma $. If $\omega^1_{(Y,L)/(S,L')}$ is locally free, this condition is equivalent to the 
condition $\gamma(\nabla_E(D)(e))=\nabla_F(D)(\gamma(e))$ for every $e\in E$ and $D\in \mathcal{D}er((Y,L)/(S,L'))$.
\end{defi}
\begin{rmk}\label{rmk:formal}
When $Y=\Sp(B)$ where $B$ is a Noetherian $I$-adically complete ring for $I$ an ideal of $B$ and $S=\Sp(A)$, we also consider the $I$-adic completion $\hat{\omega}^1_{(Y,L)/(S,L')}$ of  $\omega^1_{(Y,L)/(S,L')}$ (regarded as $B$-module). 
Assume that $\hat{\omega}^1_{(Y,L)/(S,L')}$ is finitely generated as $B$-module. 
Then a formal connection on a $B$-module $E$ is an $A$-linear map
$$\nabla_E: E \rightarrow E\otimes \hat{\omega}^1_{(Y,L)/(S,L')}$$ 
which satisfies the Leibniz rule
$$\nabla_E(be)=b\nabla_E(e)+e\otimes db \quad (b \in B, e\in E).$$
The notion of integrability is defined in the same way as before. 
We denote by $\widehat{MIC}((Y,L)/(S,L'))$ the category whose objects are pairs $(E,\nabla_E)$ consisting of a $B$-module $E$ and a formal integrable connection $\nabla_E$ and whose morphisms are horizontal, with the same meaning as before,  morphisms of $B$-modules. The module of formal log derivations 
$\widehat{\mathcal{D}er}((Y,L)/(S,L'))$ is defined by 
$\widehat{\mathcal{D}er}((Y,L)/(S,L')) := Hom_{B}(\hat{\omega}^1_{(Y,L)/(S,L')}, B)$. 
(A more intrinsic definition of $\widehat{\mathcal{D}er}((Y,L)/(S,L'))$ is possible but 
we do not need it.) As before, we have a description of a formal integrable connection 
in terms of the module of formal log derivations when $\hat{\omega}^1_{(Y,L)/(S,L')}$ is projective 
as $B$-module. 

Note that the category $\widehat{MIC}((Y,L)/(S,L'))$ depends not only on 
the morphism $(Y,L) \rightarrow (S,L')$ but also on the ideal $I$. However, we would like to 
denote it without mentioning $I$ explicitly in the symbol, because the ideal $I$ will be clear 
in the cases we treat in this paper. 
\end{rmk}
\begin{rmk}\label{rmk:formal-open}
Let $Y=\Sp(B)$, $I \subset B$, $S=\Sp(A)$, 
$\hat{\omega}^1_{(Y,L)/(S,L')}$ be as in remark \ref{rmk:formal} and 
let $Y' = \Sp(B') \subset \Sp(B) = Y$ be an affine open subscheme of $Y$. 
Then we put $\hat{\omega}^1_{(Y',L)/(S,L')} := B' \otimes_B 
\hat{\omega}^1_{(Y,L)/(S,L')}$, and using this, we 
define the notion of 
formal connection on a $B'$-module $E$ 
as an $A$-linear map
$$\nabla_E: E \rightarrow E\otimes \hat{\omega}^1_{(Y',L)/(S,L')}$$ 
satisfying the Leibniz rule. 
The notion of integrability is also defined.
We denote by $\widehat{MIC}((Y',L)/(S,L'))$ the category whose objects are pairs $(E,\nabla_E)$ consisting of a $B'$-module $E$ and a formal integrable connection $\nabla_E$ and whose morphisms are horizontal 
morphisms of $B'$-modules. 
\end{rmk}
\begin{rmk}\label{tensorstructure}
When $\omega^1_{(Y,L)/(S,L')}$ is locally free, 
the category $MIC((Y,L)/(S,L'))$ is an abelian tensor category. Given $(E, \nabla_E), (F, \nabla_F) \in MIC((Y,L)/(S,L'))$, the tensor product of them is defined as  the pair $(E\otimes F, \nabla)$, where $\nabla=\mathrm{id}\otimes \nabla_F+\nabla_E\otimes \mathrm{id}$. The object $(\mathcal{O}_Y, d)$, where $d$ is the composition of the differential $d:\mathcal{O}_Y\rightarrow \Omega^1_{Y/S}$ with the map $\Omega^1_{Y/S}\rightarrow \omega^1_{(Y,L)/(S,L')}$, is the unit object of  $MIC((Y,L)/(S,L'))$ and is called the trivial connection.
\end{rmk}

\begin{rmk}\label{rmk2.5}
The category $MIC(\Sp(K)^{\times}/\Sp(K))$ is equivalent to the category of finite dimensional $K$-vector spaces endowed with a $K$-linear endomorphism. 
\begin{comment}
More generally, if $(Y, L) \rightarrow (S,L')$ is a 
morphism of fine log schemes and if we denote by $L \oplus \mathbb{N}^r$ the log structure associated to 
the direct sum of $L$ and the log structure associated to the monoid homomorphism 
$$  \mathbb{N}^r \rightarrow \mathcal{O}_Y; \,\,\, e_i \mapsto 0 \,(1 \leq i \leq r), $$ 
the category $MIC((Y,L \oplus  \mathbb{N}^r)/(S,L'))$ is equivalent to the category of 
objects in $MIC((Y,L)/(S,L'))$ endowed with $r$ commuting endomorphisms. 
\end{comment}
\end{rmk}

We will use the theory of Tannakian categories over $K$ to define log algebraic fundamental groups. We recall the definition and the main theorem which relates the fundamental group and the category of its representations.
\begin{defi}\label{Tannakiancategory}
Let $\mathcal{C}$ be a rigid abelian tensor category with unit object $\underline{1}$. We say that $\mathcal{C}$ is a neutral Tannakian category over $K$ if $\mathrm{End}(\underline{1})=K$ and if there exists a $K$-linear faithful exact tensor functor (fiber functor) $\mathcal{C}\rightarrow \mathrm{Vec}_K$, where  $\mathrm{Vec}_K$ is the category of finite dimensional $K$-vector spaces. \end{defi}
\begin{teo}[{\cite[Theorem 2.11]{DelMil82}}]\label{Saavedra}
Let $\mathcal{C}$ be a neutral Tannakian category over $K$ endowed with a fiber functor $\omega: \mathcal{C} \to \mathrm{Vec}_K$. Then the functor 
$$(K\text{-}\mathrm{algebras})\longrightarrow \mathrm{(Groups)}$$
which sends a $K$-algebra $R$ to the group of tensor automorphisms $\mathrm{Aut}^{\otimes}(\mathcal{C}\longrightarrow \mathrm{Vec}_K\longrightarrow \mathrm{Mod}_R)$ (where $\mathrm{Mod}_R$ is the category of $R$-modules)
is representable by a pro-algebraic group $G(\mathcal{C}, \omega)$ over $K$.
Moreover, $\omega$ induces an equivalence of categories 
$$\mathcal{C}\cong \mathrm{Rep}_K(G(\mathcal{C}, \omega)),$$
where $\mathrm{Rep}_K(G(\mathcal{C}, \omega))$ denotes the category of finite dimensional $K$-representations of $G(\mathcal{C}, \omega)$. 

\end{teo}

It is also known that, 
for a pro-algebraic group $G$ over $K$, 
$\mathrm{Rep}_K(G)$ is a neutral Tannakian category 
with forgetful functor $\omega: \mathrm{Rep}_K(G) \to \mathrm{Vec}_K$ as fiber functor  
and that 
there is a canonical isomorphism $G \overset{\cong}{\to} G(\mathrm{Rep}_K(G),\omega)$ 
(\cite[Proposition 2.8]{DelMil82}).

\begin{rmk}
We cannot consider the whole categories $ MIC(X^{\times}/\Sp(K)^{\times})$ and $MIC(X^{\times}/\Sp(K))$ to define log algebraic fundamental groups, 
because they are not Tannakian. 
We explain this for the category $MIC(X^{\times}/\Sp(K)^{\times}),$ but analogous examples exist for $MIC(X^{\times}/\Sp(K)).$ Unlike the case of trivial log structures, it is no longer true that every coherent module with integrable connection is locally free. To see this, we consider as $X^{\times}/\Sp(K)^{\times}$ the log schemes 
$(\Sp(K[x,y]/(xy)),M)/\Sp(K)^{\times}$, where $M$ is the log structure 
associated to $\mathbb{N}^2 \to K[x,y]/(xy); \, e_1 \mapsto x, e_2 \mapsto y$ 
and the morphism $(\Sp(K[x,y]/(xy)),M) \to \Sp(K)^{\times}$ over $\Sp(K)$ 
is defined by $\mathbb{N} \to \mathbb{N}^2; \, 1 \mapsto e_1 + e_2$. 
In this case, the module of $1$-differentials is 
$$\omega^1_{(\Sp(K[x,y]/(xy)), M)/\Sp(K)^{\times}}\cong \frac{(K[x,y]/(xy))\mathrm{dlog}x\oplus (K[x,y]/(xy))\mathrm{dlog}y}{(K[x,y]/(xy))(\mathrm{dlog}x+\mathrm{dlog}y)},$$
which is free of rank $1$. 
Then, $(xK[x,y]/(xy), d)$, where $d$ is the connection induced by the trivial connection on $K[x,y]/(xy)$, is an object of $MIC(X^{\times}/\Sp(K)^{\times})$; indeed, $d(xf)=fdx+xdf=x(f\mathrm{dlog}x+df)$. But $xK[x,y]/(xy)$ is a coherent $K[x,y]/(xy)$-module which is not free. 
This object is a non-trivial subobject of the unit object 
$(K[x,y]/(xy), d)$, which is impossible for a Tannakian category over $K$. %\cite[proposition 1.17]{DelMil82}

Even if we consider the full subcategory of $MIC(X^{\times}/\Sp(K)^{\times})$ of locally free modules with integrable connection, this is not a Tannakian category because it is not abelian. To see this, we consider the affine $1$-dimensional case as before and we look at the map of free modules with connections
$$\varphi:(K[x,y]/(xy), \nabla)\longrightarrow (K[x,y]/(xy), d); \quad f\mapsto xf$$
where $\nabla(f):=df+f\mathrm{dlog}x.$ One can check that the map $\varphi$ is horizontal and that $\mathrm{Coker}(\varphi)$  is given by $(K[y], d')$ (where $d'$ is the induced connection), which is not locally free as $K[x,y]/(xy)$-module. 
\end{rmk}
If $H^0_{\mathrm{dR}}(X^{\times}/\Sp(K))$ (resp. $H^0_{\mathrm{dR}}(X^{\times}/\Sp(K)^{\times})$) is a field, the full subcategory of $MIC(X^{\times}/\Sp(K))$ (resp. of $MIC(X^{\times}/\Sp(K)^{\times})$) consisting of the unipotent objects 
(iterated extensions of the unit object) 
is a Tannakian category (\cite[proposition 3.1.2]{Shi00}).
We want to construct a subcategory of $MIC(X^{\times}/\Sp(K))$ (resp. $MIC(X^{\times}/\Sp(K)^{\times})$) which is Tannakian and properly contains the subcategory of all the unipotent objects. To do this we introduce a notion of nilpotent residues, which will be a punctual notion as in \cite[definition 2.1.1]{Ogu03}.

The definitions which follow are written for a morphism of fine log schemes of the form   $(Y,L)\rightarrow(\Sp(K),L')$. This will be mainly applied to $f:X^{\times}\rightarrow \Sp(K)^{\times}$, $h:X^{\times}\rightarrow \Sp(K)$ and $g:\Sp(K)^{\times}\rightarrow \Sp(K)$.

Let $a: (Y,L)\rightarrow(\Sp(K),L')$ be as above. Let $y$ be a geometric point 
of $Y$, $\mathfrak{m}_{Y,y}$ the unique maximal ideal of $\mathcal{O}_{Y,y}$ 
($:=$ the strict localization of $\mathcal{O}_Y$ at $y$) 
and $K(y)=\mathcal{O}_{Y,y}/\mathfrak{m}_{Y,y}$ the residue field. 
As observed in \cite[lemma 1.3.1]{Ogu94}, there exists a surjective morphism 
$$\omega^1_{(Y,L)/(\Sp(K),L')}(y)\rightarrow K(y)\otimes \overline{L}^{\mathrm{gp}}_{y}/\overline{L'}^{\mathrm{gp}}_{y},$$
where $\omega^1_{(Y,L)/(\Sp(K),L')}(y)=\omega^1_{(Y,L)/(\Sp(K),L'),y}\otimes K(y)$, $\overline{L}^{\mathrm{gp}}_{y}=L^{\mathrm{gp}}_{y}/\mathcal{O}^*_{Y,y}$,  $\overline{L'}^{\mathrm{gp}}_{y}=(a^*L')^{\mathrm{gp}}_{y}/\mathcal{O}^*_{Y,y}$ and  $L^{\mathrm{gp}}$, $(a^*L')^{\mathrm{gp}}$ are the sheaves of groups associated to $L$ and $a^*L'$, respectively.
The fineness of the log structures $L, L'$ implies that 
$K(y)\otimes \overline{L}^{\mathrm{gp}}_{y}/\overline{L'}^{\mathrm{gp}}_{y}$ is 
a finite-dimensional $K$-vector space. 

%% We define the residue of $(E, \nabla_E)$ at $y$ as follows.
Let $(E, \nabla_E)$ be an object of  $MIC((Y, L)/(S,L'))$ and 
we denote by $E(y)=E_y/\mathfrak{m}_{Y,y} E_y\cong E_y\otimes K(y)$ the fiber of $E$ at $y$, which is 
a finite dimensional vector space over $K(y)$. Then, $\nabla_E$ induces a unique linear morphism called the residue $\rho_{y}:E(y)\rightarrow E(y)\otimes \overline{L}^{\mathrm{gp}}_{y}/\overline{L'}^{\mathrm{gp}}_{y}$; indeed, if $e\in E(y)$ and $b\in K(y)$, $\rho_{y}(be)=b\nabla_E(e)+e \otimes d(b)$ and $d(b)=0$ in $\overline{L}^{\mathrm{gp}}_{y}/\overline{L'}^{\mathrm{gp}}_{y}$. 

\begin{defi}\label{residuenilpMN}
Let $(E, \nabla_E)$ be an object of  $MIC((Y, L)/(\Sp(K),L'))$. 
For a geometric point $y$ of $Y$, 
we say that $(E, \nabla_E)$ has nilpotent residues at $y$ if, for every $K(y)$-linear map $ t_{y}:K(y)\otimes \overline{L}^{\mathrm{gp}}_{y}/\overline{L'}^{\mathrm{gp}}_{y}\rightarrow K(y)$, the composite map $({\rm id}\otimes t_y)\circ \rho_y$ is nilpotent:
\begin{equation*}
\xymatrix{
E(y)\ar[dr]\ar[r]^(.33){\rho_{y}}&E(y)\otimes \overline{L}^{\mathrm{gp}}_{y}/\overline{L'}^{\mathrm{gp}}_{y}\ar[d]^{{\rm id} \otimes t_{y}}\\
                             &E(y). 
}
\end{equation*}
We say that $(E, \nabla_E)$ has nilpotent residues if it has nilpotent residues at any 
geometric point over a closed point of $Y$. 

We denote by $MIC((Y,L)/(\Sp(K),L'))^{\mathrm{nr}}$ the category of locally free $\mathcal{O}_{Y}$-modules of finite rank on $(Y,L)/(\Sp(K),L')$ with integrable connection having nilpotent residues.

\end{defi}

\begin{rmk}
It is an interesting problem if any object $(E,\nabla)$ in $MIC((Y,L)/(\Sp(K),L'))^{\mathrm{nr}}$ 
has nilpotent residues at any geometric points of $Y$ which is not necessarily over a closed point, 
namely, the nilpotence of residues at closed points implies that over any points. 
Unfortunately, we do not have an answer in general case. Nevertheless,  
we remark here that the above question has the affirmative answer in the case 
where $(Y,L)/(\Sp(K),L')$ is $X^{\times}/\Sp(K)$, $X^{\times}/\Sp(K)^{\times}$ or 
$\Sp(K)^{\times}/\Sp(K)$, namely, the case of our main interest. 

We give a proof only in the case $X^{\times}/\Sp(K)$, because the proof in 
the other cases is similar. Since the assertion is local, we may assume that 
there exists an strict etale morphism 
$$X^{\times} \to (\Sp(K[x_1, \dots, x_n]/(x_1 \cdots x_r), M), $$ 
where $M$ is the log structure associated to 
$\mathbb{N}^r \to K[x_1, \dots, x_n]/(x_1 \cdots x_r); e_i \mapsto x_i$. 
For $i \in \{1, \dots, r\}$, let $X_i$ be the zero locus of $x_i$ and 
for $I \subseteq \{1, \dots, r\}$, put $X_I := \bigcap_{i \in I} X_i, 
X_I^{\circ} := X_I \setminus \bigcup_{J \supsetneq I} X_J$. 
Then $X = \sqcup_{I \subseteq \{1, \dots, r\}} X_I^{\circ}$ set-theoretically. 
Also, we have the equality $\overline{M}^{\rm gp}_{|X_I^{\circ}} = \Z^{|I|}_{X^{\circ}_I}$, 
where $\overline{M}^{\rm gp} := M^{\rm gp}/\mathcal{O}_X^*$. 

Now take any object $(E,\nabla)$ in $MIC(X^{\times}/\Sp(K))^{\mathrm{nr}}$, 
a geometric point $x$ of $X$ which is not necessarily over a closed point. Let 
$I$ be the subset of $\{1, \dots, r\}$ such that the image of $x$ in $X$ belongs to 
$X_I^{\circ}$. Then $\nabla$ induces a linear map 
$$\rho_I: E_{|X^{\circ}_I} \rightarrow E_{|X^{\circ}_I} \otimes 
\overline{M}^{\rm gp}_{|X_I^{\circ}} = E_{|X^{\circ}_I} \otimes 
\Z^{|I|}_{X^{\circ}_I}.$$
For a geometric point $y$ of $X^{\circ}_I$, 
we denote the specialization of $\rho_I$ to $y$ by 
$$ \rho_y: E(y) \rightarrow E(y) \otimes 
\overline{M}^{\rm gp}_{y} = E(y) \otimes 
\Z^{|I|}. $$
To prove the assertion, we should prove that, for any $K(x)$-linear map $t_x: K(x) \otimes \overline{M}^{\rm gp}_{x} = K(x) \otimes \Z^{|I|} \rightarrow K(x)$, the map $({\rm id} \otimes t_x) \circ \rho_x: E(x) \rightarrow E(x)$ is nilpotent. 
By replacing $X$ by a small affine open subscheme containing the image of $x$ in $X$, 
we may assume that $E_{|X^{\circ}_I}$ is free of rank $r$. 

For $1 \leq i \leq |I|$, 
we denote the $i$-th projection 
$\overline{M}^{\rm gp}_{|X_I^{\circ}} = \Z^{|I|}_{X_I^{\circ}} \rightarrow \Z_{X_I^{\circ}}$
by $\pi_i$ and denote its fiber 
$\overline{M}^{\rm gp}_{x} = \Z^{|I|} \rightarrow \Z$ at $x$ by 
$\pi_{i,x}$. Then, by integrability of $\nabla$, it suffices to prove the nilpotence of 
$\rho_{x,i} := ({\rm id} \otimes \pi_{i,x}) \circ \rho_x: E(x) \rightarrow E(x)$ for any $i$. 
Since this map is the specialization of the map 
$$ \rho_{I,i}:= ({\rm id} \otimes \pi_i) \circ \rho_I: E_{|X^{\circ}_I} \rightarrow 
E_{|X^{\circ}_I}, $$
we are reduced to proving the nilpotence of the map $\rho_{I,i}$. 
Choose for any closed point of $X^{\circ}_I$ a geometric point over it and 
let $S$ be the set of such geometric points. Then, by assumption, 
for any $y \in S$ and for any $1 \leq i \leq |I|$, the map $\rho_{y,i}$ is nilpotent. 
Hence, as an element in ${\rm End}(E_{|X^{\circ}_I}(y))$, $\rho_{y,i}^{r+1}$ is equal to $0$. 
On the other hand, since $X^{\circ}_I$ is reduced and Jacobson and $E_{|X^{\circ}_I}$ is free, 
the natural map 
${\rm End}(E_{|X^{\circ}_I}) \to \prod_{y \in S} {\rm End}(E_{|X^{\circ}_I}(y))$ 
is injective. Hence $\rho_{I,i}^{r+1}$ is equal to $0$ in ${\rm End}(E_{|X^{\circ}_I})$. 
So the proof of the assertion is finished. 
\end{rmk}

\begin{rmk}\label{rmk:formalnilpotentres}
When we are in the situation of remark \ref{rmk:formal} (resp.  
remark \ref{rmk:formal-open}) with $A = K$, 
we can define the notion of having nilpotent residues at a geometric point $y$ of $Y$ analogously to definition \ref{residuenilpMN} for formal integrable connections. 
We denote the category of finitely generated projective $B$-modules 
(resp. $B'$-modules) 
with formal integrable connection having nilpotent residues at any geometric point over a closed point of $Y$ by $\widehat{MIC}((Y, L)/(\Sp(K),L'))^{\mathrm{nr}}$ (resp. $\widehat{MIC}((Y', L)/(\Sp(K),L'))^{\mathrm{nr}}$). 
\begin{comment}
Note that the definition of the categories 
$\widehat{MIC}((Y, L)/(\Sp(K),L'))^{\mathrm{nr}}$, $\widehat{MIC}((Y', L)/(\Sp(K),L'))^{\mathrm{nr}}$ 
here is slightly different from the definition of the category $MIC((Y,L)/(\Sp(K),L'))^{\mathrm{nr}}$ above 
in the sense that the nilpotence of residues are imposed only on geometric points over closed points 
here. This is due to technical reasons.
\end{comment}   
\end{rmk}
\begin{rmk} \label{vecspatannakian}The category $MIC(\Sp(K)^{\times}/\Sp(K))^{\mathrm{nr}}$ is equivalent to the category of finite dimensional $K$-vector spaces endowed with a nilpotent endomorphism. 
%It is a rigid abelian tensor category with unit object given by the pair $(K, 0)$. It is a neutral Tannakian category over $K$; actually  $\mathrm{End}((K,0))\cong K$, 
It is a neutral Tannakian category and 
the functor
$$\nu:MIC(\Sp(K)^{\times}/\Sp(K))^{\mathrm{nr}}\longrightarrow \mathrm{Vec}_K; 
\quad (V,\nabla_V) \mapsto V $$
%which sends an object $(V, \nabla_V)\in MIC(\Sp(K)^{\times}/\Sp(K))^{\mathrm{nr}}$ to the $K$-vector space $V$ 
gives a fiber functor.
\begin{comment}
More generally, if $(Y, L) \rightarrow (\Sp(K),L')$ is a 
morphism of fine log schemes and 
if $(Y,L \oplus  \mathbb{N}^r)$ is as in remark \ref{rmk2.5}, 
the category $MIC((Y,L \oplus  \mathbb{N}^r)/(\Sp(K),L'))^{\mathrm{nr}}$ 
is equivalent to the category of 
objects in $MIC((Y,L)/(\Sp(K),L'))^{\mathrm{nr}}$ endowed with 
$r$ commuting locally nilpotent endomorphisms. 
\end{comment} 
\end{rmk}
\begin{defi}\label{pi1KN/K} 
We define the log algebraic fundamental group of $\Sp(K)^{\times}$ over $\Sp(K)$ with base point $\nu$ as the Tannaka dual of $MIC(\Sp(K)^{\times}/\Sp(K))^{\mathrm{nr}}$, \emph{i.e.} 
$$\pi_{1}(\Sp(K)^{\times}/\Sp(K), \nu):=G(MIC(\Sp(K)^{\times}/\Sp(K))^{\mathrm{nr}}, \nu)$$
with the notation as in theorem \ref{Saavedra}.
\end{defi}

The group $\pi_{1}(\Sp(K)^{\times}/\Sp(K), \nu)$ is isomorphic to $\mathbb{G}_a$: indeed, in 
\cite[example VIII 2.1, example VIII 6.4]{Mil12}, 
it is shown that any object $\rho: \mathbb{G}_a \to GL(V)$ in 
$\mathrm{Rep}_K(\mathbb{G}_a)$ has the form 
$\rho(t) = \exp(Nt)$ for a unique nilpotent endomorphism $N$ on $V$, and so 
the category $\mathrm{Rep}_K(\mathbb{G}_a)$ is equivalent to the category of 
finite dimensional $K$-vector spaces endowed with a nilpotent endomorphism, which is nothing but $MIC(\Sp(K)^{\times}/\Sp(K))^{\mathrm{nr}}$. 
Thus we see that 
$$\mathbb{G}_a \cong G(MIC(\Sp(K)^{\times}/\Sp(K))^{\mathrm{nr}}, \nu) =: 
\pi_{1}(\Sp(K)^{\times}/\Sp(K), \nu)$$ 
by the fact given after theorem \ref{Saavedra}.

We are going to prove now that $MIC(X^{\times}/\Sp(K)^{\times})^{\mathrm{nr}}$ is a neutral Tannakian category over $K$ under suitable assumptions. 
\begin{prop}\label{relativeabelian}
The category  $MIC(X^{\times}/\Sp(K)^{\times})^{\mathrm{nr}}$ is abelian. 
\end{prop}
\begin{proof} Since $MIC(X^{\times}/\Sp(K)^{\times})^{\mathrm{nr}}$ is a full subcategory of the abelian category $MIC(X^{\times}/\Sp(K)^{\times})$, it is sufficient to prove that the kernel and the cokernel of any morphism in $MIC(X^{\times}/\Sp(K)^{\times})^{\mathrm{nr}}$ are locally free $\mathcal{O}_X$-modules with nilpotent residues.

To prove that a coherent $\mathcal{O}_X$-module $E$ is locally free it is enough to prove that, for every geometric point $x$ over a closed point in $X$, the stalk $E_x$ is a free $\mathcal{O}_{X, x}$-module.  
It is moreover enough to prove this at the level of completed local ring 
 $\widehat{\mathcal{O}}_{X, x}$, which has the form 
$K[[x_1, \dots, x_n]]/(x_1\cdots x_r)$ with $K$ algebraically closed. 
Also, to prove the nilpotence of residues, it suffices to check it at the closed point 
of $\Sp(\widehat{\mathcal{O}}_{X, x})$. 
Thanks to lemma \ref{RS} below, we can suppose that 
$\widehat{\mathcal{O}}_{X, x} \cong K[[x_1, \dots, x_n]]/(x_1\cdots x_n)$ with $n \geq 2$.
Then we proceed by induction on $n$: in proposition \ref{lacrocesuN} below 
we prove the result for $n=2$, and 
in proposition \ref{inductivepass} below we prove the induction step.

\end{proof}

\begin{rmk}\label{notationlocalderivations} In what follows we consider on 
the spectrum of  $S=K[[x_1, \dots, x_n]]/(x_1\cdots x_r)$ 
the log structure $M$ associated to 
the monoid homomorphism 
$$\mathbb{N}^r \rightarrow 
K[[x_1, \dots, x_n]]/(x_1\cdots x_r); \,\,\,\, e_i \mapsto x_i. $$ 
Then 
$$
\hat{\omega}^1_{(\Sp(S), M)/\Sp(K)^{\times}}\cong\frac{\bigoplus_{i=1}^rS\mathrm{dlog}x_i}{S \cdot \sum^r_{i}\mathrm{dlog}x_i} \oplus \bigoplus_{i = r+1}^n S dx_{i}$$ 
and a basis of it is given by 
$\{\mathrm{dlog}x_1, 
% \mathrm{dlog}x_1 + \mathrm{dlog}x_2, 
\sum_{i=1}^{2}\mathrm{dlog}x_{i}, \dots, 
\sum_{i=1}^{r-1}\mathrm{dlog}x_{i},dx_{r+1}, \dots, dx_{n}\}. $  
We will use several times in the remaining part of this section its dual basis as a basis
of $$\widehat{\mathcal{D}er}((\Sp(S), M)/\Sp(K)^{\times}) = Hom(\hat{\omega}^1_{(\Sp(S), M)/\Sp(K)^{\times}}, S),$$ which we denote by 
$\{\partial_1, \dots, \partial_{r-1}, D_{r+1}, \dots, D_n\}$; namely, 
for $i=1, \dots, r-1$ we denote by $\partial_i$ the derivation that sends $\mathrm{dlog}x_i$ to $1$, $\mathrm{dlog}x_{i+1}$ to $-1$, $\mathrm{dlog}{x_j}$ to $0$ for every $j\neq i, i+1$, and $dx_j$ to $0$ for every $j=r+1, \dots, n$, while we denote by $D_i$ for $i=r+1, \dots, n$ the derivation which sends $d{x_{i}}$ to $1$, $d{x_{j}}$ to $0$ for $j=r+1, \dots, n$, $j\neq i$ and $\mathrm{dlog}{x_j}$ to $0$ for every $j=1, \dots, r-1$. 
For a formal connection $\nabla_E$ on a coherent $S$-module $E$, 
the integrability of $\nabla_E$ is equivalent to the commutativity of 
the operators $\nabla_E(\partial_1), \dots, \nabla(\partial_{r-1}), 
\nabla_E(D_{r+1}), \dots, \nabla_E(D_n)$. 
\end{rmk}

\begin{lemma}\label{RS}
Let $R=K[[x_1, \dots, x_r]]/(x_1\cdots x_r)$ and $S=K[[x_1, \dots, x_n]]/(x_1\cdots x_r)$ be as above. 
If the category $\widehat{MIC}((\Sp(R), M)/\Sp(K)^{\times})^{\mathrm{nr}}$ is stable by kernel and cokernel of any morphism, then the category $\widehat{MIC}((\Sp(S), M)/\Sp(K)^{\times})^{\mathrm{nr}}$ is stable by kernel and cokernel of any morphism.
\end{lemma}

\begin{proof}

For an object $(E,\nabla_E)$ in 
$\widehat{MIC}((\Sp(S), M)/\Sp(K)^{\times})^{\mathrm{nr}}$, put 
$$ \overline{E} := \{ e \in E \,|\, \nabla_E(D_i)(e) = 0 \,(r+1 \leq i \leq n)\}. $$ 
Then $\overline{E}$ is stable by the action of $\nabla_E(\partial_i) \, (1 \leq i \leq r-1)$. 
Then, it suffices to prove that $\overline{E}$ with the above action defines 
an object in $\widehat{MIC}((\Sp(R), M)/\Sp(K)^{\times})^{\mathrm{nr}}$ and that 
the functor 
$$ \widehat{MIC}((\Sp(S), M)/\Sp(K)^{\times})^{\mathrm{nr}} \to 
\widehat{MIC}((\Sp(R), M)/\Sp(K)^{\times})^{\mathrm{nr}}; \quad (E,\nabla_E) \mapsto \overline{E} $$ 
is an equivalence of categories whose quasi-inverse is given by 
$(\overline{E},\nabla_{\overline{E}}) \mapsto (E := \overline{E} \otimes_R S, \nabla_E)$ 
with the action $\nabla_E(\partial_i) \, (1 \leq i \leq r-1), \nabla_E(D_i)\,(r+1 \leq i \leq n)$ defined by 
$\nabla_E(\partial_i) := {\rm id} \otimes d(\partial_i) + \nabla_{\overline{E}}(\partial_i) \otimes {\rm id}, 
\nabla_E(D_i) := {\rm id} \otimes d(D_i)$: 
indeed, if this claim is proved and if we are given a morphism $\varphi: E \to F$ in 
$\widehat{MIC}((\Sp(S), M)/\Sp(K)^{\times})^{\mathrm{nr}}$, it induces a morphism 
$\overline{\varphi}: \overline{E} \to \overline{F}$ in 
$\widehat{MIC}((\Sp(R), M)/\Sp(K)^{\times})^{\mathrm{nr}}$, and 
${\rm Ker}(\overline{\varphi}), {\rm Coker}(\overline{\varphi})$ are defined as objects in 
$\widehat{MIC}((\Sp(R), M)/\Sp(K)^{\times})^{\mathrm{nr}}$. Then we have 
${\rm Ker}(\varphi) = {\rm Ker}(\overline{\varphi}) \otimes_R S, 
{\rm Coker}(\varphi) = {\rm Coker}(\overline{\varphi}) \otimes_R S$ and they are 
objects in $\widehat{MIC}((\Sp(S), M)/\Sp(K)^{\times})^{\mathrm{nr}}$, as required. 

We prove the claim in the previous paragraph, following  
\cite[proposition 8.9]{Kat70}. Define the map $P: E \to E$ by 
$$ P(e) := \sum_{{\bold k} \in \mathbb{N}^{n-r}} (-1)^{|\bold k|} x^{\bold k} D^{\bold k}(e)/{\bold k}!, $$
where, for ${\bold k} = (k_1, \dots, k_{n-r}) \in \mathbb{N}^{n-r}$, 
$|{\bold k}| = \sum_{i=1}^{n-r} k_i$, $x^{\bold k} = x_{r+1}^{k_1} \cdots x_{n}^{k_{n-r}}$, $D^{\bold k} = D_{r+1}^{k_1} \cdots D_n^{k_{n-r}}$ and ${\bold k}! = k_1! \cdots k_{n-r}!$. 
Then we see that the image of $P$ is contained in $\overline{E}$ and  that 
$P|_{\overline{E}} = {\rm id}_{\overline{E}}$. Moreover, 
if we denote the kernel of the surjection $\pi: S \to R$ defined by 
$x_i \mapsto x_i \, (1 \leq i \leq r), x_i \mapsto 0 \, (r+1 \leq i \leq n)$ by $I$, we have 
$P(e) \equiv e \, ({\rm mod}\,IE)$. Thus we see that $P$ induces the isomorphism 
\begin{equation}\label{eq:eie}
E/IE \xrightarrow{\cong} \overline{E}. 
\end{equation}
In particular, $\overline{E}$ is a free $R$-module 
and it has nilpotent residues. 
So $\overline{E}$ (endowed with the action of $\nabla_E(\partial_i) \, (1 \leq i \leq r-1)$) 
defines an object in  
$\widehat{MIC}((\Sp(R), M)/\Sp(K)^{\times})^{\mathrm{nr}}$. 
To prove the rest of the claim, it suffices to prove that the natural map 
$\iota: \overline{E} \otimes_R S \to E$ is an isomorphism. 
Surjectivity follows from the isomorphism \eqref{eq:eie}. 
To prove injectivity, let $e_1, ..., e_s$ be a basis of $\overline{E}$ and 
assume $\sum_{i=1}^s f_ie_i = 0$ in $E$ with $f_i \in S$. 
If $f_i$ is non-zero for some $i$, 
we have $\pi(D^{\bold k}(f_i)) \not= 0$ for some $i$ and 
some ${\bold k} \in \mathbb{N}^{n-r}$. 
Then we have 
$$ 0 = P(D^{\bold k}(\sum_{i=1}^s f_ie_i)))= \sum_{i=1}^s \pi(D^{\bold k}(f_i)) e_i, $$
and this contradicts the linear independence of $e_i$'s over $R$. Thus 
$f_i = 0$ for all $i$ and this shows the injectivity of the map $\iota$. 
So we are done. 
\end{proof}

\begin{prop}\label{lacrocesuN} 
The category $\widehat{MIC}((\Sp(K[[x,y]]/(xy)), M)/\Sp(K)^{\times})^{\mathrm{nr}}$ is abelian. 
\end{prop}
\begin{proof}
Let $\partial_1$ be as in the notation introduced in remark \ref{notationlocalderivations}, 
which is a basis of $$\widehat{\mathcal{D}er}((\Sp(K[[x,y]]/(xy)), M)/\Sp(K)^{\times}) =  
\mathcal{H}om(\hat{\omega}^1_{(\Sp(K[[x,y]]/(xy)), M)/\Sp(K)^{\times}}, K[[x,y]]/(xy)).$$ 
For an object $(E,\nabla_E)$ in 
$\widehat{MIC}((\Sp(K[[x,y]]/(xy)), M)/\Sp(K)^{\times})^{\mathrm{nr}}$, put 
$$\overline{E} := \{ e \in E \,|\, \exists N \in \mathbb{N}, \nabla_E^N(\partial_1)(e) = 0\}. $$ 
Then $\overline{E}$ is stable by the action of 
$\nabla_E(\partial_1)$ and this action is locally nilpotent. 
Then, it suffices to prove that the correspondence 
$(E,\nabla_E) \mapsto (\overline{E}, \nabla_E(\partial_1))$ defines the 
functor 
$$ \widehat{MIC}((\Sp(K[[x,y]]/(xy)), M)/\Sp(K)^{\times})^{\mathrm{nr}} \to 
\left\{ 
(V,N) \,\left|\, 
\begin{aligned} 
& \text{$V$: a finite-dimensional $K$-vector space} \\
& \text{$N: V \to V$: a nilpotent endomorphism}
\end{aligned}
\right. 
\right\} $$ 
and that it 
is an equivalence whose quasi-inverse is given by 
$(V,N) \mapsto (E := V \otimes_K K[[x,y]]/(xy), \nabla_E)$ 
with the action $\nabla_E(\partial_1)$ defined by 
$\nabla_E(\partial_1) :=  {\rm id} \otimes d(\partial_1) + N \otimes {\rm id}.$  
(The reason is the same as that in lemma \ref{RS}.) 

To prove the above claim, it suffices to construct a basis of $\overline{E}$ over 
$K$ which is a basis of $E$ over $K[[x,y]]/(xy)$. 
To do so, first we prove that there exists a basis of $E$ on which $\nabla_E(\partial_1)$ 
acts as a strictly upper triangular matrix with entries in $K$. If $E$ has rank $n,$ 
by hypothesis of nilpotent residues, we can write 
$$\nabla_E(\partial_1)= d(\partial_1) + H $$
with respect to some basis, where 
$H=(a_{i,j}(x,y))_{i,j}$ is an $n\times n$ matrix such that $H_0=(a_{i,j}(0,0))_{i,j}$ is a 
strictly upper triangular matrix. We want to prove that there exists a change of basis given by a matrix $U$ such that
\begin{equation}\label{cambiodibase} 
HU+d(\partial_1) U=UH_0.
\end{equation}
 We proceed as in \cite[lemma 3.2.8]{Ked07}. We  write $U$ as the sum $U=U_0+\sum_{i=1}^{\infty}U_{i}x^i+\sum_{j=1}^{\infty}U'_{j}y^j$ with $U_0$, $U_{i}$ and $U'_{j}$ matrices with entries in $K$. If $U_0$ is invertible, $U$ is invertible.  %it is a calculus of power series

If we write $H=H_0+\sum_{i=1}^{\infty}H_{i}x^i+\sum_{j=1}^{\infty}H'_{j}y^j$, then the equation \eqref{cambiodibase} is equivalent to the condition that for every $i>0$ 
\begin{equation}\label{cambiodibasex}
H_0U_{i}-U_{i}H_0+iU_{i}=-\sum_{k=0}^{i-1}H_{i-k}U_{k}
\end{equation}
and that for every $j> 0$ 
\begin{equation}\label{cambiodibasey}
H_0U'_{j}-U'_{j}H_0-jU'_{j}=-\sum_{k=0}^{j-1}H'_{j-k}U'_{k}.
\end{equation}
(Here we put $U'_0 := U_0$.) 
For $i=0,$ the equation (\ref{cambiodibase}) is nothing but $H_0U_0=U_0H_0$, hence we can choose $U_0$ to be the identity matrix. 

Next, note that the linear map $\phi:X\mapsto H_0X-XH_0$ is nilpotent, since so is $H_0$. Thus the linear map $X \mapsto H_0X-XH_0+iX$ (for $i\neq 0$) is the sum of an invertible map 
$\psi: X \mapsto iX$ and a nilpotent map $\phi$. Hence it is invertible, because 
$$\frac{\psi}{\psi+\phi}=\sum_{i=0}^{\infty}\left(- \frac{\phi}{\psi}\right)^i. $$
Hence we can construct $U_i$'s $(i > 0)$ and $U'_j$'s $(j > 0)$ uniquely so that \eqref{cambiodibasex} and 
\eqref{cambiodibasey} are satisfied. 

So we have a basis $e_1,\dots ,e_s$ of $E$ on which $\nabla_E(\partial_1)$ 
acts as a strictly upper triangular matrix with entries in $K$. 
We prove the equality $Ke_1+\cdots +Ke_s = \overline{E}$, which implies the claim we want. 
The inclusion $Ke_1+\cdots +Ke_s \subset \overline{E}$ follows from the definition of the basis. 
We prove the opposite inclusion. Let $e=b_1e_1+\dots +b_se_s$ be an element of $ E$ with $b_i\in K[[x,y]]/(xy)$ for $i=1, \dots ,s$ such that there exists an integer $N$ with $\nabla_E(\partial_1)^{N}(e)=0$. We have to prove that $b_1, \dots, b_s$ are in fact elements of $K$.  
Since $\nabla_E(\partial_1)$ acts as a strictly upper triangular matrix with respect to the basis $e_1, \dots, e_s,$ we can see in the expression of $\nabla_E(\partial_1)^{N}(e)$ as linear combination of $e_1, \dots, e_s$  that the coefficient of $e_s$ is $d(\partial_1)^{N}(b_s)$. Since $\nabla_E(\partial_1)^{N}(e)=0$,  $d(\partial_1)^{N}(b_s)=0.$ Writing  $b_s=\sum_{i=0}^{\infty}\alpha_i x^i+\sum_{j=1}^{\infty}\beta_jy^j $ with $\alpha_i, \beta_j$ $\in$ $K$, we see that, if $d(\partial_1)^N(b_s)=0$, then $b_s=\alpha_0 \in K$. %(recall that $\partial_1$ is indeed the derivation which seen as an element of $\mathcal{H}om(\omega^1_{(\Sp(K[[x,y]]/xy), M)/\Sp(K)^{\times}}, K[[x,y]]/xy)$ 
%sends $\mathrm{dlog}x$ to $1$ and so $\mathrm{dlog}y$ to $-1$).
  We look now at the coefficient of $e_{s-1}$ of  $\nabla^{N}(\partial_1)(e)$; using the strict upper triangularity again ($e_{s-1}$ is sent by $\nabla_E(\partial_1)$ to a $K$-linear combination of $e_1, \dots, e_{s-2}$) and the fact that  $b_s\in K,$ we can prove that $b_{s-1}\in K$. Hence we go on and we prove the inclusion $\overline{E} \subset 
Ke_1+\cdots +Ke_s$.
So the proof of the proposition is finished. 
\end{proof}
\begin{prop}\label{inductivepass} If the category of free modules with formal integrable connection having nilpotent residues on $(\Sp(K[[x_1,\dots, x_{n-1}]]/(x_1\cdots x_{n-1})), M)/\Sp(K)^{\times}$  is abelian, then the category of free modules with formal integrable connections having nilpotent residues on $(\Sp(K[[x_1,\dots, x_{n}]]/(x_1\cdots x_{n})), M)/\Sp(K)^{\times}$  is also abelian. 

\end{prop}
\begin{proof}
To simplify the notation we will denote by $B$ the ring $K[[x_1,\dots, x_{n}]]/(x_1\cdots x_{n})$ and by $A$ the ring $K[[x_1,\dots, x_{n-1}]]/(x_1\cdots x_{n-1})$. 
The map $h: A\rightarrow B$ defined by $x_i\mapsto x_i \, (i=1, \dots n-2)$, $x_{n-1}\mapsto x_{n-1}x_n$ induces a map of log schemes $(\Sp(B), M_B)\rightarrow (\Sp(A), M_A)$, where the log structure $M_A, M_B$ are the log structure 
$M$ in remark \ref{notationlocalderivations} for $\Sp(A), \Sp(B)$, respectively. (Note that the above map is not strict.) In the following, $A$ is regarded as a subring of $B$ via the map 
$h$. 

Let $\Gamma$ be the set $\{\g{k} = (k_1, ..., k_n) \in \mathbb{N}^n \,|\, 
\text{$k_l=0$ for some $1 \leq l \leq n$}\}$ and let $\Gamma'$ be the subset of 
$\Gamma$ consisting of $\g{k} = (k_1, ..., k_n)$ with $k_{n-1} = k_n$. 
Then any element in $B$ (resp.~$A$) is uniquely written as 
$\sum_{\g{k} \in \Gamma}b_{\g{k}} x^{\g{k}} \, (b_{\g{k}} \in K)$ (resp. 
$\sum_{\g{k} \in \Gamma'}b_{\g{k}} x^{\g{k}} \, (b_{\g{k}} \in K)$), 
where $x^{\g{k}} = x_1^{k_1} \cdots x_n^{k_n}$. 
Also, let $\partial_1, ..., \partial_{n-1} \in 
\mathcal{H}om(\hat{\omega}^1_{(\Sp(B),M_B)/\Sp(K)^{\times}}, \mathcal{O}_{\Sp(B)})$ be as in  
remark \ref{notationlocalderivations}, and put $\partial := \partial_{n-1}$. 

For an object $(E,\nabla_E)$ in 
$\widehat{MIC}((\Sp(B), M_B)/\Sp(K)^{\times})^{\mathrm{nr}}$, put 
$$\overline{E} := \{ e \in E \,|\, \nabla_E^N(\partial)(e) \to 0 \, (N \to \infty)\}, $$
where, on the right hand side, $E$ is endowed with $(x_1, ..., x_n)$-adic topology. 
Then $\overline{E}$ is stable by the action of 
$\nabla_E(\partial_i) \, (1 \leq i \leq n-2)$ and the action of $\nabla_E(\partial)$ which is 
locally topologically nilpotent.  
Then, it suffices to prove the following: firstly, $\overline{E}$ and the actions 
$\nabla_E(\partial_i) \, (1 \leq i \leq n-2)$ on it define an object (which we denote by 
$(\overline{E}, \nabla_{\overline{E}})$) in 
$\widehat{MIC}((\Sp(A), M_A)/\Sp(K)^{\times})^{\mathrm{nr}}$. Secondly, 
the correspondence 
$(E,\nabla_E) \mapsto ((\overline{E}, \nabla_{\overline{E}}), \nabla_E(\partial_1))$ defines the 
functor 
$$ \widehat{MIC}((\Sp(B), M_B)/\Sp(K)^{\times})^{\mathrm{nr}} \to 
\left\{ 
((F,\nabla_F),N) \,\left|\, 
\begin{aligned} 
& \text{$(F,\nabla_F) \in \widehat{MIC}((\Sp(A), M_A)/\Sp(K)^{\times})^{\mathrm{nr}}$} \\
& \text{$N: (F,\nabla_F) \to (F,\nabla_F)$: a nilpotent endomorphism}
\end{aligned}
\right. 
\right\} $$ 
and it 
is an equivalence whose quasi-inverse is given by 
$((F,\nabla_F),N) \mapsto (E := F \otimes_A B, \nabla_E)$ 
with the action $\nabla_E(\partial_i) \, (1 \leq i \leq n-2)$ and $\nabla_E(\partial)$ defined by 
$\nabla_E(\partial_i) :=  {\rm id} \otimes d(\partial_i) + \nabla_F(\partial_i) \otimes {\rm id},$ 
$\nabla_E(\partial) :=  {\rm id} \otimes d(\partial) + N \otimes {\rm id}.$ 
(The reason is the same as that in lemma \ref{RS}.) 

To prove the above claim, it suffices to construct a basis of $\overline{E}$ over 
$A$ which is a basis of $E$ over $B$. To do so, first we prove the following claim. \\

CLAIM 1: There exists a basis of $E$ as $B$-module on which $\nabla_E(\partial)$ acts as a matrix $M$ with entries in $A$. Moreover, if we write $M=\sum_{\g{k} \in \Gamma'}M_{\g{k}}\g{x}^{\g{k}}$ with entries of $M_{\g{k}}$ in $K,$ 
$M_0$ is strictly upper triangular. 

\begin{proof} 
Let us fix a basis $\g{e}=(e_1, \dots, e_s)$ of $E$ and let us denote by $H$ the matrix such that 
$$\nabla_E(\partial) = d(\partial) + H$$
with respect to the basis $\g{e}$. 
We write $H = \sum_{\g{k} \in \Gamma}H_{\g{k}} x^{\g{k}}$ with entries of $H_{\g{k}}$ in $K$. 
We may assume that $H_{0}$ is strictly upper triangular 
because we assumed that $(E, \nabla_E)$ has nilpotent residues. 
Hence it suffices to prove the existence of matrices 
$U = \sum_{\g{k} \in \Gamma}U_{\g{k}} x^{\g{k}}, X = \sum_{\g{k} \in \Gamma}X_{\g{k}} x^{\g{k}}$ 
satisfying the following conditions: \\

\noindent 
(a) \, $U_0$ is the identity matrix and $U_{\g{k}} = 0$ if $\g{k} \in \Gamma' \setminus \{0\}$. \\ 
(b) \, $X_0 = 0$ and $X_{\g{k}} = 0$ if $\g{k} \in \Gamma \setminus \Gamma'$. \\ 
(c) \, $HU + d(\partial)U = U(H_0 + X)$. \\

\noindent 
In fact, $U$ is invertible by the condition (a) and then the conditions 
(b) and (c) imply that, after some change of basis, 
the action of $\nabla_E(\partial)$ on the new basis is described by the matrix $M := H_0 + X$, whose entries belong to $A$ with its constant term $H_0$ strict upper triangular. 

The equality (c) is equivalent to the condition that, for any $\g{k} \in \Gamma$, 
\begin{equation}\label{eq:knkn-1}
H_0U_{\g{k}} - U_{\g{k}}H_0 + (k_{n-1} - k_{n}) U_{\g{k}} = 
- \sum_{\g{i}<\g{k}} H_{\g{k}-\g{i}}U_{\g{i}} + \sum_{\g{i} < \g{k}} U_{\g{i}} X_{\g{k}-\g{i}}. 
\end{equation}
(Here, for $\g{i} = (i_1, ..., i_n)$ and $\g{k} = (k_1, ..., k_n)$, 
we write $\g{i} < \g{k}$ if $i_j \leq k_j \, (1 \leq j \leq n)$ and $\g{i} \not= \g{k}$.) 
We construct (uniquely) the matrices $U_{\g{k}}, X_{\g{k}}$ by induction on $|\g{k}| = \sum_i k_i$
so that the conditions (a), (b) and the equality \eqref{eq:knkn-1} are satisfied. If $\g{k} = 0$, 
$U_0$ should be the identity matrix and $X_0$ should be zero by (a) and (b). 
In this case, the equality \eqref{eq:knkn-1} for $\g{k}=0$ is obviously satisfied. 
For general $\g{k}$ with $k_{n-1} = k_n$, $U_{\g{k}}$ should be zero by (a) and 
so the equality \eqref{eq:knkn-1} is written as 
$$ 0 = - \sum_{\g{i}<\g{k}} H_{\g{k}-\g{i}}U_{\g{i}} + \sum_{0 < \g{i} < \g{k}} U_{\g{i}} X_{\g{k}-\g{i}}
+ X_{\g{k}}. $$
Then we can take a unique matrix $X_{\g{k}}$ satisfying this equality. 
For general $\g{k}$ with $k_{n-1} \not= k_n$, $X_{\g{k}}$ should be $0$ by (b) 
and so the equality \eqref{eq:knkn-1} is written as 
$$ 
H_0U_{\g{k}} - U_{\g{k}}H_0 + (k_{n-1} - k_{n}) U_{\g{k}} = 
- \sum_{\g{i}<\g{k}} H_{\g{k}-\g{i}}U_{\g{i}}. 
$$
We can take a unique matrix $U_{\g{k}}$ satisfying this equality 
because the map $Y \mapsto H_0Y - YH_0 + (k_{n-1} - k_{n}) Y$, being the sum of 
a nilpotent map $Y \mapsto H_0Y - YH_0$ and an invertible map 
$Y \mapsto (k_{n-1} - k_{n}) Y$, is invertible. So we proved the claim. 
\end{proof}

Let $e_1, \dots, e_s$ be a $B$-basis of $E$ which satisfies the condition in the statement of 
CLAIM 1. We prove the following claim, which implies the claim we want: \\

CLAIM 2:\,\, $Ae_1 + \cdots + Ae_s = \overline{E}$. 

\begin{proof}
Let us prove first the inclusion $Ae_1+\dots+Ae_s \subset \overline{E}$. 
For an element $e := a_1 e_1+\dots+ a_s e_s$ with $a_i\in A$, we calculate $\nabla_E(\partial)^N(a_1 e_1+\dots+a_s e_s)=a_1\nabla_E(\partial)^N(e_1)+\dots+ a_s\nabla_E(\partial)^N(e_s).$ 
Let us denote by $v$ the $(x_1, \dots, x_n)$-adic valuation on $B$. 
Then this induces a valuation $v$ on $E$ defined by $v(\sum_{i=1}^sb_ie_i):=\textrm{min}_iv(b_i).$ 
To prove that $e$ is in $\overline{E}$, 
we need to prove that for every $m\in \mathbb{N}$ there exists an $N\in \mathbb{N}$ such that  $v(\nabla_E(\partial)^N(e))\geq m.$ 
Since $v$ is non-archimedean $v(\nabla_E(\partial)^N(e))\geq \textrm{min}_iv(a_i\nabla_E(\partial)^N(e_i))= \textrm{min}_i\{v(a_i)+v(\nabla_E(\partial)^N(e_i))\},$ hence it is enough to prove that for every $m\in \mathbb{N}$ there exists an $N\in \mathbb{N}$ such that $v(\nabla_E(\partial)^N(e_i))\geq m$. The action of 
$\nabla_E(\partial)$ on $Ae_1 + \cdots + Ae_s$ is written by some matrix $\sum_{\g{k}}M_{\g{k}}\g{x}^{k}$ with $M_0$ nilpotent. Hence each entry of the matrix $\nabla_E(\partial)^{s}$ has strictly positive $(x_1, \dots, x_n)$-adic valuation. 
Then we see that 
$v(\nabla_E(\partial)^{ms}e_i)\geq m$. Hence, if we choose $N\geq ms$, we are done.

Next we prove the inclusion $\overline{E} \subset Ae_1 + \cdots + Ae_s$. 
Let $e=a_1e_1+\dots +a_se_s$ $\in E$ with $a_j\in B$ for every $j$, such that 
for all $m\in \mathbb{N}$ there exists $N\in \mathbb{N}$ such that $\nabla_E(\partial)^N(e)=b_1e_1+\dots +b_se_s$ with  $b_j=\sum_{\g{k}, |\g{k}|\geq m}\beta_{\g{k}}\g{x}^{\g{k}}$, namely, $v(\nabla_E(\partial)^N(e)) \geq m$. 
We want to prove that $a_j\in A$ for every $j$. Let us calculate 
\begin{equation}\label{Npower}
\nabla_E(\partial)^N(a_1e_1+\dots +a_se_s)=\sum_{l=0}^{N}{N\choose l}(d(\partial)^{N-l}(a_1)\nabla_E(\partial)^l(e_1)+\dots +d(\partial)^{N-l}(a_s)\nabla_E(\partial)^l(e_s)).
\end{equation}
Let us fix $i$ and write $a_i=\sum_{\g{k}}\alpha_{\g{k}}\g{x}^{\g{k}}$; we are interested in how $d(\partial)$ acts on $a_i$. It acts on monomials of the form  $\g{x}^{\g{k}}=x_1^{k_1}\cdots x_{n-1}^{k_{n-1}}x_n^{k_n}$ by $d(\partial)(\g{x}^{\g{k}})=(k_{n-1}-k_{n})x^{\g{k}}$. We define a function $\delta:B\rightarrow \mathbb{N}\cup \{\infty\}$ as follows; if  $a_i=\sum_{\g{k}}\alpha_{\g{k}}\g{x}^{\g{k}}\in B$, then $\delta(a_i):=\textrm{min}\{|\g{k}| \,|\,\alpha_{\g{k}}\neq 0 , k_{n-1}-k_{n}\neq 0\}$. (We define that $\delta(a_i) = \infty$ if the set on the right hand side is 
empty.) Then $\delta(a_i)=\infty$ if and only if $a_i\in A$. 

Using the function $\delta$ we prove the claim that, for every $l\geq 1$, the element $d(\partial)^l(a_i)$ has $(x_1, \dots, x_n)$-adic valuation equal to $\delta(a_i)$. 
If $a_i \in A,$ the claim is true because $d(\partial)^l (a_i)=0$ and $\delta(a_i) = \infty$. 
% hence we have to prove that $0$ has $(x_1, \dots, x_n)$-adic valuation equal $\infty$ which is true. 
If $a_i\in B$ but not in $A$, then there exists $\overline{\g{k}} = (\overline{k}_1, \dots, \overline{k}_n)$ such that 
$\overline{k}_{n-1} \not= \overline{k}_n$, 
$|\overline{\g{k}}|=\delta(a_i)<\infty$. We can write $a_i$ as 
$$a_i=\sum_{\g{k}, |\g{k}|<|\overline{\g{k}}|}\alpha_{\g{k}}\g{x}^{\g{k}}+\alpha_{\overline{\g{k}}}\g{x}^{\overline{\g{k}}}+\sum_{\g{k}\neq \overline{\g{k}}, |\g{k}|\geq|\overline{\g{k}}|, }\alpha_{\g{k}}\g{x}^{\g{k}}.$$
By definition of $\overline{\g{k}}$, we have the equality 
$$d(\partial)^{l}(a_i)=\alpha_{\overline{\g{k}}}\g{x}^{\overline{\g{k}}}(\overline{k}_{n-1}-\overline{k}_{n})^{l}+
\sum_{\g{k}\neq \overline{\g{k}}, |\g{k}|\geq|\overline{\g{k}}|, }\alpha_{\g{k}}\g{x}^{\g{k}}
(k_{n-1}-k_{n})^l. $$ 
%\textrm{terms\, of\, bigger \,}(x_1, \dots x_n)-\textrm{adic valuation},$$
So we conclude that $d(\partial)^l(a_i)$ has $(x_1, \dots, x_n)$-adic valuation equal to $\delta(a_i)$, as we wanted.

We come back to $e=a_1e_1+\dots +a_se_s$. We know by hypothesis that for $m\in \mathbb{N}$, there exists $N$ such that  $v(\nabla_E(\partial)^N(e))\geq m$. We have to show that $a_i$ is in $A$ for $i=1, \dots, s$. 
Let us suppose by absurd that there exists some $j$ such that $\delta(a_j)\neq \infty$. We can suppose that $\delta(a_j)\leq \delta(a_i)$ for every $i=1,\dots ,s$ and that $\delta(a_j)<\delta(a_i)$ for all $i>j$. We look at the $e_j$-component of $\nabla_E(\partial)^N(e)$ for $N \geq 1$, using the expression of \eqref{Npower}. Thanks to what we have shown $v(d(\partial)^N(a_j)e_j)=\delta(a_j)$;  for $1 \leq l \leq N-1$, the $e_j$-component of $d(\partial)^{N-l}(a_i)\nabla_E(\partial)^l(e_i)$, denoted by $[d(\partial)^{N-l}(a_i)\nabla_E(\partial)^l(e_i)]_j$, has the following $(x_1, \dots, x_n)$-adic valuation for $i>j$ :
$$v([d(\partial)^{N-l}(a_i)\nabla_E(\partial)^l(e_i)]_j)\geq \delta(a_i)>\delta(a_j),$$ 
while for $i\leq j$
$$v([d(\partial)^{N-l}(a_i)\nabla_E(\partial)^l(e_i)]_j)\geq \delta(a_i)+1>\delta(a_j)$$
 because $e_i$ is sent via the matrix $M_0$ to a $K$-linear combination of $e_1, \dots, e_{i-1}$ due to the assumption that $M_0$ is strictly upper triangular. Also, 
the $e_j$-component $[a_i\nabla_E(\partial)^N(e_i)]_j$ of $a_i\nabla_E(\partial)^N(e_i)$ 
has the following $(x_1, \dots, x_n)$-adic valuation, by the argument in the first paragraph of the proof: 
$$ v([a_i\nabla_E(\partial)^N(e_i)]_j) \geq \lfloor N/s \rfloor. $$
Therefore, if $N \geq s(\delta(a_j)+1)$, the $(x_1, \dots, x_n)$-adic valuation of the $e_j$-component 
$[\nabla_E(\partial)^{N}(a)]_j$ 
of $\nabla_E(\partial)^{N}(a) $ is given as 
$$v([\nabla_E(\partial)^{N}(a)]_j)=\delta(a_j),$$
and this does not go to infinity as $N \to \infty$. This is a contradiction. 
Hence $a_i\in A$ for all $i$, as we wanted.
\end{proof}

Since CLAIM 2 is proved, the proof of the proposition is finished. 
\end{proof}

\begin{prop}\label{relativeTannakian}
Let us suppose that there exists a $K$-rational point $x\in X$. Then  $MIC(X^{\times}/\Sp(K)^{\times})^{\mathrm{nr}}$ is a neutral Tannakian category over $K$.
\end{prop}
\begin{proof}
As we saw in proposition \ref{relativeabelian} the category $MIC(X^{\times}/\Sp(K)^{\times})^{\mathrm{nr}}$ is abelian; moreover with the tensor structure defined in $MIC(X^{\times}/\Sp(K)^{\times})$ and unit object given by $(\mathcal{O}_X, d)$ (see remark \ref{tensorstructure}) it is a rigid abelian tensor category, thanks to the fact that, for every object $(E, \nabla_E)$ in $MIC(X^{\times}/\Sp(K)^{\times})^{\mathrm{nr}}$, $E$ is a locally free $\mathcal{O}_X$-module. 

Taking the fiber at $x$ gives a map $\mathrm{End}((\mathcal{O}_X,d))\rightarrow K$, which is injective because, thanks to \cite[Proposition 3.1.6]{Shi00}, we know that $\mathrm{End}((\mathcal{O}_X,d))\cong H^0_{\mathrm{dR}}(X^{\times}/\Sp(K)^{\times})$ is a field. But since $K\subset \mathrm{End}((\mathcal{O}_X,d))$, we have indeed that $\mathrm{End}((\mathcal{O}_X,d))\cong K$.

We define a functor 
$$\omega_x:MIC(X^{\times}/\Sp(K)^{\times})^{\mathrm{nr}}\rightarrow \mathrm{Vec}_K$$ 
which sends every pair $(E, \nabla_E)$ to the $K$-vector space given by the fiber $E(x)$ of $E$ at $x$. Since $E$ is locally free for every object $(E, \nabla_E)$ in  $MIC(X^{\times}/\Sp(K)^{\times})^{\mathrm{nr}}$, $\omega_x$ is an exact tensor functor and by \cite[Corollaire 2.10]{Del90} $\omega_x$ is faithful. So it is a fiber functor. 
\end{proof}

\begin{defi} \label{pi1XM/KN}
Let $x$ be a $K$-rational point of $X$ and let 
$$\omega_x:MIC(X^{\times}/\Sp(K)^{\times})^{\mathrm{nr}}\rightarrow \mathrm{Vec}_K$$ 
be the fiber functor introduced in the above proposition. 
We define the log algebraic fundamental group of $X^{\times}$ over $\Sp(K)^{\times}$ with base point $x$ as the Tannaka dual of $MIC(X^{\times}/\Sp(K)^{\times})^{\mathrm{nr}}$, \emph{i.e.} 
$$\pi_{1}(X^{\times}/\Sp(K)^{\times}, x):=G(MIC(X^{\times}/\Sp(K)^{\times})^{\mathrm{nr}}, \omega_x)$$
with the same notation as in theorem \ref{Saavedra}.
\end{defi}

Next we prove that $MIC(X^{\times}/\Sp(K))^{\mathrm{nr}}$ is also a neutral Tannakian category over $K$ under some suitable assumptions.
\begin{prop}\label{globalfreenesstrivbasis}
The category  $MIC(X^{\times}/\Sp(K))^{\mathrm{nr}}$ is abelian. 
\end{prop}
\begin{proof} 
Note first that, by restricting the derivations, a functor
$$r:MIC(X^{\times}/\Sp(K))^{\mathrm{nr}}\rightarrow MIC(X^{\times}/\Sp(K)^{\times})^{\mathrm{nr}}$$
is defined. To see this, it suffices to check that, for an object $(E, \nabla_E)$ in  $MIC(X^{\times}/\Sp(K))^{\mathrm{nr}}$, the module with integrable connection $r((E, \nabla_E)):=(E, \nabla_{E|\Sp(K)^{\times}})$ has nilpotent residues. 
This is true because, for every geometric point $y$ over a closed point of $X$ and 
every $K(y)$-linear map $t_{y}:K(y)\otimes \overline{M}^{\mathrm{gp}}_{y}/\overline{N}^{\mathrm{gp}}_y\rightarrow K(y)$, 
the composite map 
$$ E(y) \overset{\rho_{y}}{\longrightarrow}  E(y)\otimes \overline{M}^{\mathrm{gp}}_{y}/\overline{N}^{\mathrm{gp}}_y
\overset{{\rm id} \otimes t_y}{\longrightarrow} E(y) $$ 
(where the notation is as in definition \ref{residuenilpMN} for $r((E, \nabla_E))$) 
coincides with the nilpotent map  
$$ E(y) \overset{\rho'_{y}}{\longrightarrow}  E(y)\otimes \overline{M}^{\mathrm{gp}}_{y} 
\overset{{\rm id} \otimes t'_y}{\longrightarrow} E(y), $$ 
where $\rho'_y$ is the map induced by $(E, \nabla_E)$ and 
$t'_y$ is the composite of the projection 
$K(y) \otimes {M}^{\mathrm{gp}}_{y} \longrightarrow K(y) \otimes 
\overline{M}^{\mathrm{gp}}_{y}/\overline{N}^{\mathrm{gp}}_y$ and $t_y$. 

For any morphism $\varphi: E \to F$ in $MIC(X^{\times}/\Sp(K))^{\mathrm{nr}}$, 
the underlying $\mathcal{O}_X$-module of ${\rm Ker}(\varphi)$
(resp. ${\rm Coker}(\varphi)$) is the same as that of ${\rm Ker}(r(\varphi))$
(resp. ${\rm Coker}(r(\varphi))$), and the latter is locally free by 
proposition \ref{relativeabelian}. Also, since the same is true 
for the image and the coimage of $\varphi$, we have the natural injection 
${\rm Ker}(r(\varphi))(y) \hookrightarrow E(y)$ and the natural surjection 
$F(y) \twoheadrightarrow {\rm Coker}(\varphi)(y)$ of fibers at any 
geometric point $y$ of $X$. Since these maps are compatible with residues, 
we conclude that ${\rm Ker}(\varphi)$ and ${\rm Coker}(\varphi)$ have 
nilpotent residues. So the proof is finished. 
\begin{comment}
***************

We proceed as in proposition \ref{relativeabelian}. Since $MIC(X^{\times}/\Sp(K))^{\mathrm{nr}}$ is a full subcategory of the abelian category $MIC(X^{\times}/\Sp(K))$, it is sufficient to prove that the kernel and the cokernel of any morphism between objects in $MIC(X^{\times}/\Sp(K))^{\mathrm{nr}}$ are locally free $\mathcal{O}_X$-modules with nilpotent residues. As in propositions \ref{relativeabelian} 
and \ref{RS}, 
it will be enough to prove this at the level of completed local rings $\widehat{\mathcal{O}}_{X,x}$ and we may assume that 
$\widehat{\mathcal{O}}_{X, x} \cong K[[x_1, \dots, x_n]]/(x_1\cdots x_n)$ with $n \geq 2$. 
Then we proceed by induction on $n$: in proposition \ref{dimension1abs} below we prove the result for $n=2$, and in proposition \ref{inductivepassabs} below we prove
the induction step. 
% that the $n$-dimensional case can be deduced from the $(n-1)$-dimensional case. 
%We remark that by, a similar argument to that in lemma \ref{RS}, we can suppose that $\widehat{
%\mathcal{O}}_{X,x}$ is isomorphic to $K[[x_1, \dots, x_n]]/(x_1\cdots x_n)$.
\end{comment}
\end{proof}

\begin{prop}
Suppose that there exists a $K$-rational point $x$ of $X$. Then the category $MIC(X^{\times}/\Sp(K))^{\mathrm{nr}}$ is a neutral Tannakian category over $K$.
\end{prop}
\begin{proof}
The proof is analogous to that of proposition \ref{relativeTannakian}. We saw in proposition \ref{globalfreenesstrivbasis} that $MIC(X^{\times}/\Sp(K))^{\mathrm{nr}}$ is an abelian category. Moreover, with the tensor structure defined in $MIC(X^{\times}/\Sp(K))$ and the unit object given by $(\mathcal{O}_X, d)$ (see remark \ref{tensorstructure}) it is a rigid abelian tensor category, thanks to the fact that, for every object $(E, \nabla_E)$ in $MIC(X^{\times}/\Sp(K))^{\mathrm{nr}}$, $E$ is a locally free $\mathcal{O}_X$-module. 

To prove that it is a neutral Tannakian category over $K$ we have to check that $\mathrm{End }((\mathcal{O}_X, d))\cong K$ and to construct a fiber functor.  Every $\alpha \in \mathrm{End }((\mathcal{O}_X, d))$ is indeed an element of $\mathrm{End }_{MIC(X^{\times}/\Sp(K)^{\times})}((\mathcal{O}_X, d_{|\Sp(K)^{\times}}))$ which is isomorphic to $K$ as we saw in proposition \ref{relativeTannakian}. Since $K\subset \mathrm{End }((\mathcal{O}_X, d)),$ $\mathrm{End}((\mathcal{O}_X, d))\cong K$ as required. 
We define a functor
 $$\eta_x:MIC(X^{\times}/\Sp(K))^{\mathrm{nr}}\rightarrow \mathrm{Vec}_K$$
 which sends a pair $(E, \nabla_E)\in MIC(X^{\times}/\Sp(K))^{\mathrm{nr}}$ to the fiber $E(x)$ of $E$ at $x$.
 Since the underlying $\mathcal{O}_X$-module $E$ is locally free for every object 
$(E, \nabla_E)$ in $MIC(X^{\times}/\Sp(K))^{\mathrm{nr}}$, $\eta_x$ is an exact tensor functor and by \cite[Corollaire 2.10]{Del90} it is faithful. Hence it is a fiber functor. 
\end{proof}

\begin{defi} \label{pi1XM/K}
Let $x$ be a $K$-rational point of $X$ and let 
$$\eta_x:MIC(X^{\times}/\Sp(K))^{\mathrm{nr}}\rightarrow \mathrm{Vec}_K$$
be the fiber functor introduced in the above proposition. 
We define the log algebraic fundamental group of $X^{\times}$ over $\Sp(K)$ with base point $x$ as the Tannaka dual of $MIC(X^{\times}/\Sp(K))^{\mathrm{nr}}$, \emph{i.e.} 
$$\pi_{1}(X^{\times}/\Sp(K), x):=G(MIC(X^{\times}/\Sp(K))^{\mathrm{nr}}, \eta_x)$$
with the same notation as in theorem \ref{Saavedra}.
\end{defi}

\section{The exact sequence with four terms}

In this section we suppose that there exists a $K$-rational point  $x$ of $X$. 

The morphism of log schemes $f:X^{\times}\rightarrow \Sp(K)^{\times}$ induces a functor 
$$f^*_{\mathrm{dR}}:MIC(\Sp(K)^{\times}/\Sp(K))^{\mathrm{nr}}\rightarrow MIC(X^{\times}/\Sp(K))^{\mathrm{nr}}:$$
indeed, if $(E, \nabla_E)$ is an object of $MIC(\Sp(K)^{\times}/\Sp(K))^{\mathrm{nr}},$ the module with integrable connection $f^*_{\mathrm{dR}}((E, \nabla_E))$ is defined as the pair $(f^*(E),f^*(\nabla_E))$, where $f^*(E)$ is the $\mathcal{O}_X$-module 
$\mathcal{O}_X \otimes_{f^{-1}K} f^{-1}(E)$ and $f^*(\nabla_E)$ is the unique connection which extends 
$f^{-1}(E) \xrightarrow{f^{-1}(\nabla_E)} f^{-1}(E) \otimes_{f^{-1}K} f^{-1}\omega^1_{\Sp(K)^{\times}/\Sp(K)} 
\rightarrow  f^*(E) \otimes_{\mathcal{O}_X} \omega^1_{X^{\times}/\Sp(K)}.$ 
The module with integrable connection $f^*_{\mathrm{dR}}((E, \nabla_E))$ has nilpotent residues: indeed, 
for every geometric point $y$ over a closed point of $X$ and every $K(y)$-linear map $t_{y}:K(y)\otimes \overline{M}^{\mathrm{gp}}_{y}\rightarrow K(y)$, the composite map 
$$ f^*(E)(y) \overset{\rho_y}{\longrightarrow} 
f^*(E)(y)\otimes \overline{M}^{\mathrm{gp}}_{y} 
\overset{{\rm id} \otimes t_y}{\longrightarrow} f^*(E)(y) $$
(where the notation is as in definition \ref{residuenilpMN}) 
is nilpotent because it coincides with the nilpotent map 
$\nabla_E:E\rightarrow E\otimes \overline{N}^{\mathrm{gp}} \cong E$ 
up to scalar after we tensor it with $K(y)$. 

Also, as we have seen in the proof of proposition \ref{globalfreenesstrivbasis}, 
the restriction functor
$$r:MIC(X^{\times}/\Sp(K))^{\mathrm{nr}}\rightarrow MIC(X^{\times}/\Sp(K)^{\times})^{\mathrm{nr}}$$
is defined. 

Let $\pi_1(\Sp(K)^{\times}/\Sp(K), \nu)$ be as in definition \ref{pi1KN/K}, let $\pi_1(X^{\times}/\Sp(K)^{\times}, x)$ be as in definition \ref{pi1XM/KN} and let $\pi_1(X^{\times}/\Sp(K)^{\times}, x)$ be as in definition \ref{pi1XM/K}. Then the functors $f^*_{\mathrm{dR}}$ and $r$ induce respectively the homomorphisms $\pi(f^*_{\mathrm{dR}})$ and $\pi(r)$ in the sequence
\begin{equation}\label{shortwithoutu}
\pi_1(X^{\times}/\Sp(K)^{\times}, x)\xrightarrow{\pi(r)} \pi_1(X^{\times}/\Sp(K), x)\xrightarrow{\pi(f^*_{\mathrm{dR}})} \pi_1(\Sp(K)^{\times}/\Sp(K), \nu).
\end{equation}
In the remaining part of the section we prove that the sequence \eqref{shortwithoutu} is exact and that $\pi(f^*_{\mathrm{dR}})$ is faithfully flat, using a theorem by H. Esnault, P. H. Hai and X. Sun (\cite[Theorem A1]{EsnHaiSun08}) which translates the exactness of the sequence of group schemes in terms of the corresponding categories of representations. We recall the theorem here for the convenience of the reader.
\begin{teo}\label{criterion}
Let $L\xrightarrow{\pi(q)} G\xrightarrow{\pi(p)} A$ be a sequence of homomorphisms of affine group schemes over a field $K$. It induces a sequence of functors
$$\mathrm{Rep}(A)\xrightarrow{p}\mathrm{Rep}(G)\xrightarrow{q}\mathrm{Rep}(L),$$
where $\mathrm{Rep}$ is the category of finite dimensional representations over $K$.

\begin{itemize}
\item[(i)] The map $\pi(p)$ is faithfully flat if and only if $p$ is fully faithful and the subcategory of $\mathrm{Rep}(G)$ given by $p(\mathrm{Rep}(A))$ is closed by subobjects. \item[(ii)] The map $\pi(q)$ is a closed immersion if and only if any object of $\mathrm{Rep}(L)$ is a subquotient of an object of the form $q(V)$ for $V\in \mathrm{Rep}(G).$

\item[(iii)] Let us assume that $\pi(p)$ is faithfully flat and that $\pi(q)$ is a closed immersion.  The sequence $L\xrightarrow{\pi(q)} G\xrightarrow{\pi(p)} A$ is exact if and only if the following conditions are satisfied: 
\begin{itemize}
\item[(a)] For every $E\in \mathrm{Rep}(G)$, $q(E)$ in $\mathrm{Rep}(L)$ is trivial if and only if there exists $V$ in $\mathrm{Rep}(A)$ such that $p(V)\cong E$.
\item[(b)] Let $E$ be an object in $\mathrm{Rep}(G)$ and let $W_0$ be the maximal trivial subobject of $q(E)$ in $\mathrm{Rep}(L)$. Then there exists $F\subset E$  in $\mathrm{Rep}(G)$ such that $q(F)\cong W_0$.
\item[(c)] Any $W$ in  $\mathrm{Rep}(L)$ is a quotient of (hence, by taking duals, embeddable in) $q(E)$ for some $E\in \mathrm{Rep}(G).$
\end{itemize} 
\end{itemize}
\end{teo}

\begin{prop}\label{derhamcohomologywithconnection}
Let $(E, \nabla_E)$ be an object in $MIC(X^{\times}/\Sp(K))^{\mathrm{nr}}$. Then  $H^0_{\mathrm{dR}}(X^{\times}/\Sp(K)^{\times}, r((E, \nabla_E)))$ comes equipped with a nilpotent endomorphism $\nabla$, in such a way that $(H^0_{\mathrm{dR}}(X^{\times}/\Sp(K)^{\times}, r((E, \nabla_E))), \nabla)$ is an object of $MIC(\Sp(K)^{\times}/\Sp(K))^{\mathrm{nr}}$. 

\end{prop}
\begin{proof}

 We need to prove that $H^0_{\mathrm{dR}}(X^{\times}/\Sp(K)^{\times}, r((E, \nabla_E)))$ is a finite dimensional $K$-vector space endowed with a nilpotent endomorphism. If we denote by $E^{\nabla_{E|\Sp(K)^{\times}}}$ the sheaf of horizontal sections of $r((E, \nabla_E))$, then by definition $H^0_{\mathrm{dR}}(X^{\times}/\Sp(K)^{\times}, r((E, \nabla_E)))=f_{*}(E^{\nabla_{E|\Sp(K)^{\times}}}).$ The sheaf $E^{\nabla_{E|\Sp(K)^{\times}}}$ is equipped with an action of $\mathcal{D}er(\Sp(K)^{\times}/K)$: indeed by \cite[proposition 3.12]{Kato89} we have an exact sequence of sheaves 
\begin{equation}\label{exactder}
0\rightarrow \mathcal{D}er(X^{\times}/\Sp(K)^{\times}) \rightarrow \mathcal{D}er(X^{\times}/\Sp(K))\rightarrow \mathcal{O}_X\otimes\mathcal{D}er(\Sp(K)^{\times}/\Sp(K))\rightarrow0
\end{equation}
which locally splits. We denote by $\tau$ a splitting $\mathcal{O}_X\otimes\mathcal{D}er(\Sp(K)^{\times}/\Sp(K)) \rightarrow \mathcal{D}er(X^{\times}/\Sp(K))$. For every $D\in \mathcal{D}er(\Sp(K)^{\times}/\Sp(K)) $ we define the action of $\nabla(D)$ on $E^{\nabla_{E|\Sp(K)^{\times}}}$ by 
$$\nabla(D)(e):=\nabla_E(\tau(1 \otimes D))(e) \quad (e\in E^{\nabla_{E|\Sp(K)^{\times}}}).$$
The action is well-defined and it does not depend on the splitting because two different splittings differ by an element of $\mathcal{D}er(X^{\times}/\Sp(K)^{\times})$. Thus we have an action $\nabla$ of $\mathcal{D}er((\Sp(K)^{\times})/\Sp(K))$ on $E^{\nabla_{E|\Sp(K)^{\times}}}$ and hence on $f_{*}(E^{\nabla_{E|\Sp(K)^{\times}}})$, which defines 
the structure of a connection (see \cite[remark 3.1]{Kat70} and \cite[lemma 2.1]{Hai13} for the analogous action in the smooth case).  

We denote by $f_{*\mathrm{dR}}((E, \nabla_E))$ the $K$-vector space $f_{*}(E^{\nabla_{E|\Sp(K)^{\times}}})$ endowed with the action $\nabla$ described above. Then there is a morphism of $\mathcal{O}_X$-modules with integrable connection
\begin{equation}\label{f^*f_*}
f^*_{\mathrm{dR}}(f_{*\mathrm{dR}}((E, \nabla_E)))\longrightarrow (E, \nabla_E),
\end{equation}
because the natural map between the sheaves $f^*(f_{*}(E^{\nabla_{E|\Sp(K)^{\times}}}))\longrightarrow E^{\nabla_{E|\Sp(K)^{\times}}}$ composed with the inclusion $E^{\nabla_{E|\Sp(K)^{\times}}}\hookrightarrow E $ is indeed horizontal. 
(Here $f^*_{\mathrm{dR}}$ denotes 
the pullback of quasi-coherent modules with integrable connection. 
Note that we have not yet proved that $f_{*\mathrm{dR}}((E, \nabla_E))$  
belongs to the category $MIC(\Sp(K)^{\times}/\Sp(K))^{\mathrm{nr}}$, which is the purpose 
of this proposition.) 
To prove that $f_{*\mathrm{dR}}((E, \nabla_E))$ belongs to 
$MIC(\Sp(K)^{\times}/\Sp(K))^{\mathrm{nr}}$, it suffices to prove that the morphism 
\eqref{f^*f_*} is injective and that it has nilpotent residues. 
Let us denote the neutral Tannakian category 
$MIC(X^{\times}/\Sp(K)^{\times})^{\rm nr}$ by ${\cal T}$ and denote the unit object 
$({\cal O}_X,d)$ of ${\cal T}$ simply by ${\bold 1}$. Also, put $V := r((E,\nabla_E))$. 
Then the morphism we obtain by applying $r$ to \eqref{f^*f_*} is the canonical map 
\begin{equation}\label{tannakamap}
{\rm Hom}_{\cal T}({\bold 1}, V) \otimes_K {\bold 1} \to V, 
\end{equation}
and this map is injective in any neutral Tannakian category. 
So we have shown the required injectivity. 
Also, since the cokernel of the map 
\eqref{tannakamap} is also locally free, 
if we pull back the injection \eqref{f^*f_*} to a smooth closed point $s$ of $X$, 
we obtain an injection 
$$ 
(f_{*\mathrm{dR}}((E, \nabla_E))) \otimes_K K(s) 
\longrightarrow s^*(E, \nabla_E), $$
and by functoriality of the nilpotence of residues, we see that 
$f_{*\mathrm{dR}}((E, \nabla_E))$ has nilpotent residues. So we are done.  
\begin{comment}
Also, since the cokernel of the map 
\eqref{tannakamap} is also locally free, the nilpotence of residues of 
$f_{*\mathrm{dR}}((E, \nabla_E))$, which is a consequence of that of $f_{\rm dR}^*f_{*\mathrm{dR}}((E, \nabla_E))$, follows from the 
nilpotence of residues of $(E,\nabla_E)$. 
So we are done. 
\end{comment}
\end{proof}
\begin{defi}
Let $(E, \nabla_E)$ be an object in $MIC(X^{\times}/\Sp(K))^{\mathrm{nr}}$. The object  $$(H^0_{\mathrm{dR}}(X^{\times}/\Sp(K)^{\times}, r((E, \nabla_E))), \nabla)$$ of $MIC(\Sp(K)^{\times}/\Sp(K))^{\mathrm{nr}}$ will be denoted by $f_{*\mathrm{dR}}((E, \nabla_E))$.

\end{defi}
We denote by $H$ the group scheme $\pi(r)(\pi_1(X^{\times}/\Sp(K)^{\times}, x))$.

\begin{prop}\label{descriptionH}
The category ${\rm Rep}(H)$ can be described as the full subcategory of $MIC(X^{\times}/\Sp(K)^{\times})^{\mathrm{nr}}$ whose objects are subquotients of objects of the type $r((E,\nabla_E))$ with $(E,\nabla_E)\in MIC(X^{\times}/\Sp(K))^{\mathrm{nr}}$.
\end{prop}

\begin{proof}
The proof is analogous to \cite[proposition 3.1]{EsnHai06}. Let $C$ be the full subcategory of $MIC(X^{\times}/\Sp(K)^{\times})^{\mathrm{nr}}$ whose objects are subquotients of objects of the form $r((E,\nabla_E))$ with $(E,\nabla_E)\in MIC(X^{\times}/\Sp(K))^{\mathrm{nr}}$. Let $i$ be the inclusion functor 
$C \rightarrow MIC(X^{\times}/\Sp(K)^{\times})^{\mathrm{nr}}$ and 
let $G(C, x)$ be the Tannaka dual of $C$ as in theorem \ref{Saavedra}. 
The morphism $\pi(r):\pi_1(X^{\times}/\Sp(K)^{\times}, x)\rightarrow \pi_1(X^{\times}/\Sp(K), x)$ factors through $G(C, x)$ with a 
faithfully flat morphism $\pi(i): \pi_1(X^{\times}/\Sp(K)^{\times}, x) 
\rightarrow G(C,x)$ followed by a closed immersion $G(C,x) \rightarrow 
\pi_1(X^{\times}/\Sp(K), x)$: 
indeed, by theorem \ref{criterion} (i), $\pi(i)$ is faithfully flat because $i$ is clearly fully faithful and $i(C)$ is closed by subobjects. Also, the map $G(C, x)\rightarrow \pi_1(X^{\times}/\Sp(K), x)^{\mathrm{nr}}$ is a closed immersion because by definition every object of $C$ is a subquotient of an object of $r(MIC(X^{\times}/\Sp(K))^{\mathrm{nr}})$, hence the condition in theorem \ref{criterion} (ii) is satisfied.
\end{proof}

\begin{prop}\label{pifsurjective}
The map of fundamental groups 
$$\pi(f^*_{\mathrm{dR}}): \pi_1(X^{\times}/\Sp(K), x)\longrightarrow \pi_1(\Sp(K)^{\times}/\Sp(K), x)$$
is faithfully flat.
\end{prop}
\begin{proof}
By theorem \ref{criterion} (i), to prove the proposition we are asked to prove that $f^*_{\mathrm{dR}}$ is fully faithful and that $f^*_{\mathrm{dR}}(MIC(\Sp(K)^{\times}/\Sp(K))^{\mathrm{nr}})$ is closed by taking subobjects. 
Let $(V, \nabla_V)\in MIC(\Sp(K)^{\times}/\Sp(K))^{\mathrm{nr}}$. To prove that $f^*_{\mathrm{dR}}$ is fully faithful, it suffices to prove that $f^*_{\mathrm{dR}}$ induces an isomorphism
\begin{equation}\label{eq:ffisom}
H^0_{\mathrm{dR}}(\Sp(K)^{\times}/\Sp(K), (V, \nabla_V))\cong H^0_{\mathrm{dR}}(X^{\times}/\Sp(K), f^*_{\mathrm{dR}}((V, \nabla_V))).
\end{equation}
\begin{comment}
We can prove that $f_{*\mathrm{dR}}(f^{*}_{\mathrm{dR}}((V, \nabla_V)))\cong (V, \nabla_V)$; indeed,  $f_{*\mathrm{dR}}(f^{*}_{\mathrm{dR}}((V, \nabla_V)))$ is by definition the $K$-vector space $f_*(f^*(V)^{f^*(\nabla_V)_{|\Sp(K)^{\times}}})$ endowed with the nilpotent endomorphism introduced in the proof of proposition \ref{derhamcohomologywithconnection}. By definition of the connection $f^*(\nabla_V)$, 
$f_*(f^*(V)^{f^*(\nabla_V)_{|\Sp(K)^{\times}}})\cong f_*((V\otimes \mathcal{O}_X, \mathrm{id}\otimes d)^{(\mathrm{id}\otimes d)_{|\Sp(K)^{\times}}})$ as vector spaces and since 
$f_{*\mathrm{dR}}(\mathcal{O}_X, d)\cong K$, $f_*((V\otimes \mathcal{O}_X, \mathrm{id}\otimes d)^{(\mathrm{id}\otimes d)_{|\Sp(K)^{\times}}})\cong V$ as vector spaces. 
Moreover, the action of the nilpotent operator on $f_*((V\otimes \mathcal{O}_X, \mathrm{id}\otimes d)^{(\mathrm{id}\otimes d)_{|\Sp(K)^{\times}}})$ coincides, by the description in proposition \ref{derhamcohomologywithconnection}, with the action of $\nabla_V$ on $V$. 
\end{comment}
Note that we have the isomorphism 
\begin{equation}\label{eq:fff}
f_{*\mathrm{dR}}(f^{*}_{\mathrm{dR}}((V, \nabla_V)))\cong (V, \nabla_V); 
\end{equation}
indeed,  $f_{*\mathrm{dR}}(f^{*}_{\mathrm{dR}}((V, \nabla_V)))$ is by definition the $K$-vector space $f_*(f^*(V)^{f^*(\nabla_V)_{|\Sp(K)^{\times}}})$ endowed with the nilpotent endomorphism introduced in the proof of proposition \ref{derhamcohomologywithconnection}. 
By definition of the connection $f^*(\nabla_V)$ and the isomorphism 
$f_{*\mathrm{dR}}(\mathcal{O}_X, d) = f_*(\mathcal{O}_X^{d_{|\Sp(K)^{\times}}}) 
\cong K$, we have the isomporphism 
$$ f_*(f^*(V)^{f^*(\nabla_V)_{|\Sp(K)^{\times}}})\cong 
f_*((V\otimes \mathcal{O}_X, \mathrm{id}\otimes d)^{(\mathrm{id}\otimes d)_{|\Sp(K)^{\times}}})\cong V$$ 
as vector spaces. Moreover, the action of the nilpotent operator on $f_*((V\otimes \mathcal{O}_X, \mathrm{id}\otimes d)^{(\mathrm{id}\otimes d)_{|\Sp(K)^{\times}}})$ coincides with the action of $\nabla_V$ on $V$ by the description 
of the former one given in proposition \ref{derhamcohomologywithconnection}. 
By applying $H^0_{\rm dR}(\Sp(K)^{\times}/\Sp(K), -)$ to the isomorphism \eqref{eq:fff}, 
we obtain the isomorphism \eqref{eq:ffisom}. Thus the proof of 
the full faithfulness of $f_{\rm dR}^*$ is finished. 

Next we prove that $f^*_{\mathrm{dR}}(MIC(\Sp(K)^{\times}/\Sp(K))^{\mathrm{nr}})$ is closed by taking subobjects. 
Let $(V,\nabla_V)$ be an object of $MIC(\Sp(K)^{\times}/\Sp(K))^{\mathrm{nr}}$, and let $(E, \nabla_E)\subset f^*_{\mathrm{dR}}((V, \nabla_V))$ and $(F, \nabla_F):=f^*_{\mathrm{dR}}((V, \nabla_V))/(E, \nabla_E)$. 
Then we have the exact sequence in $MIC(X^{\times}/\Sp(K))^{\mathrm{nr}}$ defining $(F, \nabla_F)$
\begin{equation}\label{definingF}
0\rightarrow (E, \nabla_E)\rightarrow f^*_{\mathrm{dR}}((V, \nabla_V))\rightarrow (F, \nabla_F)\rightarrow 0. 
\end{equation}
Applying the functor $f^*_{\mathrm{dR}}\circ f_{*\mathrm{dR}}$ to the sequence \eqref{definingF}, we obtain the following commutative diagram with exact rows
\begin{equation}
\xymatrix{
0\ar[r]& (E, \nabla_E)\ar[r]                                                                  & f^*_{\mathrm{dR}}((V, \nabla_V))\ar[r]                                                 &(F, \nabla_F)\ar[r]& 0\\
0\ar[r]&f^*_{\mathrm{dR}} (f_{*\mathrm{dR}}(E, \nabla_E))\ar[r])\ar[u]^{\alpha}&f^*_{\mathrm{dR}}(f_{*\mathrm{dR}}(f^*_{\mathrm{dR}}((V, \nabla_V))))\ar[u]^{\beta}\ar[r]^-{\delta} &f^*_{\mathrm{dR}} (f_{*\mathrm{dR}}(F, \nabla_F)) \ar[u]^{\gamma}.  
}
\end{equation}
The map $\beta$ is an isomorphism for what we have seen in the last paragraph and $\alpha$, $\gamma$ are injective as shown in the proof of proposition \ref{derhamcohomologywithconnection}. 
Then we see easily the surjectivity of $\delta$, and by snake lemma $\alpha$ is an isomorphism, \emph{i.e.} $(E, \nabla_E)\cong f^*_{\mathrm{dR}} (f_{*\mathrm{dR}}(E, \nabla_E))$. So the proof is finished.
\end{proof}

\begin{comment}
\begin{prop}\label{maximalpullback}
Let $(E, \nabla_E)$ be an object of $MIC(X^{\times}/\Sp(K))^{\mathrm{nr}}$.
Then the map 
$$\gamma:H^0_{\mathrm{dR}}(X^{\times}/\Sp(K)^{\times}, r((E, \nabla_E)))\otimes_K \mathcal{O}_X\longrightarrow E$$
is injective and horizontal, namely, $f_{\mathrm{dR}}^*((H^0_{\mathrm{dR}}(X^{\times}/\Sp(K)^{\times}, r((E, \nabla_E))), \nabla))$ is a subobject of $(E, \nabla_E)$.
\end{prop}
\begin{proof}
We already noticed in proposition \ref{derhamcohomologywithconnection} that $\gamma$ induces a morphism in the category $MIC(X^{\times}/\Sp(K))^{\mathrm{nr}}$ between the module with integrable connection $f_{\mathrm{dR}}^*((H^0_{\mathrm{dR}}(X^{\times}/\Sp(K)^{\times}, r((E, \nabla_E))), \nabla))$ and $(E, \nabla_E)$. Hence $\mathrm{Ker}(\gamma)$, endowed with the connection induced by that of $f_{\mathrm{dR}}^*((H^0_{\mathrm{dR}}(X^{\times}/\Sp(K)^{\times}, r((E, \nabla_E))), \nabla))$, is an object of $MIC(X^{\times}/\Sp(K)^{\times})^{\mathrm{nr}}$. Thus $\mathrm{Ker}(\gamma)$ is a locally free $\mathcal{O}_X$-module of finite rank. If $x$ is a closed point of $X_{f-\mathrm{triv}}$, then the fiber of $\mathrm{Ker}(\gamma)$ at $x$ is zero because the fiber of the map $\gamma$ at $x$ is injective by the proof of proposition \ref{derhamcohomologywithconnection}. 

Since $X$ is connected by hypothesis, the rank of $\mathrm{Ker}(\gamma)$ should be constant. Hence $\mathrm{Ker}(\gamma)=0$ and so $\gamma$ is injective.
\end{proof}
\end{comment}

\begin{prop}\label{b}
Let $(E, \nabla_E)$ be an object in $MIC(X^{\times}/\Sp(K))^{\mathrm{nr}}$ and let $(W_0, \nabla_{W_0})$ be the maximal trivial subobject of $r((E, \nabla_E))$ in $\mathrm{Rep}(H)$. Then there exists $(F, \nabla_F)\subset (E, \nabla_E)$  in $MIC(X^{\times}/\Sp(K))^{\mathrm{nr}}$ such that $r((F, \nabla_F))\cong (W_0, \nabla_{W_0})$.
\end{prop}
\begin{proof}
The maximal trivial subobject of $r((E, \nabla_E))$ in $\mathrm{Rep}(H)$ is $(H^0_{\mathrm{dR}}(X^{\times}/\Sp(K)^{\times}, r((E, \nabla_E)))\otimes_K \mathcal{O}_X ,1\otimes d )$ because, as we saw in the proof of proposition of 
\ref{derhamcohomologywithconnection}, the map $\gamma: H^0_{\mathrm{dR}}(X^{\times}/\Sp(K)^{\times},r((E, \nabla_E)))\otimes_K \mathcal{O}_X\rightarrow E$ is injective, $(H^0_{\mathrm{dR}}(X^{\times}/\Sp(K)^{\times},r((E, \nabla_E)))\otimes_K \mathcal{O}_X ,1\otimes d )$ is trivial as subobject of $r((E, \nabla_E))$, and any trivial subobject of $r((E, \nabla_E))$ must be generated by horizontal sections. Since $$(H^0_{\mathrm{dR}}(X^{\times}/\Sp(K)^{\times},r((E, \nabla_E)))\otimes_K \mathcal{O}_X ,1\otimes d )\cong r(f_{\mathrm{dR}}^*((H^0_{\mathrm{dR}}(X^{\times}/\Sp(K)^{\times}, r((E, \nabla_E))), \nabla))),$$ one can prove the proposition by defining 
$(F, \nabla_F) := f_{\mathrm{dR}}^*((H^0_{\mathrm{dR}}(X^{\times}/\Sp(K)^{\times}, r((E, \nabla_E))), \nabla))$.

\end{proof}

\begin{prop}\label{(a)and(b)}
Conditions (a) and (b) of theorem \ref{criterion}(iii) are satisfied for 
the sequence $$H \rightarrow \pi_1(X^{\times}/\Sp(K), x) \rightarrow 
\pi_1(\Sp(K)^{\times}/\Sp(K), \nu).$$  
\end{prop}
\begin{proof}
Proposition \ref{b} gives (b) of theorem \ref{criterion}. The if part of (a) of theorem \ref{criterion} is clear by the fact that the functor $r\circ f_{\mathrm{dR}}^*$ sends every object to a finite sum of the trivial object. Let $(E, \nabla_E)$ be an object of $MIC(X^{\times}/\Sp(K))^{\mathrm{nr}}$ such that $r((E, \nabla_E))$ is trivial. 
Then, because $r((E, \nabla_E))$ itself is the maximal trivial subobject of $r((E, \nabla_E))$, 
$r((E, \nabla_E))\cong r(f_{\mathrm{dR}}^*((H^0_{\mathrm{dR}}(X^{\times}/\Sp(K)^{\times}, r((E, \nabla_E))), \nabla)))$ by the proof of proposition \ref{b}. This implies that the inclusion $f_{\mathrm{dR}}^*((H^0_{\mathrm{dR}}(X^{\times}/\Sp(K)^{\times}, r((E, \nabla_E))), \nabla))\hookrightarrow (E, \nabla_E)$ is an isomorphism. \end{proof}

The following theorem is analogous to  \cite[theorem 5.10]{EsnHai06}. 
\begin{teo}\label{deligne}
Let $(G, \nabla_G)$ be an object of $MIC(X^{\times}/\Sp(K)^{\times})^{\mathrm{nr}}$ and suppose that there exists an object $(E, \nabla_E)$ in $MIC(X^{\times}/\Sp(K))^{\mathrm{nr}}$ such that $(G, \nabla_G)\subset r((E, \nabla_E))$.  Then there exists an object $(F, \nabla_F)$ in $MIC(X^{\times}/\Sp(K))^{\mathrm{nr}}$ such that  $r((F, \nabla_F))$ surjects to $(G, \nabla_G)$.
\end{teo}

\begin{proof}
Suppose first that  $(G, \nabla_G)$ has rank $1$. In this case, we construct $(F, \nabla)$ as the $(G, \nabla_G)$-isotypical component of $(E, \nabla_{E|\Sp(K)^{\times}}):=r((E, \nabla_E))$, in the following way. Let us consider $(E', \nabla'):=(E, \nabla_{E|\Sp(K)^{\times}})\otimes (G, \nabla_G)^{\vee}$: it is an object of $MIC(X^{\times}/\Sp(K)^{\times})^{\mathrm{nr}}$ and the inclusion $(G, \nabla_G)\subset (E, \nabla_{E|\Sp(K)^{\times}})$ corresponds to a non-trivial section  of $H^0_{\mathrm{dR}}(X^{\times}/\Sp(K)^{\times}, (E', \nabla'))$. 
We denote by $(E_0', \nabla_0')$ the object $(H^0_{\mathrm{dR}}(X^{\times}/\Sp(K)^{\times}, (E', \nabla'))\otimes \mathcal{O}_X, {\rm id} \otimes d)$ 
and consider the $\mathcal{O}_X$-module $E_1=E'/E_0'$ with the induced connection $\nabla_1$. If $H^0_{\mathrm{dR}}(X^{\times}/\Sp(K)^{\times}, (E_1, \nabla_1))=0$ we define $(F, \nabla)$  as $(E_0', \nabla_0') \otimes (G, \nabla_G)$. 
Otherwise, we denote by 
$(E_1', \nabla_1')$ the inverse image of $(H^0_{\mathrm{dR}}(X^{\times}/\Sp(K)^{\times}, (E_1, \nabla_1))\otimes \mathcal{O}_X, {\rm id} \otimes d)$ in $E'$ and we consider $E_2=E'/E_1'$ with the induced connection $\nabla_2$ and we proceed as before. If $H^0_{\mathrm{dR}}(X^{\times}/\Sp(K)^{\times}, (E_2, \nabla_2))=0$ we define $(F, \nabla)$ as 
$(E_1', \nabla_1') \otimes (G, \nabla_G)$. Otherwise, we go on with this process until we find the maximal subobject 
$(F,\nabla)$ of $(E, \nabla_{E|\Sp(K)^{\times}})$ which is a successive extension of $(G, \nabla_G)$. In the following, we denote $(F,\nabla)$ also by 
$(F,  (\nabla_{E|\Sp(K)^{\times}})_{|F})$.

We remark that $(F, \nabla)\in MIC(X^{\times}/\Sp(K)^{\times})^{\mathrm{nr}}$ is a subobject of $(E, \nabla_{E|\Sp(K)^{\times}})$.
Thus the $\mathcal{O}_X$-linear map  
\begin{equation}\label{eq:613-1}
F\xrightarrow{\nabla_E} E \otimes \omega^1_{X^{\times}/\Sp(K)} \longrightarrow (E/F) \otimes \omega^1_{X^{\times}/\Sp(K)}
\end{equation}
takes values in $(E/F) \otimes f^*(\omega^1_{\Sp(K)^{\times}/\Sp(K)})$, because the composition with the projection $(E/F) \otimes \omega^1_{X^{\times}/\Sp(K)} \rightarrow (E/F) \otimes \omega^1_{X^{\times}/\Sp(K)^{\times}}$ is the zero map. 
We take an isomorphism $\varphi: f^*(\omega^1_{\Sp(K)^{\times}/\Sp(K)}) 
\xrightarrow{\cong} \mathcal{O}_X$ and let $l: F \to E/F$ be the composition 
$$ F \longrightarrow (E/F) \otimes f^*(\omega^1_{\Sp(K)^{\times}/\Sp(K)}) 
\xrightarrow{{\rm id} \otimes \varphi} E/F, $$
where the first map is the one induced by \eqref{eq:613-1}. 

We have the following local description of $l$ which is also useful. 
Let $D_0 \in \mathcal{H}om(\omega^1_{X^{\times}/\Sp(K)}, \mathcal{O}_X)$ be an element 
with ${D_0}_{|f^*(\omega^1_{\Sp(K)^{\times}/\Sp(K)})} = \varphi$. 
(There exists such an element $D_0$ locally on $X$.) 
Then $l$ is equal to the 
composition 
\begin{equation}\label{eq:613-2}
F \hookrightarrow E \xrightarrow{\nabla_E(D_0)} E \twoheadrightarrow E/F. 
\end{equation}

Now we prove that the diagram 
\begin{equation}\label{morphisml}
\xymatrix{
F\ar[d]_{(\nabla_{E|\Sp(K)^{\times}})_{|F}}\ar[r]^-l& (E/F) 
\ar[d]^{(\nabla_{E|\Sp(K)^{\times}})_{|E/F}}\\
F \otimes \omega^1_{X^{\times}/\Sp(K)^{\times}}\ar[r]^-{l \otimes {\rm id}}& (E/F) \otimes \omega^1_{X^{\times}/\Sp(K)^{\times}}} 
\end{equation}
is commutative. To see this, it suffices to work locally and 
to prove the commutativity of the diagram 
\begin{equation*}
\xymatrix{
F\ar[d]_{(\nabla_{E|\Sp(K)^{\times}})_{|F}(D)}\ar[r]^-l& (E/F) 
\ar[d]^{(\nabla_{E|\Sp(K)^{\times}})_{|E/F}(D)}\\
F \ar[r]^-{l}& (E/F)} 
\end{equation*}
for any $D \in \mathcal{H}om(\omega^1_{X^{\times}/\Sp(K)^{\times}}, \mathcal{O}_X)$.  
By the description of the map $l$ given in \eqref{eq:613-2}, the commutativity of 
the above diagram follows from that of the diagram 
\begin{equation*}
\xymatrix{
E\ar[d]_{\nabla_{E}(D)}\ar[r]^-{\nabla_E(D_0)}& E 
\ar[d]^{\nabla_{E}(D)}\\
E \ar[r]^-{\nabla_E(D_0)}& E} 
\end{equation*}
(where we denoted the composition 
$\omega^1_{X^{\times}/\Sp(K)} \to \omega^1_{X^{\times}/\Sp(K)^{\times}} 
\xrightarrow{D} \mathcal{O}_X$ also by $D$), which follows from the 
integrability of $(E,\nabla_E)$. So the diagram \eqref{morphisml} is commutative, as required. 

Thus we have constructed a morphism 
$$ l: (F, (\nabla_{E|\Sp(K)^{\times}})_{|F}) \longrightarrow (E/F, (\nabla_{E|\Sp(K)^{\times}})_{|E/F})$$
in the category $MIC(X^{\times}/\Sp(K)^{\times})^{\rm nr}$. Since 
all irreducible subquotients of $(F, (\nabla_{E|\Sp(K)^{\times}})_{|F})$ are isomorphic to 
$(G,\nabla_G)$ and $(E/F, (\nabla_{E|\Sp(K)^{\times}})_{|E/F})$ has no irreducible 
subobject isomorphic to $(G,\nabla_G)$, the morphism $l$ must be zero. 
%The object $(E/F) \otimes f^*(\omega^1_{\Sp(K)^{\times}/\Sp(K)})$ with the induced connection $(\nabla_{E|\Sp(K)^{\times}})_{|E/F} \otimes {\rm id}$ is in the category $MIC(X^{\times}/\Sp(K)^{\times})^{\mathrm{nr}}$ and the commutativity of the diagram \eqref{morphisml} implies that $l$ is a morphism in $MIC(X^{\times}/\Sp(K)^{\times})^{\mathrm{nr}}$. Thus $l(F)$ with the induced connection is a subobject of it in $MIC(X^{\times}/\Sp(K)^{\times})^{\mathrm{nr}}$. Since $(G, \nabla_G)$ is a subobject of $(F, \nabla_{E|\Sp(K)^{\times}})$, then $l(G)$, with the induced connection, is also a subobject of it in $MIC(X^{\times}/\Sp(K)^{\times})^{\mathrm{nr}}$. 
%If $l(G)\neq 0$ then it means that the inverse image of $l(G)$ to $E$ with the induced connection is bigger than $(F, \nabla_{E|\Sp(K)^{\times}})$ as $(G, \nabla_G)$-isotypical subobject of $(E, \nabla_{E|\Sp(K)^{\times}})$, but this is absurd by definition of $F$. Hence $l(G)=0$. By replacing $F$ by $F/G$ 
%on the left hand side of the diagram \eqref{morphisml}
%and repeating this argument, we see that $l(F)=0$. 
This proves that $\nabla_E$ stabilizes $F$. Hence $(F,\nabla_{E|F})$ is 
a subobject of $(E,\nabla_E)$ such that $r((F,\nabla_{E|F}))$ has a surjection to 
$(G, \nabla_G)$. So the proof is finished in the rank $1$ case. 

If $G$ has rank $r$ we use the isomorphism
$$(G, \nabla_G)\cong \mathrm{det}(G, \nabla_G)\otimes \wedge^{r-1}(G, \nabla_G)^{\vee}. $$
We define $(F, \nabla_F):=(F', \nabla_{F'})\otimes \wedge^{r-1}(E, \nabla_{E})^{\vee},$ where $(F', \nabla_{F'})$ is the object in $MIC(X^{\times}/\Sp(K))^{\mathrm{nr}}$ we constructed above for the rank $1$ object $\mathrm{det}(G, \nabla_G)$. 
Then we see that $r((F, \nabla_{F})) = 
(F', \nabla_{F'|\Sp(K)^{\times}})\otimes \wedge^{r-1}(E, \nabla_{E|\Sp(K)^{\times}})^{\vee}$ surjects to $(G, \nabla_G)$. So the proof is finished. 

\end{proof}
\begin{prop}\label{(c)}
Condition (c) of theorem \ref{criterion}(iii) is satisfied for the sequence $$H \rightarrow \pi_1(X^{\times}/\Sp(K), x) \rightarrow 
\pi_1(\Sp(K)^{\times}/\Sp(K), \nu).$$  
\end{prop}
\begin{proof} The proof is analogous to the proof of \cite[theorem 5.11]{EsnHai06}.
Take any object $(W, \nabla_W)$ in $\mathrm{Rep}(H)$. Then, by proposition \ref{descriptionH}, there exist $(E, \nabla_E)\in MIC(X^{\times}/\Sp(K))^{\mathrm{nr}}$ and $(G_1, \nabla_{G_1}), (G, \nabla_G)\in MIC(X^{\times} \allowbreak /\Sp(K)^{\times})^{\mathrm{nr}}$ such that $(G_1, \nabla_{G_1})\subset (G, \nabla_G)\subset r((E, \nabla_E))$ and $(W, \nabla_W)\cong (G, \nabla_G)/(G_1, \nabla_{G_1})$. Thanks to theorem \ref{deligne} there exists an object $(F, \nabla_F)\in MIC(X^{\times}/\Sp(K))^{\mathrm{nr}}$ such that $r((F, \nabla_F))$ surjects to  $(G, \nabla_G)$. Hence (c) of theorem \ref{criterion} is verified.
\end{proof}
Putting together all the results, we have the main theorem of this section
\begin{teo}\label{Esnsequence}
The  sequence 
\begin{equation}\label{exactsequencewithoutu}
\pi_1(X^{\times}/\Sp(K)^{\times}, x)\xrightarrow{\pi(r)} \pi_1(X^{\times}/\Sp(K), x)\xrightarrow{\pi(f^*_{\mathrm{dR}})} \pi_1(\Sp(K)^{\times}/\Sp(K), \nu)\xrightarrow{} 1
\end{equation}
is exact. 
\end{teo}
\begin{proof}
The map $\pi(f^*_{\mathrm{dR}})$ is faithfully flat (proposition \ref{pifsurjective}) and the sequence is exact at $\pi_1(X^{\times}/\Sp(K), x)$ because of proposition \ref{(a)and(b)} and proposition \ref{(c)}. 
\end{proof}

\section{Description of the kernel of the second map} 

In this section, 
we introduce an auxiliary neutral Tannakian category denoted by $MIC(X^{\times}/\Sp(K)[u], x)^{\mathrm{nr}}$ and its corresponding fundamental group denoted by $\pi_1(X^{\times}/\Sp(K)[u], x)$, 
and we will prove (proposition \ref{exactu}) that $\pi_1(X^{\times}/\Sp(K)[u], x)$ is isomorphic to $\mathrm{Ker}(\pi(f^*_{\mathrm{dR}}))$. 
As before, we consider the map of log schemes $g:\Sp(K)^{\times} = (\Sp(K), N) \rightarrow \Sp(K)$ and denote simply by $\mathrm{dlog}1 \in \omega^1_{\Sp(K)^{\times}/\Sp(K)}$ the element $\mathrm{dlog} (1,1)$ 
(in the notation of \cite[(1.7)]{Kato89}) with  
$(1,1) \in \mathbb{N} \times K^{*} \cong N$. 
(We remark that $\mathrm{dlog}1$ is independent of the choice of the isomorphism 
${\mathbb{N}} \times K^{*} \cong N$.) 

\begin{defi}\label{categoriau} 
Let $z:(Y, L)\rightarrow \Sp(K)^{\times}$ be a morphism of fine log schemes and let $u$ be a variable. We define an action of the exterior differential on $u^i$ $(i > 0)$ in such a way that $du$ is the image of $\mathrm{dlog}1$ under the map $z^{-1}(\omega^1_{\Sp(K)^{\times}/\Sp(K)})\rightarrow \omega^1_{(Y,L)/\Sp(K)}$ and that $du^i$ is equal to $iu^{i-1}du$. We denote by $\mathcal{O}_Y[u]$ the sheaf of algebras $\bigoplus_{i=0}^{\infty}\mathcal{O}_Yu^i$ and by $d:\mathcal{O}_Y[u]\rightarrow \mathcal{O}_Y[u] \otimes \omega^1_{(Y,L)/\Sp(K)}$ the $K$-linear extension of 
the derivation on $\mathcal{O}_Y$ 
such that $d(cu^i)=d(c)u^i+ciu^{i-1}du$ for $c\in \mathcal{O}_Y$ and $i\geq1$.

For a coherent $\mathcal{O}_Y[u]$-module $E$, 
a connection $\nabla_E$ on $E$ is a $K$-linear map 
$$\nabla_E:E\longrightarrow E\otimes \omega^1_{(Y,L)/\Sp(K)}$$
which satisfies the Leibniz rule 
$$\nabla_E(ae)=a\nabla_E(e)+e\otimes da \quad (a \in \mathcal{O}_Y[u], e\in E). $$
We can extend $\nabla_E$ to the map 
$$ \nabla_{E,i}: E\otimes \omega^i_{(Y,L)/\Sp(K)} \rightarrow E\otimes \omega^{i+1}_{(Y,L)/\Sp(K)} \quad (i \geq 1) $$ 
by $\nabla_{E,i}(e\otimes \omega)=e \otimes d\omega+\nabla_E(e)\wedge \omega$. We say that $\nabla_{E}$ is integrable if $\nabla_{E,1} \circ \nabla_{E}=0.$

If $(E, \nabla_E)$ and $(F, \nabla_F)$ are two coherent $\mathcal{O}_Y[u]$-modules endowed with integrable connection, and if $\gamma:E\rightarrow F$ is a morphism of $\mathcal{O}_{Y}[u]$-modules, we say that $\gamma$ is horizontal if 
$ (\gamma \otimes {\mathrm{id}}) \circ \nabla_E = \nabla_F \circ \gamma $. 
Given $(E, \nabla_E)$ and $(F, \nabla_F)$, two coherent $\mathcal{O}_Y[u]$-modules endowed with integrable connection, the tensor product is defined as  the pair $(E\otimes F, \nabla)$, where $\nabla=\mathrm{id}\otimes \nabla_F+\nabla_E\otimes \mathrm{id}$.

We denote by $MIC((Y,L)/\Sp(K)[u])$ the category whose objects are pairs $(E, \nabla_E)$ consisting of a coherent $\mathcal{O}_Y[u]$-module $E$ and an integrable connection $\nabla_E$ on $E$. The morphisms between two objects $(E, \nabla_E)$ and $(F, \nabla_F)$  are morphisms of $\mathcal{O}_X[u]$-modules which are horizontal. 
\end{defi}

\begin{rmk}
We can give the definition of a connection on a coherent $\mathcal{O}_Y[u]$-module $E$ 
using the sheaf of log derivations $\mathcal{D}er((Y,L)/\Sp(K))$, as in definition \ref{connections}. 
The details are left to the reader. 
\end{rmk}

\begin{defi}\label{categoriauu}
We denote by $MIC((Y,L)/\Sp(K)[u])^{\mathrm{nr}}$ the full subcategory 
of $MIC((Y,L)/\Sp(K)[u])$ consisting of objects $(E, \nabla_E)$ which satisfy 
the following conditions: 
\begin{itemize}
\item[(i)] for any geometric point $y$ over a closed point of $Y$, $\hat{E_{y}}:=E_y\otimes_{\mathcal{O}_{Y,y}[u]} \mathcal{O}_{Y,y}[u]^{\wedge}$ is free as $\mathcal{O}_{Y,y}[u]^{\wedge}$-module, 
where $\mathcal{O}_{Y,y}[u]^{\wedge}$ is the $\mathfrak{m}_{Y,y}[u]$-adic completion 
of $\mathcal{O}_{Y,y}[u]$ and $\mathfrak{m}_{Y,y}$ is the maximal ideal of $\mathcal{O}_{Y,y}$. 
\item[(ii)] if we define the residue $\rho_{y}:E(y)\rightarrow E(y)\otimes \overline{L}^{\mathrm{gp}}_{y}$ for a geometric point $y$ over a closed point of $Y$ 
in the same way as 
the paragraph before definition \ref{residuenilpMN}, $E(y)$ can be written as a union 
$E(y) = \bigcup_{i \in \mathbb{N}} E(y)_i$ of finite dimensional subspaces $E(y)_i$ such that 
\begin{itemize}
\item[(ii-1)] $\rho_y(E(y)_i) \subset E(y)_i \otimes \overline{L}^{\mathrm{gp}}_{y}$. 
\item[(ii-2)] for every $K(y)$-linear map $t_{y}:K(y)\otimes \overline{L}^{\mathrm{gp}}_{y} \rightarrow K(y)$, the composite map 
$({\rm id} \otimes t_{y}) \circ (\rho_{y|E(y)_i}): E(y)_i \rightarrow E(y)_i$ 
defined as in definition \ref{residuenilpMN} is nilpotent. 
\end{itemize}
%\item[(iii)] there exists $(E_0, \nabla_{E_0})\in MIC((Y,L)/\Sp(K))^{\mathrm{nr}}$ (which will be called a lattice) such that $E_0\subset E$, $\nabla_{E|E_0}=\nabla_{E_0}$ and the morphism $s:E_0\otimes_{\mathcal{O}_X}\mathcal{O}_{X}[u]\rightarrow E$ induced by the inclusion $E_0\subset E$ is surjective.
\item[(iii)] as a quasi-coherent $\mathcal{O}_Y$-module with integrable connection, 
$(E,\nabla_E)$ is a colimit of objects $(E_i, \nabla_{E_i})$ in 
$MIC((Y,L)/\Sp(K))^{\mathrm{nr}}$ indexed by $\mathbb{N}$. 
\end{itemize}
Also, we denote by $MIC((Y,L)/\Sp(K)[u])'$ the full subcategory 
of $MIC((Y,L)/\Sp(K)[u])$ consisting of objects $(E, \nabla_E)$ which satisfy 
the conditions (i), (ii) above. 
\end{defi}

\begin{rmk}\label{rmk4.4}
The condition (iii) implies the condition (ii) in definition \ref{categoriauu}; indeed, 
for any $y$ as in (ii), if we denote the image of $E_i(y)$ (where $E_i$ is as in (iii)) 
in $E(y)$ by $E(y)_i$, the subspaces $E(y)_i \,(i \in \mathbb{N})$ of $E(y)$ 
satisfy the condition (ii). 
We write the conditions (ii), (iii) separately due to technical reasons in the proofs below. 
%Note that 
%$\mathcal{O}_{X}[u]$ is naturally endowed with 
%the integrable connection $d$ which extends the derivation on $\mathcal{O}_{X}$ and satisfies $du^i = iu^{i-1}du$, and 
%the $\mathcal{O}_X$-submodule $\mathcal{O}_{X}[u]_{\leq i}$ of $\mathcal{O}_{X}[u]$ generated by $1,u,\ldots, u^i$ 
%is stable under the connection $d$. Then $E_0\otimes_{\mathcal{O}_X}\mathcal{O}_{X}[u]$ is endowed with an integrable connection 
%and the surjection $s:E_0\otimes_{\mathcal{O}_X}\mathcal{O}_{X}[u]\rightarrow E$ in the condition (iii) in definition \ref{categoriauu}
%is compatible with connections. Then the condition (ii) for $(E, \nabla_E)$ follows from that for 
%$E_0\otimes_{\mathcal{O}_X}\mathcal{O}_{X}[u]$, and this is true because we can take 
%$E_0(y) \otimes_{K(y)} K(y)[u]_{\leq i}$ as $E(y)_i$ in the condition (ii). \par 
%We write the conditions (ii), (iii) separately due to technical reasons in the proofs below. 
\end{rmk}

Next we give a `formal version' of definitions  
\ref{categoriau}, \ref{categoriauu}. 

\begin{defi}\label{categoriau-formal}
Let $B$ be a Noetherian $I$-adically complete ring for an ideal $I$ and let 
$Y=\Sp(B)$. Let $z: (Y,L) \to \Sp(K)^{\times}$ be a morphism of fine log schemes and 
assume that the completed differential module 
$\hat{\omega}^1_{(Y,L)/\Sp(K)}$ associated to the composite $(Y,L) \xrightarrow{z} \Sp(K)^{\times} \to \Sp(K)$ (defined in remark \ref{rmk:formal}) is finitely generated 
as $B$-module. Let $u$ be a variable and we define 
an action of the exterior differential on $u^i$ $(i > 0)$ 
in the same way as in definition \ref{categoriau}. 
Let $B[u]^{\wedge}$ be the 
$I[u]$-adic completion of 
the polynomial ring $B[u]$ and let 
$d: B[u]^{\wedge} \rightarrow B[u]^{\wedge} \otimes \hat{\omega}^1_{(Y,L)/\Sp(K)}$ the $K$-linear continuous extension of the derivation on $B$ 
such that $d(cu^i)=d(c)u^i+ciu^{i-1}du$ for $c\in B$ and $i\geq1$.

For a finitely generated $B[u]^{\wedge}$-module $E$, 
a formal connection $\nabla_E$ on $E$ is a $K$-linear map 
$$\nabla_E:E\longrightarrow E\otimes \hat{\omega}^1_{(Y,L)/\Sp(K)}$$
which satisfies the Leibniz rule as before. We can define the integrability 
of a formal connection and the horizontality of a morphism of 
formal integrable connections as in definition \ref{categoriau}. 

We denote by $\widehat{MIC}((Y,L)/\Sp(K)[u])$ the category whose objects are pairs $(E, \nabla_E)$ consisting of a finitely generated $B[u]^{\wedge}$-module $E$ and an integrable connection $\nabla_E$ on $E$. The morphisms between two objects $(E, \nabla_E)$ and $(F, \nabla_F)$  are morphisms of $B[u]^{\wedge}$-modules which are horizontal. 
\end{defi}

\begin{defi}\label{categoriauu-formal}
We denote by $\widehat{MIC}((Y,L)/\Sp(K)[u])'$ the full subcategory 
of $\widehat{MIC}((Y,L)/\Sp(K)[u])$ consisting of objects $(E, \nabla_E)$ which satisfy 
the following conditions: 
\begin{itemize}
\item[(i)] $E$ is free as $B[u]^{\wedge}$-module. 
\item[(ii)] if we define the residue $\rho_{y}:E(y)\rightarrow E(y)\otimes \overline{L}^{\mathrm{gp}}_{y}$ 
(here $E(y) := E \otimes_{B[u]^{\wedge}} K(y)[u]$) 
for a geometric point $y$ over a closed point of $Y$ 
in the same way as the paragraph before definition \ref{residuenilpMN}, 
$E(y)$ can be written as a union 
$E(y) = \bigcup_i E(y)_i$ of finite dimensional subspaces $E(y)_i$ 
such that the conditions (ii-1), (ii-2) in definition \ref{categoriauu} are 
satisfied. 
\end{itemize}
\end{defi} 

\begin{comment}
\begin{rmk}
We can give definitions of `formal versions' in an analogous way to definitions \ref{categoriau}, \ref{categoriauu} if $Y=\Sp(B)$, where $B$ is a Noetherian $I$-adically complete ring for $I$ an ideal of $B$, $z:(Y, L)\rightarrow \Sp(K)^{\times}$ is a morphism of fine log schemes and $u$ is a variable such that $du$ is the image of $\mathrm{dlog}1$ under the map $z^{-1}(\omega^1_{\Sp(K)^{\times}/\Sp(K)})\rightarrow \hat{\omega}^1_{(Y,L)/\Sp(K)}$. 
However, by technical reasons, the conditions (i), (ii) are imposed only at geometric 
points $y$ over closed points of $Y$. 
The analogous categories we obtain (with the above slight modification) will be denoted by 
$\widehat{MIC}((Y,L) \allowbreak /\Sp(K)[u])$, $\widehat{MIC}((Y,L)/\Sp(K)[u])^{\mathrm{nr}}$ and  $\widehat{MIC}((Y,L)/\Sp(K)[u])'$. 
\end{rmk}
\end{comment}

\begin{rmk}
Definitions \ref{categoriau-formal}, \ref{categoriauu-formal} 
are not a naive analogue of definitions \ref{categoriau}, \ref{categoriauu} 
in the sense that $E$ in definitions \ref{categoriau-formal}, \ref{categoriauu-formal}  
is a $B[u]^{\wedge}$-module (not a $B[u]$-module).  
%and that the condition of 
%nilpotent residues are imposed only on geometric points over closed points of $Y$. 
Our definition is designed in order that an induction argument works well later. 
We will not define a formal analogue of the category 
$MIC((Y,L)/\Sp(K)[u])^{\mathrm{nr}}$ because we will not need it. 

Also, we remark that the condition (i) in definition 
\ref{categoriauu-formal} is a global condition, but this will not cause 
any problem because we will use this definition only when $B$ is a complete local ring 
with maximal ideal $I$. 
\end{rmk}

\begin{lemma}\label{K[u]/kabeliana}
%\begin{itemize}
%\item[(i)] The kernel of a horizontal morphism $\varphi: (E_0, \nabla_{E_0}) \rightarrow (F,\nabla_F)$ 
%of modules with integrable connection on $\Sp(K)^{\times}/\Sp(K)$  
%from an object $(E_0, \nabla_0)$ in $MIC(\Sp(K)^{\times}/\Sp(K))^{\mathrm{nr}}$ to 
%an object $(F,\nabla_F)$ in $MIC(\Sp(K)^{\times}/\Sp(K)[u])'$ belongs to   $MIC(\Sp(K)^{\times}/\Sp(K))^{\mathrm{nr}}$. 
%\item[(ii)] 
The category $MIC(\Sp(K)^{\times}/\Sp(K)[u])'$ is abelian. 
%\item[(iii)]  For an object $(E, \nabla_E)$ in $MIC(\Sp(K)^{\times}/\Sp(K)[u])'$, 
%the $K$-vector space $E^{\nabla_E}$ of horizontal sections is isomorphic to $E/uE$. 
%\end{itemize}
\end{lemma}
\begin{proof}
%The assertion (i) is easy. 
To prove the assertion, we only need to check that property (i) of definition \ref{categoriau} is stable by kernel and cokernel of any morphism. We will prove that every finitely generated $K[u]$-module endowed with an integral connection is in fact free. Since $K[u]$ is a PID, it is enough to prove 
the torsion-freeness. The proof is analogous to \cite[proposition 6.1]{Cre98}. Let $(E, \nabla_E)$ be a finitely generated $K[u]$-module endowed with an integrable connection. The annihilator of the torsion submodule of $E$ is an ideal of $K[u]$. Assume it is generated by $h\in K[u].$ Then, for $e$ in the torsion submodule of $E,$ $he=0$. If we denote by $\partial_u$  the derivation which sends $du$ to $1$, then $d(\partial_u)(h)e+h\nabla_E(\partial_u)(e)=0,$ and multiplying by $h$ we obtain that $h^2\nabla_E(\partial_u)(e)=0$. Hence $\nabla_E(\partial_u)(e)$ is in the torsion submodule of $E$, so that $h\nabla_E(\partial_u)(e)=0,$ hence $d(\partial_u)(h)e=0.$ This implies that $d(\partial_u)(h)\in (h),$ hence $1\in (h).$ Hence $E$ does not have any nonzero torsion element. 
\begin{comment}
Finally, we prove the assertion (iii). 
Thanks to nilpotence condition (ii) of definition \ref{categoriau}, for every $e\in E$ the element $\sum_k(-1)^ku^k\frac{\nabla_E(\partial_u)^k(e)}{k!}$ is 
still in $E$. Hence we can proceed as in \cite[proposition 8.9]{Kat70} and prove that the $K$-linear map 
$$P:E\rightarrow E$$
$$P:e\mapsto \sum_k(-1)^ku^k\frac{\nabla_E(\partial_u)^k}{k!}(e)$$
induces an isomorphism of $K$-vector spaces $E/uE\cong E^{\nabla_E}$. 
\end{comment} 
\end{proof}
\begin{prop}\label{(i)general}
%\begin{itemize}
%\item[(i)] The kernel of a horizontal morphism $\varphi: (E_0, \nabla_{E_0}) \rightarrow (F,\nabla_F)$ of modules with integrable connection on $X^{\times}/\Sp(K)$ from an object $(E_0, \nabla_0)$ in $MIC(X^{\times}/\Sp(K))^{\mathrm{nr}}$ to 
%an object $(F,\nabla_F)$ in $MIC(X^{\times}/\Sp(K)[u])'$ belongs to the category $MIC(X^{\times}/\Sp(K))^{\mathrm{nr}}$. 
%\item[(ii)] 
The category $MIC(X^{\times}/\Sp(K)[u])'$ is abelian. 
%\end{itemize}
\end{prop}
\begin{proof}
%First we prove (ii). 
It is sufficient to prove that the kernel and the cokernel of 
any morphism in $MIC(X^{\times}/\Sp(K)[u])'$ belongs to 
$MIC(X^{\times}/\Sp(K)[u])'$. 
Take any geometric point $x$ over a closed point of $X$ and consider 
 the restriction functor 
$$ MIC(X^{\times}/\Sp(K)[u]) \rightarrow 
\widehat{MIC}(\Sp(\hat{\mathcal{O}}_{X,x}), M)/\Sp(K)[u]). $$
Since the map $\mathcal{O}_{X,x}[u] \to \hat{\mathcal{O}}_{X,x}[u]^{\wedge}$ is 
flat, the functor is exact. Also, it is easy to see that it induces the functor 
\begin{equation}\label{eq:rest'}
MIC(X^{\times}/\Sp(K)[u])' \rightarrow 
\widehat{MIC}(\Sp(\hat{\mathcal{O}}_{X,x}), M)/\Sp(K)[u])' 
\end{equation}
and that an object in $MIC(X^{\times}/\Sp(K)[u])$ belongs to 
$MIC(X^{\times}/\Sp(K)[u])'$ if the restriction of it to 
$\widehat{MIC}(\Sp(\hat{\mathcal{O}}_{X,x}), M)/\Sp(K)[u])$ 
belongs to $\widehat{MIC}(\Sp(\hat{\mathcal{O}}_{X,x}), M)/\Sp(K)[u])'$ 
for any $x$. So it suffices to prove that 
$\widehat{MIC}(\Sp(\hat{\mathcal{O}}_{X,x}), M)/\Sp(K)[u])'$ is an 
abelian category for the completed local ring 
$\hat{\mathcal{O}}_{X,x} = K[[x_1, \dots, x_n]]/(x_1 \cdots x_r)$. 
By lemma \ref{RS2} below, we can reduce to the case $r=n$. 
Then we proceed by induction on $n$. The case $n=2$ is proven in proposition \ref{(i)1dim} below and the induction step is proven in proposition \ref{(i)ndim} below. 
\end{proof}

\begin{rmk}\label{notationlocalderivations2} 
In what follows we consider on 
the spectrum of  $S=K[[x_1, \dots, x_n]]/(x_1\cdots x_r)$ 
the log structure $M$ as in remark \ref{notationlocalderivations}. 
Then we have 
$$
\hat{\omega}^1_{(\Sp(S), M)/\Sp(K)} \cong \bigoplus_{i=1}^rS\mathrm{dlog}x_i 
\oplus \bigoplus_{i = r+1}^n S dx_{i}. $$ 
We will use a basis $\{\partial_1, \dots, \partial_{r-1}, \partial_r, D_{r+1}, \dots, D_n\}$ of its dual 
$$\widehat{\mathcal{D}er}((\Sp(S), M)/\Sp(K)) = Hom(\hat{\omega}^1_{(\Sp(S), M)/\Sp(K)}, S),$$ 
defined in the following way: $\partial_1, \dots, \partial_{r-1}, D_{r+1}, \dots, D_n$ are as in 
remark \ref{notationlocalderivations} and $\partial_r$ is the derivation that sends 
$\mathrm{dlog}x_j$ to $0$ for $j=1, \dots, r-1$, $\mathrm{dlog}x_r$ to $1$ and 
$dx_j$ to $0$ for every $j=r+1, \dots, n$. 
 \end{rmk}

\begin{lemma}\label{RS2}
Let $R=K[[x_1, \dots, x_r]]/(x_1\cdots x_r)$ and $S=K[[x_1, \dots, x_n]]/(x_1\cdots x_r)$ be as above. 
\begin{comment}
\begin{itemize}
\item[(i)]
If the kernel of a horizontal morphism of modules with formal integrable connection
on $(\Sp(R),M)\allowbreak /\Sp(K)$ 
from an object in the category 
$\widehat{MIC}((\Sp(R),M)/\Sp(K))^{\mathrm{nr}}$ to 
an object in the category $\widehat{MIC}((\Sp(R),M)/\Sp(K)[u])'$ belongs to the category 
$\widehat{MIC}((\Sp(R),M)/\Sp(K))^{\mathrm{nr}}, $
the kernel of a horizontal morphism $\varphi: (E_0, \nabla_{E_0}) \rightarrow (F,\nabla_F)$ of modules with formal formal integrable connection
on $(\Sp(S),M)/\Sp(K)$ 
from an object $(E_0, \nabla_0)$ in the category $\widehat{MIC}((\Sp(S),M)/\Sp(K))^{\mathrm{nr}}$ to 
an object $(F,\nabla_F)$ in the category $\widehat{MIC}(\Sp(S),M)/\Sp(K)[u])'$ belongs to the category 
$\widehat{MIC}((\Sp(S), \allowbreak M)/\Sp(K))^{\mathrm{nr}}.$
\item[(ii)] 
\end{comment}
If the category $\widehat{MIC}((\Sp(R), M)/\Sp(K)[u])'$ is stable by kernel and cokernel of any morphism, then the category $\widehat{MIC}((\Sp(S), M)/\Sp(K)[u])'$ is stable by kernel and cokernel of any morphism.
%\end{itemize}
\end{lemma}

\begin{proof}
Note that $R[u]^{\wedge}=K[u][[x_1, \dots, x_r]]/(x_1\cdots x_r)$ and 
$S[u]^{\wedge}=K[u][[x_1, \dots, x_n]]/(x_1\cdots x_r)$. 
% First we prove (ii). 
For an object $(E,\nabla_E)$ in 
$\widehat{MIC}((\Sp(S), M)/\Sp(K)[u])'$, put 
$$ \overline{E} := \{ e \in E \,|\, \nabla_E(D_i)(e) = 0 \,(r+1 \leq i \leq n)\}. $$ 
Then $\overline{E}$ is stable by the action of $\nabla_E(\partial_i) \, (1 \leq i \leq r)$. 
It suffices to prove that $\overline{E}$ with the above action defines 
an object in $\widehat{MIC}((\Sp(R), M)/\Sp(K)[u])'$ and that 
the functor 
$$ \widehat{MIC}((\Sp(S), M)/\Sp(K)[u])' \to 
\widehat{MIC}((\Sp(R), M)/\Sp(K)[u])'; \quad (E,\nabla_E) \mapsto \overline{E} $$ 
is an equivalence of categories whose quasi-inverse is given by 
$(\overline{E},\nabla_{\overline{E}}) \mapsto (E := \overline{E} \otimes_{R[u]^{\wedge}} S[u]^{\wedge}, \nabla_E)$ 
with the action $\nabla_E(\partial_i) \, (1 \leq i \leq r), \nabla_E(D_i)\,(r+1 \leq i \leq n)$ defined by 
$\nabla_E(\partial_i) := {\rm id} \otimes d(\partial_i) + \nabla_{\overline{E}}(\partial_i) \otimes {\rm id}, 
\nabla_E(D_i) := {\rm id} \otimes d(D_i)$. (The reason is the same as that in lemma \ref{RS}.) 
%indeed, if this claim is proved and if we are given a morphism $\varphi: E \to F$ in 
%$\widehat{MIC}((\Sp(S), M)/\Sp(K)^{\times})^{\mathrm{nr}}$, it induces a morphism 
%$\overline{\varphi}: \overline{E} \to \overline{F}$ in 
%$\widehat{MIC}((\Sp(R), M)/\Sp(K)^{\times})^{\mathrm{nr}}$, and 
%${\rm Ker}(\overline{\varphi}), {\rm Coker}(\overline{\varphi})$ are defined as objects in 
%$\widehat{MIC}((\Sp(R), M)/\Sp(K)^{\times})^{\mathrm{nr}}$. Then we have 
%${\rm Ker}(\varphi) = {\rm Ker}(\overline{\varphi}) \otimes_R S, 
%{\rm Coker}(\varphi) = {\rm Coker}(\overline{\varphi}) \otimes_R S$ and they are 
%objects in $\widehat{MIC}((\Sp(S), M)/\Sp(K)^{\times})^{\mathrm{nr}}$, as required. 
We can prove the above claim in the same way as lemma \ref{RS}. So we are done. 
\begin{comment}
Next we prove (i). If we define $\overline{E}_0$ as in lemma \ref{RS} and 
$\overline{F}$ as in the previous paragraph, a morphism 
$\varphi: (E_0, \nabla_{E_0}) \to (F,\nabla_F)$ induces a morphism 
$\overline{\varphi}: \overline{E}_0 \to \overline{F}$, which is regarded 
as a horizontal morphism of modules with formal integrable connection 
on $(\Sp(R),M)\allowbreak /\Sp(K)$ 
from an object in the category 
$\widehat{MIC}((\Sp(R),M)/\Sp(K))^{\mathrm{nr}}$ to 
an object in the category $\widehat{MIC}((\Sp(R),M)/\Sp(K)[u])'$. 
Also, $\overline{\varphi}$ induces a map 
$$ \overline{E}_0 \otimes_R S \to \overline{F} \otimes_{R[u]^{\wedge}} S[u]^{\wedge}, $$ 
which is the same as $\varphi$ by lemma \ref{RS} and the proof of (ii). 
The above map is written as the composite 
$$ \overline{E}_0 \otimes_R S \xrightarrow{\overline{\varphi} \otimes {\rm id}} 
\overline{F} \otimes_R S \to 
\overline{F} \otimes_{R[u]^{\wedge}} S[u]^{\wedge}. $$
Since the natural map $R[u]^{\wedge} \otimes_R S \to S[u]^{\wedge}$ is injective and 
$\overline{F}$ is a free $R[u]^{\wedge}$-module, the second map in the above 
diagram is injective. Thus we see that ${\rm Ker}(\varphi) = {\rm Ker}(\varphi_0) \otimes_R S$ 
and using this, we see that ${\rm Ker}(\varphi)$ belongs to the category 
$\widehat{MIC}((\Sp(S),M)/\Sp(K))^{\mathrm{nr}},$ as required. 
\end{comment}
\end{proof}

\begin{prop} \label{(i)1dim}
Let $M$ be the log structure on $\Sp(K[[x,y]]/(xy))$ defined in remark \ref{notationlocalderivations2}. 
%\begin{itemize}
%\item[(i)]
%The kernel of a horizontal morphism $\varphi: (E_0, \nabla_{E_0}) \rightarrow (F,\nabla_F)$ of modules with integrable connection on $(\Sp(K[[x,y]]/(xy)),M)/\Sp(K)$ 
%from an object $(E_0, \nabla_0)$ in $\widehat{MIC}((\Sp(K[[x,y]]/(xy)), \allowbreak M)/\Sp(K))^{\mathrm{nr}}$ to 
%an object $(F,\nabla_F)$ in $\widehat{MIC}((\Sp(K[[x,y]]/(xy)),M)/\Sp(K)[u])'$ belongs to $\widehat{MIC} \allowbreak ((\Sp(K[[x,y]]/(xy)),M)/\Sp(K))^{\mathrm{nr}}$. 
%\item[(ii)] 
Then the category $\widehat{MIC}((\Sp(K[[x,y]]/(xy)),M)/\Sp(K)[u])'$ is abelian. 
%\end{itemize}
\end{prop}
\begin{proof} 
%We give a proof of (ii) first and give a proof of (i) briefly at the end. 
Let $\partial_1, \partial_2$ be as in the notation introduced in remark \ref{notationlocalderivations2}. 
%which is a basis of $$\widehat{\mathcal{D}er}((\Sp(K[[x,y]]/(xy)), M)/\Sp(K)^{\times}) =  
%\mathcal{H}om(\hat{\omega}^1_{(\Sp(K[[x,y]]/(xy)), M)/\Sp(K)^{\times}}, K[[x,y]]/(xy)).$$ 
For an object $(E,\nabla_E)$ in 
$\widehat{MIC}((\Sp(K[[x,y]] \allowbreak /(xy)), M)/\Sp(K)[u])'$, put 
$$\overline{E} := \{ e \in E \,|\, \exists N \in \mathbb{N}, \nabla_E^N(\partial_1)(e) = 0\}. $$ 
Then $\overline{E}$ is stable by the action of 
$\nabla_E(\partial_i) \, (i=1,2)$ and the action of $\nabla_E(\partial_1)$ is locally nilpotent. 
Then, it suffices to prove that the correspondence 
$(E,\nabla_E) \mapsto (\overline{E}, \nabla_E(\partial_1))$ defines the 
functor 
\begin{align*}
\widehat{MIC}((\Sp(K[[x,y]]/(xy)), & M)/\Sp(K)[u])' \\ & \to 
\left\{ 
((F,\nabla_F),N) \,\left|\, 
\begin{aligned} 
& \text{$(F,\nabla_F) \in MIC(\Sp(K)^{\times}/\Sp(K)[u])'$} \\
& \text{$N: (F,\nabla_F) \to (F,\nabla_F)$: a nilpotent endomorphism}
\end{aligned}
\right. 
\right\} 
\end{align*}
and that it 
is an equivalence whose quasi-inverse is given by 
$((F,\nabla_F),N) \mapsto (E := F \otimes_{K[u]} K[u][[x,y]]/(xy), \nabla_E)$ 
with the action $\nabla_E(\partial_i) \,(i=1,2)$ defined by 
$\nabla_E(\partial_1) :=  {\rm id} \otimes d(\partial_1) + N \otimes {\rm id}, 
\nabla_E(\partial_2) :=  {\rm id} \otimes d(\partial_2) + \nabla_F(\partial_2) \otimes {\rm id}
$. (The reason is the same as that in lemma \ref{RS}.) 

To prove the above claim, it suffices to construct a basis of $\overline{E}$ over 
$K[u]$ which is a basis of $E$ over $K[u][[x,y]]/(xy)$. 
To do so, first we prove that there exists a basis of $E$ on which $\nabla_E(\partial_1)$ 
acts as a strictly upper triangular matrix with entries in $K[u]$. If $E$ has rank $n,$ 
by hypothesis of nilpotent residues, we can write 
$$\nabla_E(\partial_1)= d(\partial_1) + H $$
with respect to some basis, where 
$H=(a_{i,j}(x,y))_{i,j}$ is an $n\times n$ matrix such that $H_0=(a_{i,j}(0,0))_{i,j}$ is a 
nilpotent matrix with entries in $K[u]$. Moreover, by changing the basis, 
we may assume that $H_0$ is strictly upper triangular. 
Using the fact $d(\partial_1)(u)=0$, the same argument as proposition \ref{lacrocesuN} works to deduce that there exists a basis $e_1,\dots ,e_s$ of $E$ with respect to which the action of $\nabla_E(\partial_1)$ is given by the matrix $H_0$.
Also, we can prove the equality $K[u]e_1+\cdots +K[u]e_s = \overline{E}$
in the same way as the proof of proposition \ref{lacrocesuN}. 
Thus we have the claim we want and so the proof is finished. 
\begin{comment}
We give a proof of (i). If we define $\overline{E}_0$ (resp. $\overline{F}$) 
as in the proof of proposition \ref{lacrocesuN} (resp. as in the proof of (ii) above), 
the morphism $\varphi$ induces the morphism $\overline{\varphi}: \overline{E}_0 \to 
\overline{F}$, which is regarded as a morphism of 
modules with integrable connection on $\Sp(K)^{\times}/\Sp(K)$ 
from an object in $MIC(\Sp(K)^{\times}/\Sp(K))^{\mathrm{nr}}$ to 
an object in $MIC(\Sp(K)^{\times}/\Sp(K)[u])'$. Hence 
${\rm Ker}(\varphi_0)$ belongs to $MIC(\Sp(K)^{\times}/\Sp(K))^{\mathrm{nr}}$ 
by lemma \ref{K[u]/kabeliana}. Also, $\varphi$ is written as the composite 
$$ E = \overline{E} \otimes_K K[[x,y]]/(xy) \xrightarrow{\overline{\varphi} \otimes {\rm id}}
\overline{F} \otimes_K K[[x,y]]/(xy) \rightarrow \overline{F} \otimes_{K[u]} 
K[u][[x,y]]/(xy) = F. $$
Since $\overline{F}$ is a free $K[u]$-module, the second map is injective and so 
${\rm Ker}(\varphi) = {\rm Ker}(\varphi_0) \otimes_K K[[x,y]]/(xy)$. 
Thus we see that $\mathrm{Ker}(\varphi)$ is an object in  
$\widehat{MIC}((\Sp(K[[x,y]]/(xy)), M)\allowbreak /\Sp(K))^{\mathrm{nr}}$. 
So we finished the proof of (i). 
\end{comment}
\end{proof}

\begin{prop}\label{(i)ndim} 
\begin{comment}
\begin{itemize}
\item[(i)]
If the kernel of a horizontal morphism of modules with formal integrable connections 
on $(\Sp(K[[x_1,...,x_{n-1}]] \allowbreak /(x_1 \cdots x_{n-1})),M)/\Sp(K)$ 
from an object in the category 
$$\widehat{MIC}((\Sp(K[[x_1,...,x_{n-1}]] \allowbreak /(x_1 \cdots x_{n-1})),M)/\Sp(K))^{\mathrm{nr}}$$ to 
an object in the category $$\widehat{MIC}((\Sp(K[[x_1,...,x_{n-1}]]/(x_1 \cdots x_{n-1})),M)/\Sp(K)[u])'$$ belongs to the category 
$$\widehat{MIC}((\Sp(K[[x_1,...,x_{n-1}]]/(x_1 \cdots x_{n-1})),M)/\Sp(K))^{\mathrm{nr}}, $$
the kernel of a horizontal morphism $\varphi: (E_0, \nabla_{E_0}) \rightarrow (F,\nabla_F)$ of modules with formal integrable connection
on $(\Sp(K[[x_1,...,x_{n}]] \allowbreak /(x_1 \cdots x_{n})),M)/\Sp(K)$ 
from an object $(E_0, \nabla_0)$ in the category $$\widehat{MIC}((\Sp(K[[x_1,...,x_{n}]]/(x_1 \cdots x_{n})),M)/\Sp(K))^{\mathrm{nr}}$$ to 
an object $(F,\nabla_F)$ in the category $$\widehat{MIC}((\Sp(K[[x_1,...,x_{n}]]/(x_1 \cdots x_{n})),M)/\Sp(K)[u])'$$ belongs to the category $$\widehat{MIC}(\Sp(K[[x_1,...,x_{n}]]/(x_1 \cdots x_{n}))/\Sp(K))^{\mathrm{nr}}. $$
\item[(ii)] 
\end{comment}
If the category $\widehat{MIC}((\Sp(K[[x_1, ..., x_{n-1}]]/(x_1\cdots x_{n-1})),M)/\Sp(K)[u])'$ is abelian, 
the category $\widehat{MIC} \allowbreak ((\Sp(K[[x_1, ..., x_{n}]]/(x_1\cdots x_{n})),M)/\Sp(K)[u])'$ is also abelian. 
%\end{itemize}
\end{prop}

\begin{proof}
%We give a proof of (ii) first and give a proof of (i) briefly at the end. 
The proof is analogous to the proof of proposition \ref{inductivepass}. We use the same notation as in that proposition. Thus we denote by $B$ the ring $K[[x_1,\dots, x_{n}]]/(x_1\cdots x_{n})$ and by $A$ the ring $K[[x_1,\dots, x_{n-1}]]/(x_1\cdots x_{n-1})$.  
The map $h: A\rightarrow B$ defined by $x_i\mapsto x_i \, (i=1, \dots, n-2), 
x_{n-1}\mapsto x_{n-1}x_n$ induces a map of log schemes 
$(\Sp(B), M_B)\rightarrow (\Sp(A), M_A)$, where $M_A, M_B$ are as in the proof of proposition \ref{inductivepass}. 
Moreover, $h$ naturally induces the map 
$A[u]^{\wedge} = K[u][[x_1,\dots, x_{n-1}]]/(x_1\cdots x_{n-1}) \to 
K[u][[x_1,\dots, x_{n}]]/(x_1\cdots x_{n}) = B[u]^{\wedge}$. 
Also, let $\partial_1, \dots, \partial_n$ be as in remark \ref{notationlocalderivations2} and put $\partial := \partial_{n-1}$. 

For an object $(E,\nabla_E)$ in 
$\widehat{MIC}((\Sp(B), M_B)/\Sp(K)[u])'$, put 
$$\overline{E} := \{ e \in E \,|\, \nabla_E^N(\partial)(e) \to 0 \, (N \to \infty)\}, $$
where, on the right hand side, $E$ is endowed with $(x_1, ..., x_n)$-adic topology. 
Then $\overline{E}$ is stable by the action of 
$\nabla_E(\partial_i) \, (1 \leq i \leq n-2, i=n)$ and the action of $\nabla_E(\partial)$ which is 
locally topologically nilpotent.  
Then, it suffices to prove the following: firstly, $\overline{E}$ and the actions 
$\nabla_E(\partial_i) \, (1 \leq i \leq n-2, i=n)$ on it define an object (which we denote by 
$(\overline{E}, \nabla_{\overline{E}})$) in 
$\widehat{MIC}((\Sp(A), M_A)/\Sp(K)[u])'$. Secondly, 
the correspondence 
$(E,\nabla_E) \mapsto ((\overline{E}, \nabla_{\overline{E}}), \nabla_E(\partial_1))$ defines the 
functor 
$$ \widehat{MIC}((\Sp(B), M_B)/\Sp(K)[u])' \to 
\left\{ 
((F,\nabla_F),N) \,\left|\, 
\begin{aligned} 
& \text{$(F,\nabla_F) \in \widehat{MIC}((\Sp(A), M_A)/\Sp(K)[u])'$} \\
& \text{$N: (F,\nabla_F) \to (F,\nabla_F)$: a nilpotent endomorphism}
\end{aligned}
\right. 
\right\} $$ 
and it 
is an equivalence whose quasi-inverse is given by 
$((F,\nabla_F),N) \mapsto (E := F \otimes_A B, \nabla_E)$ 
with the action $\nabla_E(\partial_i) \, (1 \leq i \leq n-2, i=n)$ 
and $\nabla_E(\partial)$ defined by 
$\nabla_E(\partial_i) :=  {\rm id} \otimes d(\partial_i) + \nabla_F(\partial_i) \otimes {\rm id},$ 
$\nabla_E(\partial) :=  {\rm id} \otimes d(\partial) + N \otimes {\rm id}.$ 
(The reason is the same as that in lemma \ref{RS}.) 

To prove the above claim, it suffices to construct a basis of $\overline{E}$ over 
$A[u]^{\wedge}$ which is a basis of $E$ over $B[u]^{\wedge}$. 
First, we construct a basis $e_1, \dots, e_s$ of $E$ as $B[u]^{\wedge}$-module on which $\nabla_E(\partial)$ acts as a matrix $H_0$ with entries in 
$A[u]^{\wedge}$ such that, if we write $H_0 = \sum_{\g{k}} M_{\g{k}} \g{x}^{\g{k}}$ 
with entries of $M_{\g{k}}$ in $K[u]$, $M_0$ is strictly upper triangular. 
Noting the fact $d(\partial)(a)=0$ for any $a \in A[u]^{\wedge}$, the proof of this claim is perfectly analogous to the proof of CLAIM 1 of proposition \ref{inductivepass}. 
Next, we can proceed as in the proof 
 of CLAIM 2 of proposition \ref{inductivepass} and prove that 
$A[u]^{\wedge}e_1 + \cdots + A[u]^{\wedge}e_s = \overline{E}$. 
Hence we have the claim we want and so the proof is finished. 
\begin{comment}
We give a proof of (i). If we define $\overline{E}_0$ (resp. $\overline{F}$) 
as in the proof of proposition \ref{inductivepass} (resp. as in the proof of (ii) above), 
the morphism $\varphi$ induces the morphism $\overline{\varphi}: \overline{E}_0 \to 
\overline{F}$, which is regarded as a morphism of 
modules with formal integrable connection on $(\Sp(A),M_A)/\Sp(K)$ 
from an object in $\widehat{MIC}((\Sp(A),M_A)/\Sp(K))^{\mathrm{nr}}$ to 
an object in $\widehat{MIC}((\Sp(A),M_A)/\Sp(K)[u])'$. Hence 
${\rm Ker}(\varphi_0)$ belongs to $\widehat{MIC}((\Sp(A),M_A)/\Sp(K))^{\mathrm{nr}}$ 
by assumption. Also, $\varphi$ is written as the composite 
$$ E = \overline{E} \otimes_A B \xrightarrow{\overline{\varphi} \otimes {\rm id}}
\overline{F} \otimes_A B \rightarrow \overline{F} \otimes_{A[u]^{\wedge}} 
B[u]^{\wedge} = F. $$
Using the fact that $\overline{F}$ is a free $A[u]^{\wedge}$-module, we see 
that the second map is injective. So  
${\rm Ker}(\varphi) = {\rm Ker}(\varphi_0) \otimes_A B$. 
Thus we see that $\mathrm{Ker}(\varphi)$ is an object in  
$\widehat{MIC}((\Sp(B),M_B)/\Sp(K))^{\mathrm{nr}}$. 
So we finished the proof of (i).
\end{comment} 
\end{proof}

Now we consider the condition (iii) in definition \ref{categoriauu}. 

\begin{comment}
\begin{rmk}\label{equivconfinito}
Let $(E, \nabla_E)$ be an object of $MIC(X^{\times}/\Sp(K)[u])^{\mathrm{nr}}$, and let 
$\mathcal{O}_X[u]_{\leq i}, d$ be as in remark \ref{rmk4.4}. 
Then $(\mathcal{O}_X[u]_{\leq i}, d)$ is an object in $MIC(X^{\times}/\Sp(K))^{\mathrm{nr}}$. 
Let  $s_i:E_0\otimes_{\mathcal{O}_X}\mathcal{O}_X[u]_{\leq i}\rightarrow E$ be the restriction of the morphism $s$ of definition \ref{categoriauu} (iii) (with $(Y,L)$ replaced by 
$X^{\times}$) to $E_0\otimes_{\mathcal{O}_X}\mathcal{O}_X[u]_{\leq i}$. The condition (iii) in definition \ref{categoriauu} implies that $\bigcup_{i\in \mathbb{N}}\mathrm{Im}(s_i)=E$, and 
by the stability of $MIC(X^{\times}/\Sp(K))^{\mathrm{nr}}$ under tensor product and 
by proposition \ref{(i)general} (i), each $\mathrm{Im}(s_i)$ is an object in 
$MIC(X^{\times}/\Sp(K))^{\mathrm{nr}}$. So $(E, \nabla_E)$ is a union (colimit) of 
objects in $MIC(X^{\times}/\Sp(K))^{\mathrm{nr}}$. 

Conversely, let $(E, \nabla_E)$ be an object in $MIC(X^{\times}/\Sp(K)[u])'$ 
which is a colimit of objects $(E_i,\nabla_{E_i})$ in $MIC(X^{\times}/\Sp(K))^{\mathrm{nr}}$. 
Then we see that there exists some $i$ such that the local generators of $E$ 
as $\mathcal{O}_X[u]$-modules are contained in the image of $E_i$. Then 
the map $s_i: (E_i, \nabla_{E_i}) \rightarrow (E,\nabla_E)$ induces the surjective 
morphism ${\rm Im}(s_i) \otimes_{\mathcal{O}_X} \mathcal{O}_X[u] 
\rightarrow E$. Hence $(E,\nabla_E)$ is an object of 
$MIC(X^{\times}/\Sp(K)[u])^{\mathrm{nr}}$. 
\end{rmk}
\end{comment}

\begin{prop}\label{abelu}
The category $MIC(X^{\times}/\Sp(K)[u])^{\mathrm{nr}}$ is abelian. 
\end{prop}

\begin{proof}
It suffices to prove that, 
for a horizontal morphism $\varphi: (E,\nabla_E) \rightarrow (F,\nabla_F)$ 
in $MIC(X^{\times}/\Sp(K)[u])^{\mathrm{nr}}$, the kernel and the cokernel 
of $\varphi$ belong to $MIC(X^{\times}/\Sp(K)[u])^{\mathrm{nr}}$. 
By proposition \ref{(i)general} (ii), they belong to 
$MIC(X^{\times}/\Sp(K)[u])'$. Moreover, 
by definition, $(E,\nabla_E), (F,\nabla_F)$ are written as 
the colimit $\varinjlim_{i \in \mathbb{N}} (E_i, \nabla_{E_i})$, 
$\varinjlim_{i \in \mathbb{N}} (F_i, \nabla_{F_i})$
of objects in $MIC(X^{\times}/\Sp(K))^{\mathrm{nr}}$ indexed by $\mathbb{N}$, and 
by changing the index set suitably, the morphism $\varphi$ is written as 
the colimit of the morphisms $\varphi_i: (E_i, \nabla_{E_i}) \rightarrow 
(F_i, \nabla_{F_i})$. So the kernel of $\varphi$ is written as the 
colimit $\varinjlim_{i \in \mathbb{N}} {\rm Ker}\,\varphi_i$ of objects 
in $MIC(X^{\times}/\Sp(K))^{\mathrm{nr}}$ indexed by $\mathbb{N}$,
 and the same is true for the cokernel. 
%so we see easily (using also proposition \ref{(i)general} (i)) 
%that the kernel and the cokernel are written as 
%the colimit of objects in $MIC(X^{\times}/\Sp(K))^{\mathrm{nr}}$. 
Hence the kernel and the cokernel belong to $MIC(X^{\times}/\Sp(K)[u])^{\mathrm{nr}}$, 
as desired. 
\end{proof}

\begin{prop}Let us suppose that there exists a $K$-rational point $x\in X$. Then $MIC(X^{\times}/\Sp(K)[u])^{\mathrm{nr}}$ is a neutral Tannakian category over $K$.

\end{prop}
\begin{proof}
The category $MIC(X^{\times}/\Sp(K)[u])^{\mathrm{nr}}$ is an abelian category by 
proposition \ref{abelu}. 
We define the unit object of the category as the pair $(\mathcal{O}_X[u], d)$ and considering the tensor structure induced by the tensor structure of the category of the coherent $\mathcal{O}_X[u]$-modules with integrable connection, the category $MIC(X^{\times}/\Sp(K)[u])^{\mathrm{nr}}$ is an abelian tensor category. Moreover, it is a rigid abelian tensor category, thanks to condition (i) of definition \ref{categoriauu}.

We will prove in proposition \ref{quasiisoHirsch} below that $\mathrm{End}((\mathcal{O}_X[u], d))\cong H^0_{\mathrm{dR}}(X^{\times}/\Sp(K)^{\times}, (\mathcal{O}_X, d))$. Then, because there exists a $K$-rational point of $X$, $\mathrm{End}((\mathcal{O}_X[u], d))\cong K$ (as proven in proposition \ref{relativeTannakian}). 

Let $x$ be the $K$-rational point which exists by hypothesis. We define a functor 
$$\gamma_x:MIC(X^{\times}/\Sp(K)[u])^{\mathrm{nr}}\rightarrow \mathrm{Vec}_K$$ 
which is the composition of the functor which sends an object $(E, \nabla_E)$ to 
$E(x)/uE(x)$, where $E(x)$ is the fiber of $E$ at $x$. 
%the the fiber $E(x)$ of $E$ at $x$ and the functor which sends $E(x)$ to $E(x)/uE(x)$.
To prove that $\gamma_x$ is a fiber functor, it is enough to prove that it is an exact tensor functor because it will be automatically faithful thanks to \cite[Corollaire 2.10]{Del90}, 
and we see the exactness of the functor 
by the condition (i) of definition \ref{categoriauu}. So the proof is finished. 
%Since the functor $(E, \nabla_E) \mapsto E(x)$ is an exact tensor functor 
%and $E(x)$ is a finite free $K[u]$-module 
%by the condition (i) of definition \ref{categoriau}, 
%we see that the functor $(E, \nabla_E) \mapsto E(x)/uE(x)$
%we are left to prove that the functor which sends $E(x) \mapsto E(x)/uE(x)$ is an exact tensor functor. 
%Because 
%we have the isomorphism $E(x)/uE(x) \cong E(x)^{\nabla_{E(x)}}$ 
%by proposition \ref{K[u]/kabeliana}(iii), the above functor is exact because $E(x)\mapsto E(x)/uE(x)$ is right exact and $E(x)\mapsto E(x)^{\nabla_{E(x)}}$ is left exact.
\end{proof}

\begin{defi}\label{fundgroupu} Let $x$ be a $K$-rational point of $X$ and let 
$$\gamma_x:MIC(X^{\times}/\Sp(K)[u])^{\mathrm{nr}}\rightarrow \mathrm{Vec}_K$$ 
be the fiber functor introduced in the above proposition. 
We define the log algebraic fundamental group of $X^{\times}/\Sp(K)[u]$ with base point $x$ as the Tannaka dual of $MIC(X^{\times}/\Sp(K)[u])^{\mathrm{nr}}$, \emph{i.e.} 
$$\pi_{1}(X^{\times}/\Sp(K)[u], x):=G(MIC(X^{\times}/\Sp(K)[u])^{\mathrm{nr}}, \gamma_x)$$
with the same notation as in theorem \ref{Saavedra}.
\end{defi}

Let $t_u:MIC(X^{\times}/\Sp(K))^{\mathrm{nr}}\rightarrow MIC(X^{\times}/\Sp(K)[u])^{\mathrm{nr}}$ be the functor which sends $(E, \nabla_E)\in MIC(X^{\times}/\Sp(K))^{\mathrm{nr}}$ to $(E, \nabla_E)[u]$  in $MIC(X^{\times}/\Sp(K)[u])^{\mathrm{nr}}$, where $(E, \nabla_E)[u]$ consists of 
the $\mathcal{O}_{X}[u]$-module $E\otimes_{\mathcal{O}_X}\mathcal{O}_X[u]$ 
and the connection which sends $eu^i \, (e \in E, i \in \mathbb{N})$ to $\nabla(e)u^i+eidu^{i-1}$. 
For a $K$-rational point $x$ of $X$, the functor $t_u$ induces an homomorphism of fundamental groups
$$\pi(t_u):\pi_1(X^{\times}/\Sp(K)[u], x)\longrightarrow \pi_1(X^{\times}/\Sp(K), x).$$
\begin{lemma}\label{pituclosedimmersion}
The map $\pi(t_u)$ is a closed immersion.
\end{lemma}
\begin{proof}
It is enough to check that theorem \ref{criterion} (ii) is satisfied. 
Every object $(E, \nabla_E)$ in $MIC(X^{\times}/\Sp(K)[u])^{\mathrm{nr}}$ is written as 
the colimit $(E, \nabla_E) = \varinjlim_{i \in {\mathbb{N}}} (E_i, \nabla_{E_i})$ with 
$(E_i, \nabla_{E_i}) \in MIC(X^{\times}/\Sp(K))^{\mathrm{nr}}$. Then 
we have the surjection $\varinjlim_{i \in {\mathbb{N}}} t_u((E_i, \nabla_{E_i})) \to 
(E, \nabla_E)$ and so there exists some $i \in {\mathbb{N}}$ such that 
the map $t_u((E_i, \nabla_{E_i})) \to (E, \nabla_E)$ is surjective. 
In particular every object $(E, \nabla_E)\in MIC(X^{\times}/\Sp(K)[u])^{\mathrm{nr}}$ is a subquotient of 
an object in $MIC(X^{\times}/\Sp(K))^{\mathrm{nr}}$, as required by theorem \ref{criterion} (ii).
%For every object $(E, \nabla_E)\in MIC(X^{\times}/\Sp(K)[u])^{\mathrm{nr}}$ there exists a lattice (definition \ref{categoriau} (ii)), \emph{i.e.} an object $(E_0, \nabla_{E_0})\in MIC(X^{\times}/\Sp(K))^{\mathrm{nr}}$ such that $t_u((E_0, \nabla_{E_0}))$ surjects to $(E, \nabla_E)$. In particular every object $(E, \nabla_E)\in MIC(X^{\times}/\Sp(K)[u])^{\mathrm{nr}}$ is a subquotient of $t_u((E_0, \nabla_{E_0}))$, as required by theorem \ref{criterion} (i).
\end{proof}
Moreover, let $p_u:MIC(X^{\times}/\Sp(K)[u])^{\mathrm{nr}}\rightarrow MIC(X^{\times}/\Sp(K)^{\times})^{\mathrm{nr}}$ be the functor induced by the morphism which sends $u$ and $du$ to $0$. Then $p_u$ induces an homomorphism of fundamental groups:
$$\pi(p_u):\pi_1(X^{\times}/\Sp(K)^{\times}, x)\longrightarrow \pi_1(X^{\times}/\Sp(K)[u], x).$$

\begin{prop}\label{pipusurjective}
The map $\pi(p_u)$ is faithfully flat.
\end{prop}
\begin{proof}

Thanks to theorem \ref{Esnsequence} the sequence

\begin{equation*}
\pi_1(X^{\times}/\Sp(K)^{\times}, x)\xrightarrow{\pi(r)} \pi_1(X^{\times}/\Sp(K), x)\xrightarrow{\pi(f^{*}_{\mathrm{dR}})} \pi_1(\Sp(K)^{\times}/\Sp(K), \nu)\longrightarrow 1
\end{equation*}
is exact. Moreover, the first map factors as follows:

\begin{equation}\label{factorsviau}
\xymatrix{
\pi_1(X^{\times}/\Sp(K)^{\times}, x)\ar[r]^{\pi(r)}\ar[d]^{\pi(p_u)}& \pi_1(X^{\times}/\Sp(K), x)\\
\pi_1 (X^{\times}/\Sp(K)[u], x)\ar[ur]^{\pi(t_u)}.\\
}
\end{equation}

If we prove the inclusion $\pi(t_u)(\pi_1 (X^{\times}/\Sp(K)[u], x)) \subset \mathrm{Ker} (\pi(f^*_{\mathrm{dR}}))$, we see from the above diagram the following sequence of 
inclusions 
$$ \mathrm{Ker} (\pi(f^*_{\mathrm{dR}})) 
\subset
(\pi(t_u) \circ \pi(p_u))((\pi_1 (X^{\times}/\Sp(K)^{\times}, x)) 
\subset
\pi(t_u)(\pi_1 (X^{\times}/\Sp(K)[u], x)) \subset \mathrm{Ker} (\pi(f^*_{\mathrm{dR}})). $$
Since $\pi(t_u)$ is a closed immersion, this implies 
$\pi(p_u)(\pi_1 (X^{\times}/\Sp(K)^{\times}, x)) = 
\pi_1 (X^{\times}/\Sp(K)[u], x)$, hence the faithful flatness of $\pi(p_u)$.

So, we prove in what follows that $\pi(t_u)(\pi_1 (X^{\times}/K[u], x)) \subset \mathrm{Ker} (\pi(f^*_{\mathrm{dR}}))$. The composition of functors
$$
MIC(\Sp(K)^{\times}/\Sp(K))^{\mathrm{nr}}\xrightarrow{f^*_{\mathrm{dR}}}MIC(X^{\times}/\Sp(K))^{\mathrm{nr}}\xrightarrow{r}MIC(X^{\times}/\Sp(K)^{\times})^{\mathrm{nr}}
$$
sends every object of $MIC(\Sp(K)^{\times}/\Sp(K))^{\mathrm{nr}}$ to a finite sum of trivial ones. This implies that the composition 
\begin{align*}
MIC(\Sp(K)^{\times}/\Sp(K))^{\mathrm{nr}}\xrightarrow{f^*_{\mathrm{dR}}}MIC(X^{\times}/\Sp(K))^{\mathrm{nr}}& \xrightarrow{t_u}MIC(X^{\times}/\Sp(K)[u])^{\mathrm{nr}}\\ & \xrightarrow{p_u}MIC(X^{\times}/\Sp(K)^{\times})^{\mathrm{nr}}
\end{align*}
also sends every object of $MIC(\Sp(K)^{\times}/\Sp(K))^{\mathrm{nr}}$ to a finite sum of trivial ones. Since $p_u$, restricted to the image of $t_u$ is fully faithful (proposition \ref{quasiisoHirsch} below), the composition 
$$
MIC(\Sp(K)^{\times}/\Sp(K))^{\mathrm{nr}}\xrightarrow{f^*_{\mathrm{dR}}}MIC(X^{\times}/\Sp(K))^{\mathrm{nr}}\xrightarrow{t_u}MIC(X^{\times}/\Sp(K)[u])^{\mathrm{nr}} 
$$
sends every object of $MIC((K,N)/K)^{\mathrm{nr}}$ to a finite sum of the trivial ones. 
Hence $\pi(t_u)(\pi_1 (X^{\times}/\Sp(K)[u]))$ is  contained in $\mathrm{Ker} (\pi(f^*_{\mathrm{dR}}))$.
\end{proof}

\begin{prop}\label{exactu}
Let $x$ be a $K$-rational point of $X$. Then the sequence  
\begin{equation}\label{exactsequencewithu}
1\longrightarrow\pi_1(X^{\times}/\Sp(K)[u], x)\xrightarrow{\pi(t_u)} \pi_1(X^{\times}/\Sp(K), x)\xrightarrow{\pi(f^*_{\mathrm{dR}})} \pi_1(\Sp(K)^{\times}/\Sp(K), \nu)\longrightarrow 1
\end{equation}
is exact. 
\end{prop}
\begin{proof}
The homomorphism $\pi(t_u)$ is a closed immersion (lemma \ref{pituclosedimmersion}) and $\pi(f^*_{\mathrm{dR}})$ is faithfully flat  (proposition \ref{pifsurjective}). Moreover, by theorem \ref{Esnsequence}, $\mathrm{Ker}(\pi(f^*_{\mathrm{dR}}))=\pi(r)(\pi_1(X^{\times}/\Sp(K)^{\times}, x))$ and since the diagram \eqref{factorsviau} is commutative, $\mathrm{Ker}(\pi(f^*_{\mathrm{dR}}))=\pi(t_u)(\pi(p_u)(\pi_1(X^{\times}/\Sp(K)^{\times}, x)))$. Since $\pi(p_u)$ is faithfully flat (proposition \ref{pipusurjective}), $\mathrm{Ker}(\pi(f^*_{\mathrm{dR}}))=\pi(t_u)(\pi_1(X^{\times}/\Sp(K)[u], x))$ and the sequence \eqref{exactsequencewithu} is exact.
\end{proof}

\section{Injectivity of the first map}
In this section we study the injectivity of $\pi(r)$. The log point should be thought as a log algebraic analogue of the punctured open disk in the complex plane. It is then reasonable to expect that, if we can define the higher homotopy groups for the log point, they should be trivial, \emph{i.e.} it is natural to expect that $\pi(r)$ is injective. To make this intuition more precise, we recall how we can characterize the topological spaces for which the only non-vanishing homotopy group is the first. 

\begin{prop}\label{K(pi,1)}
Let $T$ be a connected topological manifold and t a point. Then $T$ is a $K(\pi,1)$ space, i.e., the fundamental group is the only non-vanishing homotopy group if and only if 
the map 
\begin{equation}\label{EilMac}
H^i(\pi_1^{\rm top}(T, t), A)\rightarrow H^i(T, \mathcal{A}),
\end{equation}
where $A$ is a representation of $\pi_1^{\rm top}(T, t)$ and $\mathcal{A}$ is the local system associated to $A$, is an isomorphism for every $i$. 
\end{prop}
We would like to translate the notion of $K(\pi,1)$ space in our context and prove that the log point $\Sp(K)^{\times}/\Sp(K)$ is a kind of $K(\pi,1)$ space in log algebraic geometry. We proceed as in \cite[section 2]{EsnHai06}. Let $\mathcal{V}$ be a representation of $\pi_1(\Sp(K)^{\times}/\Sp(K), \nu)$ 
(which can be regarded also as an object in $MIC(\Sp(K)^{\times}/\Sp(K))^{\mathrm{nr}}$) 
and let $\mathcal{K}$ be the trivial representation (which can be regarded also as 
the trivial object in $MIC(\Sp(K)^{\times}/\Sp(K))^{\mathrm{nr}}$); then we consider 
$$H^i(\pi_1(\Sp(K)^{\times}/\Sp(K), \nu), \mathcal{V}):=\mathrm{Ext}^i_{\mathrm{Rep}(\pi_1(\Sp(K)^{\times}/\Sp(K), \nu)}(\mathcal{K}, \mathcal{V}). $$
Let $\epsilon$ be an $i$-extension of $\mathcal{K}$ by $\mathcal{V}$ in 
${\mathrm{Rep}}(\pi_1(\Sp(K)^{\times}/\Sp(K), \nu)$. Then it induces an 
$i$-extension of the trivial object by $\mathcal{V}$ in 
$MIC(\Sp(K)^{\times}/\Sp(K))^{\mathrm{nr}}$, and so we have the 
connecting homomorphism 
$$\delta_{\epsilon}:H^0_{\mathrm{dR}}(\Sp(K)^{\times}/\Sp(K), \mathcal{K})\cong K\longrightarrow H^i_{\mathrm{dR}}(\Sp(K)^{\times}/\Sp(K), \mathcal{V}).$$ 
This induces the map of $K$-vector spaces
$$\delta^i:H^i(\pi_1(\Sp(K)^{\times}/\Sp(K), \nu), \mathcal{V})\longrightarrow  H^i_{\mathrm{dR}}(\Sp(K)^{\times}/\Sp(K), \mathcal{V})$$
defined by $\epsilon \mapsto \delta_{\epsilon}(1)$. The map $\delta^i$ is the analogue of \eqref{EilMac}, and the next proposition says that the log point is a kind of $K(\pi, 1)$ space in log algebraic geometry.
\begin{lemma} Let $\mathcal{V}$ be a representation of $\pi_1(\Sp(K)^{\times}/\Sp(K), \nu)$. 
The map $$\delta^i:H^i(\pi_1(\Sp(K)^{\times}/\Sp(K), \nu), \mathcal{V})\longrightarrow  H^i_{\mathrm{dR}}(\Sp(K)^{\times}/\Sp(K), \mathcal{V})$$
defined as above is an isomorphism for every $i$.
\end{lemma}
\begin{proof}
For $i=0,1$ the same proof as in \cite[proposition 2.2]{EsnHai06} works, so that $\delta^i$ is an isomorphism.
For $i\geq2$ $H^i_{\mathrm{dR}}(\Sp(K)^{\times}/\Sp(K), \mathcal{V})=0$ because $\omega^1_{\Sp(K)^{\times}/\Sp(K)}$ is of dimension $1$. On the other hand, $H^i(\mathbb{G}_a, \mathcal{K})=0$ for $i\geq 2$, thanks to \cite[remark 2 in page 71]{Jan87}. Since every $\mathcal{V}\in \mathrm{Rep}(\mathbb{G}_a)$ is a successive extension of $\mathcal{K}$, we can conclude that $H^i(\mathbb{G}_a, \mathcal{V})=0$ for every $i\geq2$ and for every $\mathcal{V}$ by long exact sequence of cohomologies.
\end{proof}
In the previous section, we described the kernel of the second map in \eqref{exactsequencewithoutu} as $\pi_1(X^{\times}/\Sp(K)[u], x)$. In this section we study the relation of $\pi_1(X^{\times}/\Sp(K)[u], x)$ with $\pi_1(X^{\times}/\Sp(K)^{\times}, x)$. In particular we prove that $\pi(p_u)$ induces an isomorphism 

$$\pi(p_u)^{\mathrm{tri}}:\pi_1(X^{\times}/\Sp(K)^{\times}, x)^{\mathrm{tri}} \xrightarrow{\cong} \pi_1(X^{\times}/\Sp(K)[u], x)^{\mathrm{tri}}, $$
where $\pi_1(X^{\times}/\Sp(K)^{\times}, x)^{\mathrm{tri}}$ and $\pi_1(X^{\times}/\Sp(K)[u], x)^{\mathrm{tri}}$ are the maximal geometrically protrigonalizable quotients of $\pi_1(X^{\times}/\Sp(K)^{\times}, x)$ and $\pi_1(X^{\times}/\Sp(K)[u], x)$ respectively, 
which we define now. 

\begin{defi}\label{trig}
(i)(\cite[definition 17.1]{Mil15})\, An algebraic group $G$ over $K$ is called trigonalizable if 
every nonzero representation of $G$ over $K$ is an iterated extension of 
$1$-dimensional representations. \\
(ii) \, An algebraic group $G$ over $K$ is called geometrically trigonalizable 
if $G \otimes_K \overline{K}$ (where $\overline{K}$ is an algebraic closure of $K$) 
is trigonalizable. \\ 
(iii) \, For an affine group scheme $G$ over $K$, we define 
$G^{\rm tri}$ by $G^{\rm tri} := \varprojlim_N G/N$, where 
$N$ runs through normal subgroups of $G$ such that $G/N$ is 
geometrically trigonalizable, and call it the maximal geometrically 
protrigonalizable quotient of $G$. \\
(iv) \, An affine group scheme $G$ over $K$ is called geometrically 
protrigonalizable if $G$ itself is the maximal geometrically 
protrigonalizable quotient of $G$. 
\end{defi}
In definition \ref{trig}(iii), the naturally induced morphism $G\rightarrow G^{\mathrm{tri}}$ is indeed surjective, 
because the Hopf algebra of $G^{\mathrm{tri}}$, 
which is the filtered direct limit of the Hopf algebras of $G/N$'s, injects into 
the Hopf algebra of $G$ by \cite[theorem 14.1]{Wat79} and 
\cite[theorem 13.3]{Wat79}. So it is reasonable to call $G^{\rm tri}$ as 
the `quotient'. 

For every affine group scheme $\Gamma$ over $K$ which is the inverse limit of geometrically trigonalizable 
algebraic groups and every morphism $\gamma: G\rightarrow \Gamma$ of group schemes over $K$, $\gamma$ uniquely factors through $G^{\rm tri}$, because 
the trigonalizability is stable under subquotients \cite[corollary 17.3]{Mil15}. 
Also, every map of affine group schemes $\gamma:G\rightarrow H$ naturally induces a map between their maximal geometrically protrigonalizable quotients  $\gamma^{\mathrm{tri}}:G^{\mathrm{tri}}\rightarrow H^{\mathrm{tri}}$, because $G\xrightarrow{\gamma} H\rightarrow H^{\mathrm{tri}}$ factors uniquely through $G^{\mathrm{tri}}$.

We prove the following proposition which we need later. 

\begin{prop}\label{gtri}
Let $K \subset L$ be an extension of fields of characteristic zero. 
Then, for an affine group scheme $G$ over $K$, $G$ is geometrically protrigonalizable if and only if 
$G \otimes_K L$ is geometrically protrigonalizable. 
\end{prop}

\begin{proof}
By using the fact that the trigonalizability is stable under subquotients \cite[corollary 17.3]{Mil15}, 
we see that it suffices to prove the following: for an algebraic group $G$ over $K$, 
$G$ is geometrically trigonalizable if and only if so is $G \otimes_K L$. 
Moreover, we may replace $K, L$ by their algebraic closures and it suffices to prove that 
$G$ is trigonalizable if and only if so is $G \otimes_K L$. 
 
When $G$ is trigonalizable, so is $G \otimes_K L$ by \cite[corollary 17.4]{Mil15}. 
We prove the converse. So assume that $G \otimes_K L$ is trigonalizable and 
take a representation $V$ of $G$. Then the action of $G$ on the set of $1$-dimensional 
subspaces of $V$ defines the action $G \times \mathbb{P}(V) \rightarrow \mathbb{P}(V)$. 
Let $R \subset G \times \mathbb{P}(V) \times \mathbb{P}(V)$ be the graph of this action 
and let $\pi: R \rightarrow \mathbb{P}(V) \times \mathbb{P}(V)$ be the natural projection. 
Let $R' \subset G \times \mathbb{P}(V)$ be $\pi^{-1}(\Delta(\mathbb{P}(V))$ 
(where $\Delta: \mathbb{P}(V) \hookrightarrow \mathbb{P}(V) \times \mathbb{P}(V)$ is 
the diagonal map) and let $\pi': R' \rightarrow \mathbb{P}(V)$ be the natural projection. 
Let $G = \coprod_{\alpha} G_{\alpha}$ be the decomposition into connected components of $G$ and 
let $\pi'_{\alpha}: R'_{\alpha} := R' \cap (G_{\alpha} \times \mathbb{P}(V)) \rightarrow \mathbb{P}(V)$ 
be the restriction of $\pi'$. Then the set 
$C_{\alpha} := \{x \in \mathbb{P}(V) \,|\, \dim {\pi'}_{\alpha}^{-1}(x) = \dim G_{\alpha} \}$ 
is closed in $\mathbb{P}(V)$ and so $C = \bigcap_{\alpha} C_{\alpha}$ (with reduced subscheme structure) 
is a closed subscheme of $\mathbb{P}(V)$. By construction, we have the equalities 
\begin{align*}
C(K) & = \{ x \in \mathbb{P}(V)(K) \,|\, \forall \alpha, \dim {\pi'}_{\alpha}^{-1}(x) = \dim G_{\alpha}\} \\ 
& =  \{ x \in \mathbb{P}(V)(K) \,|\, {\pi'}^{-1}(x) = G \} \\ 
& = \left\{ x \in \mathbb{P}(V)(K) \,\left|\, 
\begin{aligned}
& \text{the action $G \times \mathbb{P}(V) \to 
\mathbb{P}(V)$} \\ 
& \text{induces the map $G \times x \to x$} 
\end{aligned}
\right. \right\} \\ 
& = \left\{x \in \mathbb{P}(V)(K) \,\left|\, 
\begin{aligned}
& \text{the $1$-dimensional subspace of $V$} \\  
& \text{corresponding to $x$ is stable by $G$}
\end{aligned}
\right. 
\right\}, 
\end{align*}
where the second equality follows from the smoothness of $G$ over $K$, 
which is always true because $K$ is of characteristic zero. 
By the same  argument, we see the equality 
\begin{align*}
C(L) & = 
\left\{x \in \mathbb{P}(V)(L) \,\left|\, 
\begin{aligned}
& \text{the $1$-dimensional subspace of $V \otimes_K L$} \\ 
& \text{corresponding to $x$ is stable by $G \otimes_K L$}
\end{aligned}
\right. \right\}.  
\end{align*}
Hence $G$ (resp. $G \otimes_K L$) 
is trigonalizable if and only if $C(K)$ is nonempty (resp. only if $C(L)$ is nonempty) 
for any $V$. Also, if $C(L)$ is nonempty, $C$ is nonempty as a scheme and 
so $C(K)$ is nonempty because $C$ is of finite type over an algebraically closed field $K$. 
So the proof is finished. 
\end{proof}

\begin{rmk}\label{cattri}
For an affine group scheme $G$ over $K$, the full subcategory $\mathrm{Rep}(G)^{\rm tri}$ of $\mathrm{Rep}(G)$ 
consisting of representations $V$ such that $V \otimes_K \overline{K}$ is an 
iterated extension of $1$-dimensional representations of $G \otimes_K \overline{K}$ 
is a Tannakian subcategory of $\mathrm{Rep}(G)$. By definition, for every representation 
$(V, \rho)$ which factors through $\rho^{\rm tri}: G^{\rm tri} \rightarrow GL(V)$, 
$V \otimes_K \overline{K}$ is an 
iterated extension of $1$-dimensional representations of $G \otimes_K \overline{K}$, and 
so $V$ belongs to $\mathrm{Rep}(G)^{\rm tri}$. 
Conversely, if $(V,\rho)$ is a representation of $G$ which belongs to 
$\mathrm{Rep}(G)^{\rm tri}$, $\rho(G) \otimes_K \overline{K}$ is trigonalizable by \cite[proposition 17.2]{Mil15}. 
Hence $\rho(G)$ is geometrically trigonalizable and so $\rho$ factors through $G^{\rm tri}$. 
Therefore, we have an equivalence $\mathrm{Rep}(G^{\rm tri}) \cong 
\mathrm{Rep}(G)^{\rm tri}$, namely, $G^{\rm tri}$ is the Tannaka dual of 
$\mathrm{Rep}(G)^{\rm tri}$. 

By describing the above fact in terms of neutral Tannakian categories, 
we obtain the following: 
For a neutral Tannakian category $\mathcal{C}$ over $K$, the full subcategory 
$\mathcal{C}^{\rm tri}$ of $\mathcal{C}$ 
consisting of objects $V$ such that $V_{\overline{K}} \in \mathcal{C}_{\overline{K}}$ 
(where ${}_{\overline{K}}$ denotes the scalar extension to $\overline{K}$
 \cite[\S 4]{Del89}) 
is an iterated extension of rank $1$ objects of $\mathcal{C}_{\overline{K}}$ 
is a Tannakian subcategory of $\mathcal{C}$, and the Tannaka dual of $\mathcal{C}^{\rm tri}$ is equal to
the maximal geometrically protrigonalizable quotient of the Tannaka dual of $\mathcal{C}$. 
\end{rmk}

We consider as before the functor $p_u:MIC(X^{\times}/\Sp(K)[u])^{\mathrm{nr}}\rightarrow MIC(X^{\times}/\Sp(K)^{\times})^{\mathrm{nr}}$ induced by the morphism which sends $u$ and $du$ to $0$. Then $p_u$ induces a faithfully flat (proposition \ref{pipusurjective}) homomorphism of fundamental groups 
$$\pi(p_u):\pi_1(X^{\times}/\Sp(K)^{\times}, x)\longrightarrow \pi_1(X^{\times}/\Sp(K)[u], x)$$
which induces a faithfully flat homomorphism between the maximal geometrically protrigonalizable quotients:
$$\pi(p_u)^{\mathrm{tri}}:\pi_1(X^{\times}/\Sp(K)^{\times}, x)^{\mathrm{tri}}\longrightarrow \pi_1(X^{\times}/\Sp(K)[u], x)^{\mathrm{tri}}. $$

To see that $\pi(p_u)^{\mathrm{tri}}$ is an isomorphism, it is enough to prove that it is a closed immersion: according to \ref{criterion} (ii) we need to prove that every $(E, \nabla_E)\in MIC(X^{\times}/\Sp(K)^{\times})^{\rm nr,tri}$ (where ${}^{\rm tri}$ is as in remark \ref{cattri}) is isomorphic to a subquotient of an object of the form $p_u(F, \nabla_F)$ with $(F, \nabla_F) \in MIC(X^{\times}/\Sp(K)[u])^{\mathrm{nr, tri}}$. We will show more: we will show that every $(E, \nabla_E)\in MIC(X^{\times}/\Sp(K)^{\times})^{\mathrm{nr, tri}}$ is isomorphic to $p_u(F, \nabla_F)$ with $(F, \nabla_F) \in MIC(X^{\times}/\Sp(K)[u])^{\mathrm{nr, tri}}$. We look first at the case of $(E, \nabla_E)\in MIC(X^{\times}/\Sp(K)^{\times})^{\mathrm{nr}}$ of rank $1$. 

\begin{teo} \label{lifting} 
Let $(E, \nabla_E)$ be an object of $MIC(X^{\times}/\Sp(K)^{\times})^{\mathrm{nr}}$ of rank $1$. Then, there exists a unique object $(E,\widetilde{\nabla_{E}})$ in $MIC(X^{\times}/\Sp(K))^{\mathrm{nr}}$ up to canonical isomorphism such that the restriction functor  $r:MIC(X^{\times}/\Sp(K))^{\mathrm{nr}}\rightarrow MIC(X^{\times}/\Sp(K)^{\times})^{\mathrm{nr}}$ sends $(E,\widetilde{\nabla_{E}})$ to $(E, \nabla_E).$ 

\end{teo}

\begin{proof}
We say that $(E,\widetilde{\nabla_{E}})$ is a lift of $(E, \nabla_E)$ or that $(E,\widetilde{\nabla_{E}})$ lifts $(E, \nabla_E)$ if $(E,\widetilde{\nabla_{E}})$ is an object of $MIC(X^{\times}/\Sp(K))^{\mathrm{nr}}$ such that the functor $r:MIC(X^{\times}/\Sp(K))^{\mathrm{nr}}\rightarrow MIC(X^{\times}/\Sp(K)^{\times})^{\mathrm{nr}}$ sends $(E,\widetilde{\nabla_{E}})$ to $(E, \nabla_E).$ 
 
We start proving the uniqueness of the lift. Let $(E, \nabla_E)$ be an object of $MIC(X^{\times}/\Sp(K)^{\times})^{\mathrm{nr}}$ of rank $1$ and let $(E, \widetilde{\nabla_{i}})\in MIC(X^{\times}/\Sp(K))^{\mathrm{nr}}$ for $i=1,2$ be two different lifts of $(E, \nabla_E)$. Then $(\mathcal{O}_X, \widetilde{\nabla}):=(E, \widetilde{\nabla_1})^{\vee}\otimes (E, \widetilde{\nabla_2})$ is a lift of 
the unit object $(\mathcal{O}_X, d)$ in $MIC(X^{\times}/\Sp(K)^{\times})^{\mathrm{nr}}$. To prove the uniqueness, it is enough to prove that $\widetilde{\nabla}=\widetilde{d}$, where $\widetilde{d}$ is the exterior derivative 
$\mathcal{O}_X \rightarrow \omega^1_{X^{\times}/K}$: indeed, if $\widetilde{\nabla}=\tilde{d}$, then 
\begin{align*}
\mathrm{Hom}((E, \widetilde{\nabla_1}),(E, \widetilde{\nabla_2})) & = H^0_{\mathrm{dR}}(X^{\times}/\Sp(K),(\mathcal{O}_X, \widetilde{\nabla}))
= K \\ & = H^0_{\mathrm{dR}}(X^{\times}/\Sp(K)^{\times}, (\mathcal{O}_X, d)) =  \mathrm{Hom}((E, \nabla_1),(E, \nabla_2))\end{align*} 
and so there exists a unique isomorphism $\varphi:(E, \widetilde{\nabla_1})\rightarrow(E, \widetilde{\nabla_2})$ such that $r(\varphi)$ is the identity morphism on $(E,\nabla_E)$.  

To prove that $\widetilde{\nabla}=\widetilde{d}$, it is sufficient to prove that $\widetilde{\nabla}(1)=0$ in $\hat{\mathcal{O}}_{X, P}$ for any 
geometric point $P$ over a closed point of $X$. 
Thus we can work on $\hat{\mathcal{O}}_{X, P}=K[[x_1, \dots, x_n]]/(x_1 \cdots x_r)$ for some 
$r \geq 1$, where $K$ here is an algebraic closure of our original field $K$.  

Since $(\mathcal{O}_X, \widetilde{\nabla})$ is a lift of $(\mathcal{O}_X, d)$, we can write $\widetilde{\nabla}(1)=\alpha\sum_{i=1}^r\mathrm{dlog}x_i$ with $\alpha\in K[[x_1, \dots, x_n]]/(x_1 \cdots x_r)$. Since $\widetilde{\nabla}$ is integrable, 
\begin{equation}\label{integrabilitycondition}
0 = \widetilde{\nabla}\circ\widetilde{\nabla}(1) = \left(\sum_{i=1}^rx_i\frac{\partial \alpha}{\partial x_i}\mathrm{dlog}x_i+\sum_{i=r+1}^n\frac{\partial \alpha}{\partial x_i}dx_i\right)\wedge \left(\sum_{i=1}^r\mathrm{dlog}x_i\right).
\end{equation}
The equation \eqref{integrabilitycondition} implies that 
\begin{equation}\label{per1}
x_i\frac{\partial \alpha}{\partial x_i}=x_j\frac{\partial \alpha}{\partial x_j} \quad 
(1\leq i < j\leq r), \qquad \frac{\partial \alpha}{\partial x_i}=0 \quad (i>r). 
\end{equation}
The former equalities in \eqref{per1} imply that 
$$ x_i\frac{\partial \alpha}{\partial x_i} \in 
\bigcap_{i=1}^r x_i  K[[x_1, \dots, x_n]]/(x_1 \cdots x_r) = 0 \quad (1 \leq i \leq r), $$
thus $\alpha$ is constant with respect to the variables $x_i \, (1 \leq i \leq r)$. 
Also, the latter equalities in \eqref{per1} imply
that $\alpha$ is constant with respect to the variables $x_i \, (i > r)$. 
Hence $\alpha \in K$.  Then, since $\widetilde{\nabla}$ has nilpotent residues, 
$\alpha=0.$ 

Since we proved the uniqueness of the lift, we can work \'etale locally to prove the existence of the lift. So we suppose that $X=\Sp(A)$ and we fix an \'etale morphism $X=\Sp(A)\rightarrow \Sp (K[x_1, \dots, x_n]/(x_1\cdots x_r))$ 
such that the log structure on $X$ is induced by the monoid homomorphism 
$\mathbb{N}^r \to A$ which sends $e_i$ to the image of $x_i$ in $A$ for $i=1,\dots,r$. 
Also, we may assume that $E$ is free of rank $1$ on $X$. 

We introduce some notation which are in force along the proof. We put $[r]:=\{1, \dots, r\}$, $[n]:=\{1, \dots, n\}$. For a subset $I$ of $[r]$, we denote by $|I|$ the number of elements contained in $I$. For nonempty subsets $I,J$ of $[r]$, we will use the following subsets of $X$:
$$X_I:=\{P\in X |\, x_i(P)=0 \,\,\forall i \in I\},\quad X^{J}:=\{P\in X |\, x_j(P)\neq 0 \,\,\forall j\in J\}, \quad X_I^J:= X_I\cap X^J,$$
$$X_I^0:=X_I^{[r]\setminus I}=\{P\in X|\, x_i(P)=0 \,\,\forall i\in I, \,\,x_i(P)\neq 0 \,\,\forall i \in [r]\setminus I\}.$$

The $X_I$'s are closed in $X$ and the $X^J$'s are open. The $X_I^0$'s are smooth, locally closed in $X$ and set-theoretically $X=\bigsqcup_I X_I^0.$ For any fixed non empty $I\subseteq [r],$ we have that $\bigsqcup_{\emptyset \neq I'\subseteq I}X^0_{I'}=X^{[r]\setminus I}$.
We now construct a special cover of $X$: for $1\leq k\leq r,$ we define the open subscheme $X^{(k)}$ of $X$ by $X^{(k)}:=\bigsqcup_{1\leq |I|\leq k}X^0_I=\bigcup_{|I|=k}X^{[r]\setminus I}$. Then $X^{(0)}=\emptyset$ and $X^{(r)}=X$. 
We prove the existence of a lift of $(E, \nabla_E)$ on $X^{(k)}$ by induction on $k$. 

First we prove the case $k=1$. Since $f^*(N)_x\cong M_x$ for any 
geometric point $x$ of $X^{(1)}$, 
the composition 
$\omega^1_{X/\Sp(K)|X^{(1)}} 
\rightarrow \omega^1_{X^{\times}/\Sp(K)|X^{(1)}} 
\rightarrow \omega^1_{X^{\times}/\Sp(K)^{\times}|X^{(1)}}
$ is an isomorphism. Hence a lift of 
$\nabla_E:E_{|X^{(1)}}\rightarrow E_{|X^{(1)}}\otimes \omega^1_{X^{\times}/\Sp(K)^{\times}|X^{(1)}}$ is given by $\widetilde{\nabla}_E:E_{|X^{(1)}}\rightarrow E_{|X^{(1)}}\otimes \omega^1_{X^{\times}/\Sp(K)^{\times}|X^{(1)}} \cong 
E_{|X^{(1)}}\otimes \omega^1_{X/\Sp(K)|X^{(1)}} 
\rightarrow E_{|X^{(1)}}\otimes \omega^1_{X^{\times}/\Sp(K)|X^{(1)}}$.

Then we proceed to prove that, if $(E, \nabla_E)$ is liftable on $X^{(k-1)}$, then it is liftable on $X^{(k)}.$
To prove this claim we decompose $X^{(k)}$ further. We order the subsets of $[r]$ of cardinality $k$ as $\{I\subseteq [r]| |I|=k\}=\{I_1, \dots, I_m\}$, and we define $X^{(k)}_l:=X^{(k-1)}\sqcup \bigsqcup_{i=1}^l X^0_{I_i}=X^{(k-1)}\cup \bigcup_{i=1}^{l}X^{[r]\setminus I_i}$ 
for $0 \leq l \leq m$. The set $X^{(k)}_l$ is open in $X$ and $X^{(k)}_0=X^{(k-1)}$, $X^{(k)}_m=X^{(k)}$. We proceed again by induction, proving that if $(E, \nabla_E)$ is liftable on $X^{(k)}_{l-1}$, then it is liftable on $X_l^{(k)}.$ 

\begin{lemma}
The set $X^0_{I_l}$ is closed in $X^{(k)}_l$.
\end{lemma}
\begin{proof}
The set $X^{(k)}_l$ can be covered by $\bigcup_J X^J$ where $J$ runs through the sets of non empty $J\subseteq [r]$ with $|J|\geq r-k+1$ or $J=[r]\setminus I_i$ for $1\leq i\leq l$.

If $|J|\geq r-k+1$ or $J=[r]\setminus I_i$ for $1\leq i\leq l-1$, then $J\nsubseteq  [r]\setminus I_l$ which implies that $J\cap I_l\neq \emptyset$.
Hence $X^0_{I_l}\cap X^J=\emptyset $ which is closed in $X^J.$ If $J=[r]\setminus I_l$, then $X^0_{I_l}\cap X^J=X^J_{I_l}$ which is closed in $X^J$.
\end{proof}

Since $X_{I_l}^{0}$ is locally closed in $X$ and closed in $X_l^{(k)}$, which is open in $X$, the completion $\widehat{X_{|X^0_{I_l}}}$ is isomorphic to $\widehat{X^{(k)}_{l|X^0_{I_l}}}$. 
We denote by $Y^{0}_{I_{l}}$ the scheme 
$$Y^0_{I_l}:=\Sp(\Gamma(\widehat{X_{|X^0_{I_l}}}, \mathcal{O}))\cong \Sp \left(\frac{\bar{A}[x_i^{-1}]_{i\in [r]\setminus I_l}[[x_i]]_{i\in I_l}}{\prod_{i=1}^rx_i}\right)$$
where $\bar{A}:=A/(x_i)_{i\in I_l}$.
There exists a canonical map $Y^0_{I_l}\rightarrow X_l^{(k)}$, so that we have an fpqc covering of $X^{(k)}_{l}$ given by $X^{(k)}_{l-1}\coprod Y^0_{I_l}=\bigcup_{J}X^J\coprod Y^{0}_{I_l}$ where $J$ runs through the sets of non empty $J\subseteq [r]$ with $|J|\geq r-k+1$ or $J=[r]\setminus I_i$ for $1\leq i\leq l-1.$

By induction hypothesis for each $J\subseteq [r]$ with $|J|\geq r-k+1$ or $J=[r]\setminus I_i$ for $1\leq i\leq l-1$, $(E, \nabla_E)$ is liftable on $X_J$ and the lifts are compatible. 
On $Y^0_{I_l}\cong \Sp \left(\frac{\bar{A}[x_i^{-1}]_{i\in [r]\setminus I_l}[[x_i]]_{i\in I_l}}{\prod_{i\in I_l}x_i}\right) $ with log structure defined as the pullback of that on $X$, 
$(E, \nabla_E)$ is the pullback of a module with integrable connection on $\Sp \left(\bar{A}[x_i^{-1}]_{i\in [r]\setminus I_l}\right)/\Sp(K)$ (with trivial log structures) by lemma \ref{descentrk1} below, and so $(E, \nabla_E)$ is liftable on $Y^0_{I_l}$.

Hence it suffices to prove that the lifts are canonically isomorphic on $X^J\times_{X_l^{(k)}}Y^0_{I_l}$ (with log structure defined as the pullback of that on $X$). As in the proof of the uniqueness, it suffices to prove that any lift $(\mathcal{O}, \widetilde{\nabla})$ of $(\mathcal{O}, d)$ is trivial. Note that 
\begin{equation}\label{embedded1}
\Gamma(X^J\times_{X_l^{(k)}}Y^0_{I_l}, \mathcal{O})=\left(\frac{\bar{A}[x_i^{-1}]_{i\in [r]\setminus I_l}[[x_i]]_{i\in I_l}}{\prod_{i=1}^rx_i}\right)[x_i^{-1}]_{i\in J} 
= \left(\frac{\bar{A}[x_i^{-1}]_{i\in [r]\setminus I_l}[[x_i]]_{i\in I_l}}{\prod_{i \in I_l}x_i}\right)[x_i^{-1}]_{i\in I_l \cap J}.  
\end{equation}
Let us note that $Z=\Sp(\bar{A}[x_i^{-1}]_{i\in [r]\setminus I_l})$ is \'etale over $\Sp(K[x_i]_{i\in [n]\setminus I_l}[x_i^{-1}]_{i\in [r]\setminus I_l})$. Let $P\in Z$ be a closed point. Then the completion $\hat{\mathcal{O}}_{Z,P}$ has the form $K(P)[[y_{P,i}]]_{i\in[n]\setminus I_l},$ where $K(P)$ is the residue field of $Z$ at $P$, 
$y_{P,i}$'s \,$(1 \leq i \leq n, i \notin I_l)$ are local parameters with $dy_{P,i}=u_i\mathrm{dlog}x_i$ for $1\leq i\leq r$, $i\notin I_l$ with some $u_i\in \hat{\mathcal{O}}_{Z, P}^{\times}$ and $dy_{P,i}=dx_i$ for $r+1\leq i\leq n$. 

Using these notations, the ring in (\ref{embedded1}) can be embedded in 
\begin{equation}\label{embedded2}
\left( \prod_{P\in Z}\frac{K(P)[[y_{P,i}]]_{i\in [n]\setminus I_l}[[x_i]]_{i\in I_l} 
}{\prod_{i \in I_l} x_i} \right) [x_i^{-1}]_{i \in I_l \cap J} 
\hookrightarrow 
\prod_{P\in Z}\frac{K(P)[[y_{P,i}]]_{i\in [n]\setminus I_l}[[x_i]]_{i\in I_l} 
[x_i^{-1}]_{i \in I_l \cap J}}{\prod_{i \in I_l \setminus J} x_i}.   
\end{equation}

We will prove that $\widetilde{\nabla}(1)=0$ in $\frac{K(P)[[y_{P,i}]]_{i\in [n]\setminus I_l}[[x_i]]_{i\in I_l }[x_i^{-1}]_{i\in I_l \cap J}}{\prod_{i\in I_l \setminus J}x_i}$.
To simplify the notation, we put $y_i := y_{P,i}$ in the following. 

Let us suppose that $\widetilde{\nabla}(1)=\alpha\sum_{i=1}^r\mathrm{dlog}x_i;$ the  integrability $\widetilde{\nabla}\circ \widetilde{\nabla}(1)=0$ implies that $d\alpha\wedge (\sum_{i=1}^r\mathrm{dlog}x_i)=0.$ We calculate $d\alpha$ in terms of basis of $1$-differentials:
\begin{align*}
d\alpha& =\sum_{i \in I_l} x_i \frac{\partial \alpha}{\partial x_i} \mathrm{dlog}x_i+
\sum_{i \in [r] \setminus I_l} u_i \frac{\partial \alpha}{\partial y_i} u_i^{-1}dy_i + 
\sum_{i=r+1}^n \frac{\partial \alpha}{\partial y_i} dy_i \\ & = 
\sum_{i \in I_l} x_i \frac{\partial \alpha}{\partial x_i} \mathrm{dlog}x_i+
\sum_{i \in [r] \setminus I_l} u_i \frac{\partial \alpha}{\partial y_i} \mathrm{dlog}x_i + 
\sum_{i=r+1}^n \frac{\partial \alpha}{\partial y_i} dx_i. 
\end{align*}
Since $d\alpha\wedge (\sum_{i=1}^r\mathrm{dlog}x_i)=0,$ we see that the 
elements 
$ x_i \frac{\partial \alpha}{\partial x_i} \, (i \in I_l),  
 u_i \frac{\partial \alpha}{\partial y_i} \, (i \in [r] \setminus I_l) $
are the same and $\frac{\partial \alpha}{\partial y_i} = 0$ for $i > r$. 
From the latter assertion, we see that $\alpha$ is constant with respect to 
the variables $y_i \, (i > r)$. From the former assertion for $i \in I_l$, we see that 
$$ x_i \frac{\partial \alpha}{\partial x_i}  \in 
\bigcap_{j \in I_l} x_j
\frac{K(P)[[y_i]]_{i\in [n]\setminus I_l}[[x_i]]_{i\in I_l }[x_i^{-1}]_{i\in I_l \cap J}}{\prod_{i\in I_l \setminus J}x_i} = 0 \quad (i \in I_l), $$
and so $\alpha$ is constant with respect to the variables $x_i \,(i \in I_l)$. 
Thus $\alpha \in K(P)[[y_i]]_{i \in [r] \setminus I_l}$ and  
$u_i \frac{\partial \alpha}{\partial y_i} = 0 \, (i \in [r] \setminus I_l)$, hence 
$\alpha \in K(P)$. 
Then, thanks to the nilpotent residues condition we conclude that $\alpha=0$. 
So the proof is finished (modulo lemma \ref{descentrk1} below). 

\end{proof}

\begin{lemma}\label{descentrk1}
Let $C= \frac{\bar{A}[x_i^{-1}]_{i\in [r]\setminus I_l}[[x_i]]_{i\in I_l}}{\prod_{i\in I_l}x_i}$, $B=\bar{A}[x_i^{-1}]_{i\in [r]\setminus I_l}$ be as above and let $M$ be the log structure on $\Sp(C)$ associated to $\mathbb{N}^r \rightarrow C; \, e_i \mapsto x_i$. 
Let $(E, \nabla_E)$ be an object in $\widehat{MIC}((\Sp(C), M)/\Sp(K)^{\times})^{\mathrm{nr}}$ 
such that $E$ is free of rank $1$. Then there exists $(E_1, \nabla_{E_1})$ in $MIC(\Sp(B)/\Sp(K))$ 
with $E_1$ free of rank $1$ 
such that $(h, g)_{\mathrm{dR}}^*(E_1, \nabla_{E_1})\cong (E, \nabla_E)$, where $(h, g)_{\mathrm{dR}}^*$ is the pullback with respect to 
the morphisms $h,g$ in the diagram below: 
$$
\xymatrix{
(\Sp(C), M)\ar[r]^-h\ar[d]&\Sp(B)\ar[d]\\
\Sp(K)^{\times}\ar[r]^-g&\Sp(K). 
}
$$
In particular, $(E, \nabla_E)$ is liftable to an object in 
$\widehat{MIC}((\Sp(C), M)/\Sp(K))^{\mathrm{nr}}$. 
\end{lemma}
\begin{proof} We can suppose that $I_l=\{1, \dots, k\}$. For every $i\in I_l$, 
we denote by $C_i$ the ring $B[[x_1, \dots, x_i]]/(x_1 \cdots x_i)$ and 
by $M_i$ the log structure on $\Sp(C_i)$ induced by 
$\mathbb{N}^i \to C_i; e_j \mapsto x_j \, (1 \leq j \leq i)$. 
(Then $(\Sp(C_k),M_k) = (\Sp(C),M)$.) 
We regard $C_{i}$ as a $C_{i-1}$-algebra via the map $h_i: C_{i-1}\rightarrow C_i$ defined by $x_j\mapsto x_j \, (j=1, \dots, i-2)$,  $x_{i-1}\mapsto x_{i-1}x_i$. This induces a morphism of log schemes 
$(\Sp(C_i),M_i) \to (\Sp(C_{i-1}),M_{i-1})$. 

Note that $\Omega^1_{B/K}$ is a free $B$-module of finite rank. So we take an isomorphism  $\Omega^1_{B/K}\cong \bigoplus_{t=1}^s B b_t$ by fixing a basis $\{b_t\}_{t=1,\dots,s}$ of 
$\Omega^1_{B/K}$. 
Then $\hat{\omega}^1_{(\Sp(C_i), M_i)/\Sp(K)^{\times}}$ is 
isomorphic to the $C_i$-module $$\bigoplus_{t=1}^s C_ib_t\oplus \frac{C_i\mathrm{dlog}x_1\oplus\dots\oplus C_i\mathrm{dlog}x_i}{C_i \cdot \sum_{j=1}^i\mathrm{dlog}x_j}.$$
We consider the derivations $\{\partial_i\}_{i=1, \dots, k-1}$ of $(C_k, M_k)/\Sp(K)^{\times}$ as in remark \ref{notationlocalderivations}, \emph{i.e.} $\partial_i$ is an element of $\mathrm{Hom}(\hat{\omega}^1_{(\Sp(C_k), M_k)/\Sp(K)^{\times}}, C_k)$ which sends $\mathrm{dlog}x_i$ to $1$,  $\mathrm{dlog}x_{i+1}$ to $-1$, 
$\mathrm{dlog}{x_j}$ to $0$ for every $j\neq i, i+1$ and $b_t$ to $0$ for every $t=1, \dots, s$.  The derivation $\partial_{i-1}$ can be seen as a generator of the rank $1$ $C_i$-module $\mathrm{Hom}(\widehat{\omega}^1_{(\Sp(C_i),M_i)/(\Sp(C_{i-1}),M_{i-1})}, C_i)$.

To prove the lemma, it suffices to prove that, 
for an object $(E, \nabla_E)$ in $\widehat{MIC}((C_k, M_k)/\Sp(K)^{\times})^{\mathrm{nr}}$
such that $E$ is free of rank $1$,  the following two assertions hold by induction on $k$: 
%such that $E$ is free of rank $1$. To prove the lemma, it suffices 
%to prove the following two assertions by induction on $n$. 
\medskip 

\noindent
(1) \, There exists $(E_1, \nabla_{E_1})$ in $MIC(\Sp(B)/\Sp(K))$ 
with $E_1$ free of rank $1$ such that 
$(E, \nabla_E)$ is the pullback of $(E_1, \nabla_{E_1})$ to $(C_k, M_k)/\Sp(K)^{\times}$. \\
(2) \, Any endomorphism $\varphi: (E, \nabla_E) \to (E, \nabla_E)$ is the multiplication 
by some element in $B$. 
\medskip 

(In fact, the assertion (1) is the same as the statement of the lemma. 
However, we need to prove also the assertion (2) in order that 
the induction works.) 
In the case $k=1$, the assertions are true because $B = C_1$ as rings and 
$MIC((\Sp(C_1),M_1)/\Sp(K)^{\times}) \cong MIC(\Sp(B)/\Sp(K))$. 
So we prove the assertions in general case. Put $\partial := \partial_{k-1}$.  
Take $H \in C_k$ with $\nabla_E(\partial) = d(\partial) + H$ and write 
$H = \sum_{\g{k} \in \Gamma} H_{\g{k}} x^{\g{k}} \, (H_{\g{k}} \in B)$ with 
$\Gamma$ the index set as in the proof of proposition \ref{inductivepass}. 
Then, by the assumption that $(E,\nabla_E)$ has nilpotent residues, 
$H_{0}$ is sent to $0$ by any homomorphism $B \to \overline{K}$ 
(where $\overline{K}$ is the algebraic closure of $K$) over $K$. 
Because $B$ is reduced and Jacobson, we conclude that 
$H_0 = 0$. Then, by the same argument as the proof of CLAIM 1 in 
proposition \ref{inductivepass}, 
we can prove the existence of a basis $e$ of $E$ as $C_k$-module such that 
$\nabla_E(\partial)$ acts on it by the multiplication of an element 
$M = \sum_{\g{k} \in \Gamma'} M_{\g{k}} x^{\g{k}} \in C_{k-1} \, (M_{\g{k}} \in B)$
with $M_0 = 0$. Also, by the same argument as the proof of CLAIM 2 in 
proposition \ref{inductivepass}, we see that
\begin{equation*}\label{eq:cn-1}
\overline{E} := C_{n-1}e = \{ e' \in E \,|\, \nabla_E^N(\partial)(e') \to 0 \, (N \to \infty)\}. 
\end{equation*}
As in the proof of proposition \ref{inductivepass}, 
$\overline{E}$ is stable by the action of $\nabla_E(\partial_i) \, (1 \leq i \leq k-2)$ 
and with this action, $\overline{E}$ defines an object $(\overline{E}, \nabla_{\overline{E}})$ 
in $MIC((\Sp(C_{k-1}), M_{k-1})/\Sp(K)^{\times})^{\rm nr}$. Moreover, 
$\nabla_E(\partial)$ induces an endomorphism of $(\overline{E}, \nabla_{\overline{E}})$ 
which is the multiplication by $M$. 
By induction hypothesis (the assertion (2) for $(\overline{E}, \nabla_{\overline{E}})$), 
$M$ belongs to $B$. Since $M_0 = 0$, we conclude that $M=0$. 
Thus $(E,\nabla_E)$ is the pullback of $(\overline{E}, \nabla_{\overline{E}})$ to 
$(\Sp(C_{k}), M_{k})/\Sp(K)^{\times}$. Since $(\overline{E}, \nabla_{\overline{E}})$ is 
the pullback of an object $(E_1, \nabla_{E_1})$ in $MIC(\Sp(B)/\Sp(K))$ 
with $E_1$ free of rank $1$ to $(C_{k-1}, M_{k-1})/\Sp(K)^{\times}$ by induction hypothesis, 
we conclude that $(E,\nabla_E)$ is 
the pullback of $(E_1, \nabla_{E_1})$ to $(C_{k}, M_{k})/\Sp(K)^{\times}$. 
So the proof of the assertion (1) is finished. 

We prove the assertion (2). By definition of $(\overline{E}, \nabla_{\overline{E}})$, 
we see that any endomorphism $\varphi: (E, \nabla_E) \to (E, \nabla_E)$
induces an endomorphism on $(\overline{E}, \nabla_{\overline{E}})$, 
and it is the multiplication of some element in $B$ by induction hypothesis. 
Hence so is $\varphi$. So the proof of the assertion (2) is also finished. 
\end{proof}

To prove the property (ii) of theorem \ref{criterion} for iterated extensions of rank $1$ objects, 
 we need the following two results: 
first we prove that the category $MIC(X^{\times}/\Sp(K)[u])^{\mathrm{nr}}$ is closed under extensions (theorem \ref{extensionsoflattices}) and then prove that we can compare extensions of objects in $MIC(X^{\times}/\Sp(K)[u])^{\mathrm{nr}}$ with those in $MIC(X^{\times}/\Sp(K)^{\times})^{\mathrm{nr}}$ in certain case (proposition \ref{quasiisoHirsch}). 
\begin{teo}\label{extensionsoflattices}
The category $MIC(X^{\times}/\Sp(K)[u])^{\mathrm{nr}}$ is closed under extensions 
in the category of coherent $\mathcal{O}_X[u]$-modules with integrable log connections. 
\end{teo}
\begin{proof}
Let $(F, \nabla_{F})$ and $(G, \nabla_G)$ be objects in $MIC(X^{\times}/\Sp(K)[u])^{\mathrm{nr}}$ and let us suppose that there exists an exact sequence of $\mathcal{O}_X[u]$-modules with integrable connection:
\begin{equation}\label{eq:exseqext}
0\longrightarrow (F, \nabla_F) \longrightarrow (E, \nabla_E) \longrightarrow (G, \nabla_{G}) \longrightarrow 0. 
\end{equation}
It is easy to see that $(E,\nabla_E)$ satisfies the conditions (i), (ii) of 
definition \ref{categoriauu}. Hence, it suffices to see that $(E,\nabla_E)$ is written as a colimit 
of objects in $MIC(X^{\times}/\Sp(K))^{\rm nr}$ indexed by $\mathbb{N}$. Let us write 
$(F,\nabla_F), (G,\nabla_G)$ as colimits of objects in 
 $MIC(X^{\times}/\Sp(K))^{\rm nr}$ indexed by $\mathbb{N}$: 
$$ 
(F,\nabla_F) = \varinjlim_{i \in \mathbb{N}} (F_i, \nabla_{F_i}), \quad 
(G,\nabla_G) = \varinjlim_{j\in \mathbb{N}} (G_j, \nabla_{G_j}). $$  
Also, let 
\begin{equation}\label{eq:exseqext2}
0\longrightarrow (F, \nabla_F) \longrightarrow (E_j, \nabla_{E_j}) \longrightarrow 
(G_j, \nabla_{G_j}) \longrightarrow 0. 
\end{equation}
be the pullback of \eqref{eq:exseqext} by $(G_j,\nabla_{G_j}) \to (G,\nabla_G)$. 
%Then it suffices to see that $(E_j, \nabla_{E_j})$ is written as a colimit of 
%objects in $MIC(X^{\times}/\Sp(K))^{\rm nr}$. 
The exact sequence \eqref{eq:exseqext2} defines the extension class 
$[(E_j,\nabla_{E_j})]$ in 
\begin{align*}
& {\rm Ext}^1((G_j,\nabla_{G_j}), (F,\nabla_F)) = 
H^1_{\rm dR}(X^{\times}/\Sp(K), (G_j,\nabla_{G_j})^{\vee} \otimes (F,\nabla_F)) \\ 
=\, & \varinjlim_i 
H^1_{\rm dR}(X^{\times}/\Sp(K), (G_j,\nabla_{G_j})^{\vee} \otimes (F_i,\nabla_{F_i})) 
=  \varinjlim_i {\rm Ext}^1((G_j,\nabla_{G_j}), (F_i,\nabla_{F_i})). 
\end{align*}
Here, the first and the third equalities follow from the local freeness of $G_j$. 
Also, the second equality is straightforward in the case $X$ is affine and in general 
it is reduced to the affine case by using \v{C}ech cohomology for an open affine cover. 
Hence there exists a system of extension classes $[(E_{j,i}, \nabla_{E_{j,i}})] \in 
{\rm Ext}^1((G_j,\nabla_{G_j}), (F_i,\nabla_{F_i}))$ for some $i = i(j) \in \mathbb{N}$ 
which induces $[(E_j,\nabla_{E_j})]$. Also, we can take the indices $i = i(j) \in \mathbb{N}$ 
so that the map $j \mapsto i(j)$ is strictly increasing and that 
the extension classes 
\begin{align*}
& [(E_{j,i(j)}, \nabla_{E_{j,i(j)}})] \in 
{\rm Ext}^1((G_j,\nabla_{G_j}), (F_{i(j)},\nabla_{F_{i(j)}})), \\ 
& [(E_{j+1,i(j+1)}, \nabla_{E_{j+1,i(j+1)}})] \in 
{\rm Ext}^1((G_{j+1},\nabla_{G_{j+1}}), (F_{i(j+1)},\nabla_{F_{i(j+1)}}))
\end{align*} are compatible 
in the sense that they have the same image in 
${\rm Ext}^1((G_j,\nabla_{G_j}), (F_{i(j+1)},\nabla_{F_{i(j+1)}}))$. 
Then we see that $\{(E_{j,i(j)}, \nabla_{E_{j,i(j)}})\}_{j \in \mathbb{N}}$ forms an inductive system in 
$MIC(X^{\times}/\Sp(K))^{\rm nr}$ whose colimit is $(E,\nabla_{E})$. Hence we are done. 
\end{proof} 

\begin{prop}\label{quasiisoHirsch}
Let $(E, \nabla_E)$ be an object of $MIC(X^{\times}/\Sp(K))^{\mathrm{nr}}$ and let $(E, \nabla_E)[u]=t_u((E, \nabla_E))$ as defined in section 4. Then the functor $p_u:MIC(X^{\times}/\Sp(K)[u])^{\mathrm{nr}}\rightarrow MIC(X^{\times}/\Sp(K)^{\times})^{\mathrm{nr}}$ defined in section 4 induces a map from the de Rham complex of $(E, \nabla_E)[u]$, denoted by $\mathrm{DR}((E, \nabla_E)[u])$, to the de Rham complex of $p_u((E, \nabla_E)[u])$, denoted by $\mathrm{DR}(p_u((E, \nabla_E)[u]))$, and it is a quasi-isomorphism. 
\end{prop}
\begin{proof} The proof is inspired by \cite[lemma 6]{KimHai04}. 
We can regard the de Rham complex $\mathrm{DR}((E, \nabla_E)[u])$ as the total complex of the bicomplex 
\begin{equation}\label{bicomplexforE}
\xymatrix
{\vdots\ar[d]&\vdots\ar[d]&\vdots\ar[d]\\ 
0\ar[d]\ar[r]&0\ar[d]\ar[r]&E u^2\ar[d]\ar[r]&\cdots\\
0\ar[d]\ar[r]&Eu\ar[r]\ar[d]                                    &E\otimes\omega^1_{X^{\times}/\Sp(K)}u\ar[d]\ar[r]&\cdots\\
E\ar[r]       & E\otimes \omega^1_{X^{\times}/\Sp(K)}\ar[r]&E\otimes \omega^2_{X^{\times}/\Sp(K)}\ar[r]&\cdots, 
}
\end{equation}
where the vertical maps are given locally by $\alpha u^i\mapsto \alpha \wedge iu^{i-1}du$ and the horizontal maps are induced by the differentials of the de Rham complex of $(E, \nabla_E)$.
The cohomology of the total complex of (\ref{bicomplexforE}) gives the cohomology of $\mathrm{DR}((E, \nabla_E)[u])$. We first prove that all the columns of (\ref{bicomplexforE}) are exact except at the last term. Let us consider the column
$$
E\otimes \omega^{j-1}_{X^{\times}/\Sp(K)}u^{i+1}\xrightarrow{\gamma_{j-1}} E\otimes \omega^j_{X^{\times}/\Sp(K)}u^i\xrightarrow{\gamma_j} E\otimes \omega^{j+1}_{X^{\times}/\Sp(K)}u^{i-1}.
$$
It is enough to study the exactness of the maps 
$$
\omega^{j-1}_{X^{\times}/\Sp(K)}u^{i+1}\xrightarrow{\delta_{j-1}} \omega^j_{X^{\times}/\Sp(K)}u^i\xrightarrow{\delta_j} \omega^{j+1}_{X^{\times}/\Sp(K)}u^{i-1},
$$
because $\gamma_j=\mathrm{id}\otimes \delta_j $ for every $j$ and $E$ is locally free. 
An element $\beta u^i$ of $\omega^j_{X^{\times}/K}u^i$ is in the kernel of $\delta_j$ if and only if $\beta \wedge iu^{i-1}du=0.$ 
If we work locally and choose $\mathrm{dlog}x_{1},\dots, \mathrm{dlog}x_{n-1}, du $ as a basis for $\omega^1_{X^{\times}/\Sp(K)},$ $\beta$ can be written uniquely as an $\mathcal{O}_X$-linear combination of the $j$-th exterior powers of the elements of the basis of $\omega^1_{X^{\times}/\Sp(K)}$ we chose. Hence if $\beta \wedge iu^{i-1}du =0,$ then $\beta$ is an $\mathcal{O}_X$-linear combination of $j$-th exterior powers which contain $du.$ But this happens if and only if $\beta u^i$ is in the image of $\delta_{j-1}$, \emph{i.e} the columns are exact.

Hence, to calculate the cohomology of the total complex of (\ref{bicomplexforE}), we only need to calculate the cohomology of the complex given by the cokernels of the last vertical maps. It is given by
$$E\longrightarrow E\otimes \omega^1_{X^{\times}/\Sp(K)^{\times}}\longrightarrow E\otimes \omega^2_{X^{\times}/\Sp(K)^{\times}} \longrightarrow\cdots, $$
\emph{i.e.} the complex $\mathrm{DR}(p_u((E, \nabla_E)[u]))$. Hence the proof is finished. 
\end{proof}

We are ready to compare $\pi_1(X^{\times}/\Sp(K)^{\times}, x)^{\mathrm{tri}}$ and $\pi_1(X^{\times}/\Sp(K)[u], x)^{\mathrm{tri}}$. First we treat the case where $K$ is algebraically closed. 

\begin{prop}\label{pi(p_u)iso}
Suppose that $K$ is algebraically closed. Then the map $\pi(p_u)^{\mathrm{tri}}$ is an isomorphism:
$$\pi(p_u)^{\mathrm{tri}}:\pi_1(X^{\times}/\Sp(K)^{\times}, x)^{\mathrm{tri}} \xrightarrow{\cong} \pi_1(X^{\times}/\Sp(K)[u], x)^{\mathrm{tri}}. $$

\end{prop}
\begin{proof}

We will show that, for every $(E, \nabla_E)\in MIC(X^{\times}/\Sp(K)^{\times})^{\rm nr,tri}$, there exists an
object $(F, \nabla_F)$ in $MIC(X^{\times}/\Sp(K)[u])^{\rm nr,tri}$ with $p_u(F, \nabla_F) = (E,\nabla_E)$
(where ${}^{\rm tri}$ is as in remark \ref{cattri}), and that the morphisms 
of de Rham cohomologies 
\begin{equation}\label{drpu}
H^i_{\rm dR}(X^{\times}/\Sp(K),  (F, \nabla_{F})) 
\rightarrow  H^i_{\mathrm{dR}}(X^{\times}/\Sp(K)^{\times}, (E, \nabla_{E})) \,\,\,\, (i \in \mathbb{N}) 
\end{equation}
induced by $p_u$ are isomorphisms. 

We prove the above claim by induction on the rank of $(E, \nabla_E)$. When 
$(E, \nabla_E)$ is of rank $1$, thanks to theorem \ref{lifting}, we know that there exists $(E, \widetilde{\nabla}_E)\in MIC(X^{\times}/\Sp(K))^{\mathrm{nr}}$ 
such that $p_u((E, \widetilde{\nabla}_E)[u])=(E, \nabla_E)$. Moreover, 
the morphisms 
$$
H^i_{\rm dR}(X^{\times}/\Sp(K),  (E, \widetilde{\nabla}_{E})[u]) 
\rightarrow  H^i_{\mathrm{dR}}(X^{\times}/\Sp(K)^{\times}, (E, \nabla_{E})) \,\,\,\, (i \in \mathbb{N}) 
$$ 
are isomorphisms by proposition \ref{quasiisoHirsch}. So we are done in rank $1$ case. 

In general case, we have an exact sequence 
\begin{equation}\label{e-diag}
0\rightarrow (E', \nabla_{E'})\rightarrow (E, \nabla_E)\rightarrow (E'', \nabla_{E''})\rightarrow 0 
\end{equation}
in $MIC(X^{\times}/\Sp(K)^{\times})^{\rm nr,tri}$ with $E''$ of rank $1$. Also, there exist 
$(F',\nabla_{F'}), (F'',\nabla_{F''})$ in $MIC(X^{\times}/\Sp(K)[u])^{\rm nr,tri}$ with 
$p_u(F', \nabla_{F'}) = (E',\nabla_{E'}), p_u(F'', \nabla_{F''}) = (E'',\nabla_{E''})$, by induction hypothesis. 
The above exact sequence defines an element in 
$$\mathrm{Ext}^1( (E'',\nabla_{E''}), (E', \nabla_{E'}))\cong H^1_{\mathrm{dR}}(X^{\times}/\Sp(K)^{\times}, 
(E'',\nabla_{E''})^{\vee}\otimes (E', \nabla_{E'})). $$ On the other hand, the group of extension classes 
$$ \mathrm{Ext}^1_{MIC(X^{\times}/\Sp(K)[u])^{\mathrm{nr}}}((F'',\nabla_{F''}), (F', \nabla_{F'}))$$ 
is isomorphic to $H^1_{\rm dR}(X^{\times}/\Sp(K),  (F'',\nabla_{F''})^{\vee} \otimes (F', \nabla_{F'}))$ by theorem \ref{extensionsoflattices}. Because the map 
$$ 
H^1_{\rm dR}(X^{\times}/\Sp(K),  (F'',\nabla_{F''})^{\vee} \otimes (F', \nabla_{F'})) 
\rightarrow  H^1_{\mathrm{dR}}(X^{\times}/\Sp(K)^{\times}, 
(E'',\nabla_{E''})^{\vee}\otimes (E', \nabla_{E'}))
$$ 
induced by $p_u$ is an isomorphism by induction hypothesis, we obtain an exact sequence 
$$ 0\rightarrow (F', \nabla_{F'})\rightarrow (F, \nabla_F)\rightarrow (F'', \nabla_{F''})\rightarrow 0 $$
in $MIC(X^{\times}/\Sp(K)[u])^{\rm nr,tri}$ which gives rise to the exact sequence \eqref{e-diag} when we apply $p_u$ to it. 
Moreover, we see that the morphisms \eqref{drpu} 
induced by $p_u$ are isomorphisms by induction hypothesis and five lemma. So we are done. 
\end{proof}

Next we compare $\pi_1(X^{\times}/\Sp(K)^{\times}, x)^{\mathrm{tri}}$ and $\pi_1(X^{\times}/\Sp(K)[u], x)^{\mathrm{tri}}$ in the general case by reducing to the 
previous case. 
In the following, let $\overline{K}$ be an algebraic closure of $K$. 
Also, for a field $L$ containing $K$, 
let $X_{L}$ (resp. $X_{L}^{\times}$) 
be $X\otimes_K L$ (resp.  $X^{\times} \otimes_K L$), 
let $\Sp(L)^{\times}$ be $\Sp(K)^{\times} \otimes_K L$ and 
let $\alpha_L: X_{L}^{\times} \rightarrow X^{\times}$ be the 
natural projection. 

\begin{cor}\label{pi(p_u)isocor}
The map $\pi(p_u)^{\mathrm{tri}}$ is an isomorphism:
$$\pi(p_u)^{\mathrm{tri}}:\pi_1(X^{\times}/\Sp(K)^{\times}, x)^{\mathrm{tri}}\xrightarrow{\cong} \pi_1(X^{\times}/\Sp(K)[u], x)^{\mathrm{tri}}. $$
\end{cor}

\begin{proof}
We prove that $\pi(p_u)^{\mathrm{tri}}$ is a closed immersion. To do so, it suffices 
to prove the essential surjectivity of $p_u: MIC(X^{\times}/\Sp(K)[u])^{\rm nr,tri} \to 
MIC(X^{\times}/\Sp(K)^{\times})^{\rm nr,tri}$. 
Let $(E, \nabla_E)\in MIC(X^{\times}/\Sp(K)^{\times})^{\rm nr,tri}$ and let 
$(V,\rho:\pi_1(X^{\times}/\Sp(K),x) \rightarrow GL(V))$ be the corresponding representation. 
Then the base change $\rho_{\overline{K}}: \pi_1(X^{\times}/\Sp(K),x) \otimes_K \overline{K} \rightarrow GL(V\otimes_K \overline{K})$ 
of $\rho$ to $\overline{K}$ is an iterated extension of rank $1$ representations. 
The morphism 
$\alpha_{\overline{K}}: X_{\overline{K}} \rightarrow X$ induces the 
morphism $\pi(\alpha_{\overline{K}}): 
\pi_1(X^{\times}_{\overline{K}}/\Sp(\overline{K}),x) \rightarrow 
\pi_1(X^{\times}/\Sp(K),x) \otimes_K \overline{K}$ and the pullback 
$\alpha_{\overline{K}}^*(E,\nabla_E) \in MIC(X^{\times}_{\overline{K}}/\Sp(\overline{K})^{\times})^{\rm nr}$ 
of $(E,\nabla_E)$ by $\alpha_{\overline{K}}$ corresponds to the representation 
$\rho_{\overline{K}} \circ \pi(\alpha_{\overline{K}})$. Hence it is an iterated extension of rank $1$ representations and so 
$\alpha_{\overline{K}}^*(E,\nabla_E)$ is an object in  $MIC(X^{\times}_{\overline{K}}/\Sp(\overline{K})^{\times})^{\rm nr,tri}$. 
Hence, by proposition \ref{pi(p_u)iso}, there exists an object 
$(F_{\overline{K}}, \nabla_{F_{\overline{K}}}) \in MIC(X_{\overline{K}}^{\times}/\Sp(\overline{K})[u])^{\rm nr}$ such that $p_u(F_{\overline{K}}, \nabla_{F_{\overline{K}}}) \cong  \alpha_{\overline{K}}^*(E,\nabla_E)$. 
By standard argument, we see that there exists a finite Galois subextension $L$ of $K$ in $\overline{K}$ 
such that $(F_{\overline{K}}, \nabla_{F_{\overline{K}}})$ is the pullback of some 
object $(F_L,\nabla_{F_L})$ in $MIC(X_{L}^{\times}/\Sp(L)[u])^{\rm nr}$ and that 
$p_u(F_L,\nabla_{F_L})$ 
is isomorphic to $\alpha_{L}^*(E,\nabla_E)$. By full faithfulness of 
$p_u: MIC(X^{\times}/\Sp(K)[u])^{\rm nr} \to 
MIC(X^{\times}/\Sp(K)^{\times})^{\rm nr}$, 
which follows from 
proposition \ref{pipusurjective} and 
Galois descent, 
 $(F_L,\nabla_{F_L})$ descends to an object 
$(F,\nabla_F)$ in $MIC(X^{\times}/\Sp(K)[u])^{\rm nr}$ and we have an isomorphism 
$p_u(F,\nabla_F) \cong (E,\nabla_E)$. 

Thus it suffices to prove that $(F,\nabla_F)$ belongs to 
 $MIC(X^{\times}/\Sp(K)[u])^{\rm nr,tri}$. Let $$(W,\tau: \pi_1(X^{\times}/\Sp(K)[u],x) \rightarrow GL(W))$$ 
be the representation corresponding to $(F,\nabla_F)$. Then the composition 
$$ 
\pi_1(X^{\times}/\Sp(K)^{\times},x) \otimes_K \overline{K} 
\overset{\pi(p_u) \otimes {\rm id}}{\rightarrow} 
\pi_1(X^{\times}/\Sp(K)[u],x) \otimes_K \overline{K} 
\overset{\tau_{\overline{K}}}{\rightarrow} 
GL(W \otimes_K \overline{K}) 
$$ 
(where $\tau_{\overline{K}}$ is the base change of $\tau$ to $\overline{K}$)
is isomorphic to 
$\rho_{\overline{K}}$ 
by construction. Since the latter representation is an iterated extension of rank $1$ representations and the map $\pi(p_u) \otimes {\rm id}$ is faithfully flat, we see that 
$\tau_{\overline{K}}$ is also an iterated extension of rank $1$ representations. 
Hence $(F,\nabla_F)$ belongs to 
 $MIC(X^{\times}/\Sp(K)[u])^{\rm nr,tri}$ and the proof is finished. 
\end{proof}

As a lemma for the proof of the main theorem, 
we prove a base change property of our algebraic fundamental groups. 

\begin{lemma}\label{basechangeprop-corrected}
Let $L$ be a finite extension of $K$. Then 
we have an isomorphism 
\begin{equation}\label{basechange}
\pi_{1}(X_L^{\times}/\Sp(L), x) \cong 
\pi_{1}(X^{\times}/\Sp(K), x)\otimes_K L. 
\end{equation}
\end{lemma}

\begin{proof}
It is easy to check that the push-forward by $\alpha_L$ induces an equivalence 
\begin{equation*}\alpha_{L,*}:MIC(X_{L}^{\times}/\Sp(L))^{\rm nr} \overset{\cong}{\rightarrow} 
\{(E, \nabla_E)\in MIC(X^{\times}/\Sp(K))^{\rm nr} \textrm{ with a $K$-linear structure of $L$-module}\},  \end{equation*}
by the fact that any object $(E,\nabla)$ in $MIC(X_{L}^{\times}/\Sp(L))^{\rm nr}$ is a direct summand of 
$\alpha_{L}^*\alpha_{L,*}(E,\nabla)$. The claim follows from this equivalence (see \cite[\S 4]{Del89}). 
\end{proof}

Now we have all the ingredients to prove the main theorem.
\begin{teo} 
The following sequence is exact:
\begin{equation}\label{exseqsol}
1\longrightarrow \pi_{1}(X^{\times}/\Sp(K)^{\times}, x)^{\mathrm{tri}}\longrightarrow \pi_{1}(X^{\times}/\Sp(K), x)^{\mathrm{tri}}\longrightarrow \pi_{1}(\Sp(K)^{\times}/\Sp(K), \nu)^{\mathrm{tri}}\longrightarrow 1. 
\end{equation}
\end{teo}
\begin{proof} 
First note that, for an affine group scheme $G$ over $K$, 
$N := \mathrm{Ker}(G \rightarrow G^{\rm tri})$ is stable under any automorphism 
$\sigma$ on $G$. Indeed, for any surjection $\rho: G \rightarrow H$ 
with $H$ geometrically trigonalizable, 
$\rho \circ \sigma$ has the same property 
and so $\rho(\sigma(N))$ is trivial. Hence $\sigma(N)$ is contained in $N$. 
Then, by applying this argument also to $\sigma^{-1}$, we see that $\sigma(N) = N$. 

Now consider the exact sequence 
\begin{equation*}
1\longrightarrow\pi_1(X^{\times}/\Sp(K)[u], x)\xrightarrow{\pi(t_u)} \pi_1(X^{\times}/\Sp(K), x)\xrightarrow{\pi(f^*_{\mathrm{dR}})} \pi_1(\Sp(K)^{\times}/\Sp(K), \nu)\longrightarrow 1
\end{equation*}
proven in proposition \ref{exactu}. By the remark in the previous paragraph, 
$N := \mathrm{Ker}(\pi_1(X^{\times}/\Sp(K)[u], x) \rightarrow \pi_1(X^{\times}/\Sp(K)[u], x)^{\rm tri})$ 
is a normal subgroup of $\pi_1(X^{\times}/\Sp(K), x)$. 
If we put $\pi_1(X^{\times}/\Sp(K), x)' := \pi_1(X^{\times}/\Sp(K), x)/N$, 
we obtain the exact sequence 
\begin{align*}
1\longrightarrow \pi_1(X^{\times}/\Sp(K)[u], x)^{\rm tri} 
\cong \pi_1(X^{\times}/\Sp(K)^{\times}, x)^{\rm tri}  
& \longrightarrow \pi_1(X^{\times}/\Sp(K), x)' \\ & \xrightarrow{\pi(f^*_{\mathrm{dR}})'} 
\pi_1(\Sp(K)^{\times}/\Sp(K), \nu)\longrightarrow 1 
\end{align*}
and there exists a natural surjection $\pi: \pi_1(X^{\times}/\Sp(K), x)' \rightarrow 
\pi_1(X^{\times}/\Sp(K), x)^{\rm tri}$. Hence it suffices to prove that 
$\pi$ is an isomorphism, namely, $\pi_1(X^{\times}/\Sp(K), x)'$ is geometrically protrigonalizable. 
This is reduced to proving the same property for $\pi_1(X^{\times}/\Sp(K), x)' \otimes_K \widetilde{L}$ for 
some field extension $\widetilde{L}$ over $K$, by proposition \ref{gtri}. 
For some finite extension $L$ of $K$, $f_L: X^{\times}_L \rightarrow \Sp(L)^{\times}$ admits a section 
$y$ and so the map $\pi(f_{L, {\rm dR}}^*): \pi_1(X^{\times}_L/\Sp(L), y) \rightarrow \pi_1(\Sp(L)^{\times}/\Sp(L), \nu)$ 
induced by $f_L$ admits a section. On the other hand, we have a map 
\begin{align*}
\pi_1(X^{\times}_L/\Sp(L), x) \cong & 
\pi_1(X^{\times}/\Sp(K), x) \otimes_K L  \\ 
\xrightarrow{\pi(f^*_{\mathrm{dR}}) \otimes {\rm id}} & 
\pi_1(\Sp(K)^{\times}/\Sp(K), \nu) \otimes_K L \cong \pi_1(\Sp(L)^{\times}/\Sp(L), \nu). 
\end{align*}
We would like to identify these two maps after some field extension of $L$. 
Take a tensor isomorphism $\gamma$ between fiber functors 
\begin{align*}
& MIC(X_L^{\times}/\Sp(L))^{\rm nr} \xrightarrow{\omega_x} {\rm Vec}_L 
\to {\rm Vec}_{\widetilde{L}}, \\ 
& MIC(X_L^{\times}/\Sp(L))^{\rm nr} \xrightarrow{\omega_y} {\rm Vec}_L 
\to {\rm Vec}_{\widetilde{L}} 
\end{align*}
associated to $x$ and $y$, which exists for some field extension $\widetilde{L}$ of $L$. 
The tensor isomorphism $\gamma$ induces an isomorphism 
$$i_{\gamma}: \pi_1(X^{\times}_L/\Sp(L), x) \otimes_L \widetilde{L} \xrightarrow{\cong}  
\pi_1(X^{\times}_L/\Sp(L), y) \otimes_L \widetilde{L} $$ 
of `conjugation by $\gamma$'. 
The isomorphism 
$i_{\gamma}$ induces an inner automorphism 
$$\overline{i}_{\gamma}: \pi_1(\Sp(L)^{\times}_L/\Sp(L), \nu) \otimes_L \widetilde{L} \xrightarrow{\cong} \pi_1(\Sp(L)^{\times}_L/\Sp(L), \nu) \otimes_L \widetilde{L}, $$ 
but this is necessarily the identity because 
$\pi_1(\Sp(L)^{\times}_L/\Sp(L), \nu) \otimes_L \widetilde{L} \cong \mathbb{G}_a$ 
is abelian. Hence the isomorphism $i_{\gamma}$ is compatible with maps 
$$ \pi(f_{L, {\rm dR}}^*) \otimes {\rm id}: \pi_1(X^{\times}_L/\Sp(L), y) \otimes_L \widetilde{L}  \rightarrow \pi_1(\Sp(L)^{\times}/\Sp(L), \nu) \otimes_L \widetilde{L}, $$
\begin{align*}
\pi(f^*_{\mathrm{dR}}) \otimes {\rm id}: & 
\pi_1(X^{\times}_L/\Sp(L), x) \otimes_{L} \widetilde{L} =  
\pi_1(X^{\times}/\Sp(K), x) \otimes_K \widetilde{L} \\ & \rightarrow  
\pi_1(\Sp(K)^{\times}/\Sp(K), \nu) \otimes_K \widetilde{L} = 
\pi_1(\Sp(L)^{\times}/\Sp(L), \nu) \otimes_L \widetilde{L}, 
\end{align*}
as required. 
%
%and it is identified (noncanonically) with $\pi(f_{L,{\rm dR}}^*)$ after a scalar extension to %some 
%field extension $\widetilde{L}$ of $L$. 
Hence 
$$\pi(f^*_{\mathrm{dR}}) \otimes {\rm id}: \pi_1(X^{\times}/\Sp(K), x) \otimes_K \widetilde{L}  
\rightarrow
\pi_1(\Sp(K)^{\times}/\Sp(K), \nu) \otimes_K \widetilde{L} 
$$
admits a section, and so does the map  
$$\pi(f^*_{\mathrm{dR}})' \otimes {\rm id}: \pi_1(X^{\times}/\Sp(K), x)' \otimes_K \widetilde{L}  
\rightarrow
\pi_1(\Sp(K)^{\times}/\Sp(K), \nu) \otimes_K \widetilde{L}. 
$$
Therefore, the proof of the theorem is reduced to the following claim: 
For a split exact sequence 
$$ 1 \rightarrow H \rightarrow G \rightarrow \mathbb{G}_a \rightarrow 1 $$
of algebraic groups over a field $K$ of characteristic zero with  
$H$ geometrically trigonalizable, $G$ is also geometrically trigonalizable. 

We prove the above claim. By \cite[proposition 17.5]{Mil15}, $H$ has a unique (maximal) unipotent normal subgroup $U$ such that $M := H/U$ is of multiplicative type. By uniqueness, $U$ is normal in $G$. 
Let $P$ be $G/U$. Then we have a split exact sequence 
$$ 1 \rightarrow M \rightarrow P \rightarrow \mathbb{G}_a \rightarrow 1. $$
Since the action of $\mathbb{G}_a$ on $M$ induced by the splitting is trivial 
by \cite[theorem 14.29]{Mil15}, $P \cong M \times \mathbb{G}_a$. If we denote 
the kernel of $G \rightarrow P \cong M \times \mathbb{G}_a \rightarrow M$ by $U'$, we have 
exact sequences 
$$ 1 \rightarrow U' \rightarrow G \rightarrow M \rightarrow 1, \qquad 
 1 \rightarrow U \rightarrow U' \rightarrow \mathbb{G}_a \rightarrow 1. $$ 
From the second exact sequence above, $U'$ is unipotent. 
Then, by the first exact sequence above and \cite[proposition 17.2]{Mil15}, 
$G$ is geometrically trigonalizable. So the proof is finished. 
\end{proof}

\bibliographystyle{amsalpha}

\end{document}